\documentclass{amsart}
\numberwithin{equation}{section}

\usepackage{graphicx} 

\usepackage{a4wide}
\usepackage{color}
\usepackage{hyperref}
\usepackage{comment}
\usepackage[english]{babel} 
\usepackage{amsmath}
\usepackage{amsfonts}
\usepackage{amssymb}
\usepackage{amsthm}
\usepackage{comment}
\usepackage{graphics}
\usepackage{enumitem}
\usepackage{pgf,tikz}
\usetikzlibrary{patterns,calc}
\usepackage{cases}
\usepackage{subfigure}  
\usepackage{xfrac}

\renewcommand{\arraystretch}{1.3}

\usepackage{mathtools}

 \usepackage[utf8]{inputenc}   
 \usepackage[T1]{fontenc}

\newcommand{\R}{\mathbb{R}}	
\newcommand{\N}{\mathbb{N}} 
\newcommand{\e}{\varepsilon}

\def\Xint#1{\mathchoice
  {\XXint\displaystyle\textstyle{#1}}%
  {\XXint\textstyle\scriptstyle{#1}}%
  {\XXint\scriptstyle\scriptscriptstyle{#1}}%
  {\XXint\scriptscriptstyle\scriptscriptstyle{#1}}%
  \!\int}
\def\XXint#1#2#3{{\setbox0=\hbox{$#1{#2#3}{\int}$}
    \vcenter{\hbox{$#2#3$}}\kern-.5\wd0}}

\def\dashint{\Xint-}

\theoremstyle{plain}
\newtheorem{theorem}{Theorem}[section]
\newtheorem{proposition}[theorem]{Proposition}
\newtheorem{lemma}[theorem]{Lemma}
\newtheorem{corollary}[theorem]{Corollary}
\newtheorem{remark}[theorem]{Remark}
\newtheorem{definition}[theorem]{Definition}

\author{Alessandra De Luca}
\address{Alessandra De Luca\newline\indent
Dipartimento di Matematica e Applicazioni
\newline\indent
Università degli Studi di Milano-Bicocca
\newline\indent
Via Cozzi 55, 20125, Milano, Italy}
\email{alessandra.deluca@unimib.it}

\author{Veronica Felli}
\address{Veronica Felli\newline\indent
Dipartimento di Matematica e Applicazioni
\newline\indent
Università degli Studi di Milano-Bicocca
\newline\indent
Via Cozzi 55, 20125, Milano, Italy}
\email{veronica.felli@unimib.it}

\author{Stefano Vita}
\address{Stefano Vita\newline\indent
Dipartimento di Matematica ``Felice Casorati''
\newline\indent
Universit\`a degli Studi di Pavia
\newline\indent
Via Ferrata 5, 27100, Pavia, Italy}
\email{stefano.vita@unipv.it}

\title[Edge asymptotics for weighted Robin equations] {Edge
  asymptotics for
  weighted Robin equations \\
  on cylinders with applications to \\
  spectral fractional Laplacians }

\date{\today}

\begin{document}

\begin{abstract}
  For a class of weighted degenerate or singular problems on a
  cylinder, Almgren-type monotonicity methods are employed to derive
  asymptotic estimates of solutions near the edge between the base and
  the lateral surface, where a Robin boundary condition is imposed. As
  a relevant application, local asymptotics and unique continuation
  from the boundary are obtained for the Robin and Neumann spectral
  fractional Laplacians.  The local decay rates of the solutions are
  the same for both the Neumann and the Robin problems and are
  explicitly quantized by a weighted spherical spectral problem with
  some symmetry.
\end{abstract}

\maketitle


\section{Introduction}

Let $s\in (0,1)$ and $n>2s$. Let $\Omega\subset\mathbb R^n$ be an open
domain of class $C^{1,1}$ and $T>0$. Let us consider the cylinder in $\R^{n+1}$
with cross-section $\Omega$ and height $T$ and its lateral boundary:
\begin{equation}\label{eq:cil}
  \mathcal{C}_{\Omega}:=\Omega\times (0,T),\quad
  \partial_L \mathcal{C}_{\Omega}:= \partial\Omega\times [0,T).
\end{equation}
We denote the variable in $\R^{n+1}$ by $z=(x,t)\in \R^n\times\R$ and
introduce the space $H^1(\mathcal{C}_{\Omega},t^{1-2s}dz)$, defined as
the completion of $C^\infty(\overline{\mathcal{C}_{\Omega}})$ with
respect to the weighted norm
\begin{equation}\label{eq:normaH1weighted}
  \Vert \varphi\Vert_{H^1(\mathcal{C}_{\Omega},t^{1-2s}dz)}:=
  \left(\int_{\mathcal{C}_{\Omega}}t^{1-2s}(\varphi^2+|\nabla\varphi|^2)\,dz\right)^{1/2}. 
\end{equation}
Let $\alpha\geq 0$. We
focus on weak solutions $v\in H^1(\mathcal{C}_{\Omega},t^{1-2s}dz)$ to
\begin{equation}\label{eqrisoltadaV}
\begin{cases}
-\mathop{\rm div}(t^{1-2s} \nabla v)=0
&\text{in } \mathcal{C}_{\Omega}, \\
\nabla _x v\cdot\nu+\alpha v=0  & \text{on } \partial_L \mathcal{C}_{\Omega}, \\
-\lim_{t \to 0^+} t^{1-2s}\partial _t v = h(x) \mathop{\rm Tr}v & \text{on } \Omega,
\end{cases}
\end{equation} 
where $\nu$ is the outer normal vector to $\partial_L\mathcal{C}_{\Omega}$ and
$\mathop{\rm Tr}$ denotes the linear and continuous trace operator
\begin{equation}\label{passoalletracce}
    \mathop{\rm Tr}: H^1(\mathcal{C}_{\Omega},t^{1-2s}dz)\to H^s(\Omega), 
\end{equation}
being $H^s(\Omega)$ the usual fractional Sobolev space (see \cite{JLX}).
We also assume that 
\begin{equation}\label{ipotesisuh}
    h\in W^{1,p}(\Omega) \quad \text{for some $p> \frac{n}{2s}$}.
\end{equation}
By a weak solution to \eqref{eqrisoltadaV} we mean a function
  $v\in H^1(\mathcal{C}_{\Omega},t^{1-2s}dz)$ such that
  \begin{equation}\label{eq:wf}
    \int_{\mathcal C_\Omega}t^{1-2s} \nabla v\cdot\nabla\varphi \,dz+
    \alpha\int_{ \partial_L \mathcal{C}_{\Omega}}t^{1-2s}
    v\varphi\,dS=\int_\Omega h \mathop{\rm Tr}v \mathop{\rm
      Tr}\varphi\,dx
      \end{equation}
            for every $\varphi \in C^\infty_{\rm
              c}(\overline{\Omega}\times [0,T))$,
      where $dS$ denotes the volume element on the
        $n$-dimensional lateral surface
        $\partial\Omega\times (0,+\infty)$. We observe that there
        exists a natural linear and continuous trace operator
        \begin{equation*}
          \gamma_L: H^1(\mathcal{C}_{\Omega},t^{1-2s}dz)\to
          L^2(\partial_L\mathcal{C}_{\Omega},t^{1-2s}dS);
        \end{equation*}
        therefore, the second integral in \eqref{eq:wf} is
        well-defined for all
        $v,\varphi \in H^1(\mathcal{C}_{\Omega},t^{1-2s}dz)$. For
        notational simplicity, in \eqref{eq:wf} we are writing
        $v,\varphi$ instead of $\gamma_L(v),\gamma_L(\varphi)$.

In this context, we are interested in proving the validity of the
strong unique continuation principle for solutions to
\eqref{eqrisoltadaV} at any edge point
\begin{equation}\label{eq:edgepoint}
z_0=(x_0,0),\quad \text{with }x_0\in\partial\Omega.
\end{equation}
The study of unique continuation properties at boundary points has attracted
considerable attention in recent years. This is largely because it
proves more challenging than the corresponding analysis at interior
points, due to both the effect of the boundary geometry and a possible
lack of regularity of solutions.  In this regard, we mention
\cite{AdoEsc97, AdoEscKen95, Del26, DelFel21, FelFer13, KN, Tolsa}.
In our setting, the matter is particularly subtle because $z_0$ is
located on a boundary edge of the cylindrical domain, specifically on
the corner where the base and the lateral surface meet.

Our approach relies on the Almgren monotonicity argument, pioneered by
Garofalo and Lin in \cite{GarLin86} to derive doubling estimates and,
consequently, the strong unique continuation property for perturbed
second-order elliptic equations, with possibly very singular
potentials.  Adapting these techniques to prove unique continuation
from the boundary poses several technical challenges, also depending
on the type of boundary conditions being considered; in particular, in
the case of the Neumann problem, the ability to perform domain
transformations that ensure the vanishing of the conormal boundary
derivative is crucial for deriving monotonicity formulas. For this
reason, in the study of Robin problems carried out in
\cite{Li2025,Li-Wang,Zhu15}, it proved useful to construct an
auxiliary function that simplifies Robin problems into homogeneous
conormal ones. In the present work, we draw inspiration from this
strategy, reducing our problem to one with homogeneous conormal
boundary conditions, thus enabling an even reflection through the
lateral surface of the cylinder. To this aim, in Section \ref{sez2} we
introduce a suitable change of variables, involving the distance
function from $\partial\Omega$. This transformation maps problem
\eqref{eqrisoltadaV}, which features a Robin boundary condition on the
lateral surface of the cylinder, into a boundary value problem with a
homogeneous Neumann condition on the same surface (see problem
\eqref{eqrisoltadaW}), at the cost of generating additional terms
within the equation.

The monotonicity formula we aim to develop investigates the behavior
of the frequency (local energy divided by the mass) of solutions near
the point $z_0=(x_0,0)$ lying on the boundary of the base
$\Omega\times\{0\}$ of the cylinder; therefore, the geometry of
$\partial\Omega$ may significantly interfere with the argument. To
overcome such a difficulty, before carrying out the aforementioned
even reflection, we apply a local diffeomorphism that straightens
$\partial\Omega$ in a neighborhood of $x_0$. As a result, we are led
to deal with a differential operator involving a non-identity
coefficient matrix, whose regularity is closely tied to the regularity
of the domain $\Omega$, see problem \eqref{problemaraddrizzato}.

After an even reflection through the straightened lateral surface of
the cylinder, the derivation of a Pohozaev-type identity becomes quite
straightforward, since the resulting problem falls within the
assumptions considered in \cite{FelSic22}.  Subsequently, we prove
that the Almgren frequency function $\mathcal{N}(r)$ defined in
\eqref{N} is bounded and has a finite limit as $r\to 0^+$.
Typically, this would be enough to deduce the strong unique
continuation principle via suitable doubling conditions; however, we
go further and in Section \ref{seclocal} we provide a classification
of all possible vanishing orders of nontrivial solutions, by means of
a blow-up analysis. We emphasize that the limit profiles turn out to
be nontrivial homogeneous polynomials, which can be described in terms
of the eigenfunctions of a related spherical eigenvalue problem, whose
properties are discussed in Section \ref{sect3}.

The following theorem is the first main result of the present paper.

\begin{theorem}\label{thm:1} Let
    $s\in (0,1)$, $n>2s$, and $\Omega\subset\mathbb R^n$ be an open
    domain of class $C^{1,1}$.  For some $\alpha\geq0$ and $h$
    satisfying \eqref{ipotesisuh}, let
  $v\in H^1(\mathcal{C}_\Omega,t^{1-2s}dz)$ be a nontrivial weak
  solution to \eqref{eqrisoltadaV} in the sense clarified in
  \eqref{eq:wf}.  For every $z_0=(x_0,0)\in \partial\Omega\times\{0\}$,
  there exist $m_0\in\mathbb N$ (which is necessarily even if
  $n=1$) and a nontrivial homogeneous polynomial $P$ of degree $m_0$
  such that $\mathop{\rm div}(|t|^{1-2s}\nabla P)=0$ in $\R^{n+1}$,
    \begin{align*}
   &P(x,t)\text{ is invariant under reflections across both the hyperplanes}\\
         &\{(x,t)\in\R^{n+1}:x\cdot {\mathbf N}_0=0\}\quad\text{and}\quad
           \{(x,t)\in\R^{n+1}:t=0\}
  \end{align*}
  and, for any $\delta>0$,
    \begin{equation*}
      \frac{v(z_0+\lambda z)}{\lambda^{m_0}}\to P(z)\quad
      \text{as $\lambda\to 0^+ $}
    \end{equation*}
    in $H^1((C_\delta\times(0,T))\cap
      B_1,t^{1-2s}dz)$ and in $C^{0,\tau}_{\mathrm{loc}}((\overline{
      C_\delta}\setminus\{0\})\times[0,\infty))$
      for some
      $\tau\in(0,1)$, where
$B_1:=\{z\in \R^{n+1}: |z|<1\}$,     
      \begin{equation}\label{eq:def_cono}
        C_\delta:=\big\{x\in\R^n  : x\cdot {\mathbf N}_0
        <-\delta|x-(x\cdot {\mathbf N}_0) {\mathbf N}_0|\big\}
    \end{equation}
    and ${\mathbf N}_0=\nu(x_0)$ is the outer normal vector to
    $\partial\Omega$ at $x_0$.
\end{theorem}
We observe that the sets
  $x_0+C_\delta$ are circular semi-cones with axes passing through
  $x_0$, orthogonal to
  $\partial\Omega$ and pointing toward the interior of
  $\Omega$, see Figure \ref{f:cone}. In particular, the result established in Theorem
  \ref{thm:1} describes the asymptotic behavior of solutions along
  all directions transversal to the boundary.
\begin{figure}[ht]
\centering
 \begin{tikzpicture}[scale=0.35]
   \filldraw[fill=gray!40]  plot [smooth cycle]
 coordinates {(-2,-2) (-1,-2.5) (3,-3) (5,-3) (6,-1) (7,0) (0,2) (-2,1)
   (-2.5,0.5)};
  \draw[shift={(0,5)},densely dotted]  plot [smooth cycle]
 coordinates {(-2,-2) (-1,-2.5) (3,-3) (5,-3) (6,-1) (7,0) (0,2) (-2,1)
   (-2.5,0.5)};
 \draw[densely dotted] (4.5,-3.1) -- (4.5,1.9);
 \draw[densely dotted] (-2.5,0.7) -- (-2.5,5.7); 
 \draw[densely dotted] (7.25,-0.25) -- (7.25,4.7); 
 \draw[line width=0.4pt,color=black] (-2,-2) -- (4,-3) -- (4,2) -- (-2,3) -- cycle;
\draw[line width=0.4pt,color=black] (-2,-2) -- (2,0) --
 (2,5) -- (-2,3) -- cycle;
 \draw[line width=0.3pt,color=black] (-2,-2) -- (2,0) -- (2,5) -- (-2,3) -- cycle;
\draw[line width=0.3pt,color=black] (-2,-2) -- (4,-3) -- (4,2) --
(-2,3) -- cycle;
  \shade[ball color=magenta!15] 
    (-2,0) to[bend left=20] (0, -1)    to[bend left=20] (0, -2.3)
          to[bend left=20] (-2,0);      
  \draw (-2,0) to[bend left=20] (0,-1) 
              to[bend left=20] (0, -2.3);
  \filldraw[fill=magenta!45, draw=black] 
  (-2,-2) -- (0, -2.34) to[bend right=30] (-2,0) -- cycle;
  \node at (-2.7,-2) {\scriptsize $z_0$};
  \node at (6,-0.2) {\scriptsize $\Omega$};
  \filldraw[fill=black, draw=black] (-2,-2) circle (0.1);
\end{tikzpicture}
\caption{The set on which the convergence stated in Theorem
  \ref{thm:1} holds.}
 \label{f:cone}
\end{figure}
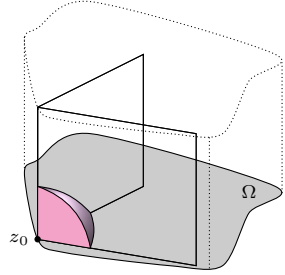

The most relevant application of Theorem \ref{thm:1} is an asymptotic
expansion at the origin for weak nontrivial solutions to the following
class of nonlocal equations
\begin{equation}\label{prob:nonloc}
  (-\Delta_{R,\alpha})^s u =\xi  u \quad \text{in $\Omega$},
\end{equation}
where the operator $(-\Delta_{R,\alpha})^s$ is the Robin spectral
fractional Laplacian of order $s\in(0,1)$ and coefficient $\alpha$
(see Section \ref{cinque} for a precise definition), and the potential
$\xi$ satisfies assumption \eqref{ipotesisuh}.
\begin{theorem}\label{thm:2}
  Let $s\in (0,1)$, $n>2s$, and $\Omega\subset\mathbb R^n$ be a
  bounded open domain of class $C^{1,1}$. For some $\alpha\geq0$ and
  $\xi\in W^{1,p}(\Omega)$, with $p> \frac{n}{2s}$, let $u$ be a
  nontrivial weak solution to \eqref{prob:nonloc} (in the sense
  specified in \eqref{form:deb:1} for $\alpha=0$ and in
  \eqref{form:deb:2} for $\alpha>0$).  For every
  $x_0\in\partial\Omega$, there exist $m_0\in\mathbb N$ (which is
  necessarily even if $n=1$) and a nontrivial homogeneous polynomial
  $\tilde P$ in $n$ variables of degree $m_0$ such that $\tilde P$ is
  invariant under reflection across the hyperplane
  $\{x\in\R^{n}:x\cdot {\mathbf N}_0=0\}$ and, for any $\delta>0$,
  \begin{equation}\label{convaPtilde} \frac{u(x_0+\lambda
      x)}{\lambda^{m_0}}\to \tilde P(x)\quad \text{as $\lambda\to 0^+$
      in $H^s(C_\delta\cap B'_1)$ and in $C^{0,\tau}_{\mathrm{loc}}(\overline{
      C_\delta}\setminus\{0\})$},
    \end{equation}
    for some $\tau\in(0,1)$, where
    $B'_1:= \{x\in \R^n : |x|<1\}$ and $C_\delta$ is defined in
    \eqref{eq:def_cono}.
\end{theorem}

An immediate consequence of the previous results is the
validity of the following strong unique continuation principles from
the boundary.

\begin{corollary}\label{cor:unique}
  Let $s\in (0,1)$, $n>2s$, and $\Omega\subset\mathbb R^n$ be an
    open domain of class $C^{1,1}$. The following statements hold true.
\begin{itemize}
\item[\rm(i)] For some $\alpha\geq0$ and $h$ satisfying
    \eqref{ipotesisuh}, let $v\in H^1(\mathcal{C}_\Omega,t^{1-2s}dz)$
  be a weak solution to \eqref{eqrisoltadaV} such that, for some
    $z_0\in\partial\Omega\times\{0\}$, $v(z)= O(|z-z_0|^k)$ as
    $|z-z_0|\to 0^+$ for every $k\in\mathbb{N}$. Then $v\equiv 0$ in
  $\mathcal{C}_\Omega$.
\item[\rm (ii)] For some $\alpha\geq0$ and $\xi$ satisfying
    \eqref{ipotesisuh}, let $u$ be a weak solution to \eqref{prob:nonloc} such
  that, for some $x_0\in\partial\Omega$, $u(x)= O(|x-x_0|^k)$ as
    $|x-x_0|\to 0^+$ for every $k\in\mathbb{N}$. Then $u\equiv 0$ in
    $\Omega$.
    \end{itemize}
\end{corollary}
The present paper -- particularly regarding the results related
  to problem \eqref{prob:nonloc} -- falls
  within the research line initiated in \cite{DelFelSic23,
    DelFelVit22, DelFelVit25}, where the authors investigate the
  problem of unique continuation from the boundary in the nonlocal
  setting, adapting frequency and monotonicity methods to scenarios
  where the nonlocality of fractional operators, combined with
  boundary interference, raises significant difficulties. 
  To be more
precise, \cite{DelFelSic23} deals with the Dirichlet spectral
fractional Laplacian, whereas \cite{DelFelVit22, DelFelVit25} treat
the  fractional Laplacian $(-\Delta)^s$ defined
via the Fourier transform $\mathcal{F}$ as
\begin{equation*}
  \mathcal{F}((-\Delta)^s u)(\omega)=
  |\omega|^{2s} \mathcal{F}(u)(\omega) \quad
  \text{for any $u\in \mathcal{S}(\mathbb{R}^n)$},
\end{equation*}
where $\mathcal{S}(\mathbb{R}^n)$ denotes the Schwartz space on
$\R^n$.  We also highlight that, consistently with \cite{DelFelSic23}
and unlike the results in \cite{DelFelVit22}, the homogeneity degrees
of limit profiles are integers as a direct consequence of the
regularity of eigenfunctions of the spherical eigenvalue problem
\eqref{prob-eigenvalues}: indeed, after an even reflection through
$\{\theta_{n+1}=0\}$, they are smooth by \cite[Theorem
1.1]{SirTerVit21a}, whereas the  eigenfunctions of the
corresponding spherical eigenvalue problem studied in
\cite{DelFelVit22}  exhibit lower regularity due to the
occurrence of mixed boundary conditions. 
As far as unique continuation principles from interior points
 for fractional problems are concerned, we refer to \cite{FalFel14,FalFel15,ruland,Yu17}.

We finally observe that our approach actually applies to the more
general class of equations \eqref{eq:general}, see Remark
\ref{rem:genarale}.

\subsection*{Notation} We adopt the convention that $0\in\mathbb
N$. For every $r>0$ we denote
\begin{itemize}
    \item $B'_r:= \{x\in \R^n : |x|<r\}$;
    \item $S_r:= \{x\in \R^n : |x|=r\}=\partial B_r'$;
    \item $B_r:=\{z\in \R^{n+1}: |z|<r\}$;
    \item $B^+_r:=\{(x,t)\in B_r: t>0\}$;
    \item $\partial ^+B^+_{r} := \{(x,t)\in\partial B_r : t>0\}$;
    \item $\mathbb{S}:= \partial B_1$,
      $\mathbb{S}^+:= \partial ^+B^+_1$, $\mathbb{S}':=
      S_1=\partial B_1'$.
\end{itemize}

\section{Pohozaev-type identity and the Almgren frequency function}\label{sez2}

\subsection{Reduction to homogeneous Neumann condition on the lateral boundary}
In the present subsection, a suitable change of variables transforms
\eqref{eqrisoltadaV} into an equivalent boundary value problem,
resulting in a homogeneous Neumann condition on the lateral surface of
the cylinder and incorporating additional terms in the equation.

By translation, it is not restrictive to assume that the point $x_0$
in \eqref{eq:edgepoint} is the origin.  Henceforth, we shall assume
that
\begin{equation*}
  x_0=0\in \partial\Omega;
\end{equation*}
we will then focus on studying the behavior of the solutions to
\eqref{eqrisoltadaV} near the point $z_0=(0,0)$, which therefore lies
on the edge of the cylinder.

We introduce the distance function
$d(x):= \mathrm{dist}(x,\partial\Omega)$, which is of class
$C^{1,1}(\Omega\cap B'_R)$ for some small $R>0$, see
  e.g. \cite{NT} (without loss of generality, up to a further
dilation, we can assume that $R=1<T$).  If
$v\in H^1(\mathcal{C}_{\Omega},t^{1-2s}dz)$ is a weak solution to
problem \eqref{eqrisoltadaV}, we consider the function 
\begin{equation}\label{VinfunxdiW}
  \tilde{v}(x,t):= e^{-\alpha d(x)}v(x,t)
  \quad \text{for all $(x,t)\in \mathcal{C}_{\Omega}\cap B_1$}.
\end{equation}
Therefore, $\tilde{v}\in H^1(\mathcal{C}_{\Omega}\cap B_1,t^{1-2s}dz)$
and, since $\nabla d(x)=-\nu(x)$ for every
  $x\in\partial\Omega\cap B_1$, $\tilde{v}$ weakly solves
\begin{equation}\label{eqrisoltadaW}
\begin{cases}
- \mathop{\rm div}(t^{1-2s} \nabla \tilde{v}) + t^{1-2s}\mathbf{b}(x)\cdot \nabla _x \tilde{v} + t^{1-2s}a(x) \tilde{v}=0 &\text{in } \mathcal{C}_{\Omega}\cap B_1, \\
\nabla _x \tilde{v}\cdot\nu=0  & \text{on } \partial_L \mathcal{C}_{\Omega}\cap B_1, \\
-\lim_{t \to 0^+} t^{1-2s}\partial _t \tilde{v} = h(x)\mathop{\rm Tr}\tilde{v}
& \text{on } \Omega\cap B_1',
\end{cases}
\end{equation}
where
\begin{equation*}
  \mathbf{b}(x) := -2\alpha  \nabla d(x) \quad\text{and}\quad
  a(x):=  -\alpha \Delta d-\alpha^2 |\nabla d|^2 \quad \text{for every
    $x\in \Omega\cap B'_1$}.
\end{equation*}
The weak validity of \eqref{eqrisoltadaW} is to be understood as
follows: for every
$\varphi\in C^\infty_{\rm c}(\overline{\mathcal{C}_\Omega}\cap
  B_1)$
there holds 
\begin{equation*}
  \int_{\mathcal{C}_\Omega\cap B_1}t^{1-2s}
  \big(\nabla \tilde{v}\cdot\nabla \varphi+(\mathbf{b}(x)\cdot\nabla
  _x \tilde v)
  \varphi + a(x)\tilde{ v}\varphi\big) \,dz =
  \int_{\Omega\cap B'_1} h(x)  \mathop{\rm Tr}\tilde{v}
  \, \mathop{\rm Tr}\varphi\,dx. 
\end{equation*}
From the regularity of the distance function $d$, we immediately
deduce that
\begin{equation}\label{regularitybanda}
  \mathbf{b}\in C^{0,1}(\Omega\cap B'_1,\mathbb{R}^{n})
  \quad \text{and}\quad  a\in L^\infty(\Omega\cap B'_1). 
\end{equation}

\subsection{A diffeomorphic deformation of the boundary and an even
  reflection}\label{diffeomorphic}
In this subsection, we first straighten
$\partial \Omega \cap B'_1$, using the local diffeomorphism employed
in \cite{DelFelVit22}, which is, in turn, inspired by
\cite{AdoEsc97}. Afterwards, denoting by
$x=(x',x_n)\in \mathbb{R}^{n-1}\times\R$ the variable in $\R^n$, we
apply an even reflection through the hyperplane $\{x_n=0\}$ in
order to extend the equation to the exterior of the cylinder in
a neighbourhood of the origin.

First, we observe that, due to the regularity of $\partial \Omega$, up
to rigid motions, without loss of generality we may assume that there
exists a function $g\in C^{1,1}(\mathbb{R}^{n-1},\mathbb{R})$ such
that
\begin{align}
\label{eq:rigid-mot}    \Omega\cap B'_1&= \{(x',x_n)\in B'_1: x_n < g(x')\},\\
  \label{eq:rigid-mot2}  \partial \Omega \cap B'_1&=\{(x',x_n)\in B'_1: x_n = g(x')\},\\ 
\end{align}
and 
\begin{equation}\label{condizionisug}
    g(0)=0 \quad \text{and} \quad \nabla g(0)= 0. 
\end{equation}
We outline the main properties of the aforementioned diffeomorphism in
the following proposition, referring the reader to \cite[Section
2]{DelFelVit22} for the proofs and further details.
\begin{proposition}\label{p:diff-deform-boun}
  There exist $F=(F_1,\dots,F_n,F_{n+1}):\mathbb{R}^{n+1}\to \mathbb{R}^{n+1}$ of class
  $C^{1,1}$ and $r_0>0$ such that
 \begin{enumerate}[label=\rm  (\roman*), itemsep=1.5ex]
        \item $0\in F(B_{r_0})\subset B_1$ and $F|_{B_{r_0}}: B_{r_0}\to F(B_{r_0})$ is a
          diffeomorphism of class $C^{1,1}$,
        \item for
          every $j=1,\dots,n$, $F_j(x',x_n,t)=\Psi_j(x',x_n)$ for some $\Psi_j\in C^{1,1}(\R^n)$,
        \item $F(x',0,0)= (x',g(x'),0)$ for every $x'\in\R^{n-1}$,\\[2pt]
          $F_n(x',x_n,t)=x_n+g(x')$ for every
          $(x',x_n,t) \in \R^{n+1}$,\\[2pt]
          $F_{n+1}(x,t)=t$ for every
          $(x,t) \in \R^n\times \R$;
 \item $F(\mathcal Q_{r_0})=\mathcal C_\Omega\cap F(B_{r_0}^+)$,
   where $\mathcal Q_{r_0}:=\{(x',x_n,t)\in B_{r_0}^+:x_n<0\}$,\\[2pt]
   $F(\partial_V \mathcal Q_{r_0})=\partial_L\mathcal{C}_{\Omega}
   \cap F(B_{r_0}^+)$, where $\partial_V \mathcal Q_{r_0}=\{(x',x_n,t)\in B_{r_0}^+:x_n=0\}$;
 \item $F$ and $F^{-1}$ admit the following expansions:
\begin{equation}\label{sviluppoF}
F(z)=z+O(|z|^2)\quad \text{as $|z|\to 0^+$}, \quad 
    F^{-1}(y)=y+O(|y|^2)\quad \text{as $|y|\to 0^+$};
  \end{equation}
\item the Jacobian matrix of $F$, denoted by $J_F$, depends only on
  the variable $x=(x',x_n)\in\R^n$, takes the form
  \begin{equation*}
    J_F(x',x_n)=
      \left(
    \renewcommand{\arraystretch}{1.5}
    \begin{array}{c|c}
 J_{\Psi} (x',x_n) & \mathbf{0} \\
  \hline
  \mathbf{0}^T & 1
    \end{array}
  \right),
\end{equation*}
where $\Psi=(\Psi_1,\dots,\Psi_n)$, and satisfies
   \begin{align}
     &\label{det}
 \det J_F(x',x_n)= 1+ O(|x'|^2) +
       O(x_n)\quad\text{as }|x'|\to0\text{ and }x_n\to0,\\
     \label{zetahat}    & \det J_F(x)>0\quad\text{in }B'_{r_0},\\
\label{sviluppoF-2} &J_{F}(x)= \mathrm{Id}_{n+1}+O(|x|)
     \quad\text{and}\quad J_{F^{-1}}(x)= \mathrm{Id}_{n+1}+O(|x|)\quad \text{as $|x|\to 0^+$},     
   \end{align}
where $\mathop{\rm Id}_{{n+1}}$ denotes the identity
$(n+1)\times (n+1)$ matrix and $O(|x|)$ denotes  matrices with all
entries being $O(|x|)$ as $|x|\to 0^+$;
\item the $(n+1)\times (n+1)$ matrix-valued function  $\mathcal A$
  defined, for every $(x',x_n)\in B_{r_0}'$, as
\begin{equation*}
  \mathcal A(x',x_n):= \big(\det J_F(x',x_n)\big) (J_F(x',x_n))^{-1}
  ((J_F(x',x_n))^{-1})^T, 
\end{equation*}
has Lipschitz continuous entries and can be written as the block matrix 
\begin{equation}\label{Ablockmatrix1}
  \mathcal A(x',x_n)=
  \left(
    \renewcommand{\arraystretch}{1.5}
    \begin{array}{c|c}
 \mathcal D(x',x_n) & \mathbf{0} \\
  \hline
  \mathbf{0}^T & \det J_F(x',x_n)
    \end{array}
  \right),
\end{equation}
with 
\begin{equation}\label{D}
 \mathcal D(x',x_n)=
  \left(
    \renewcommand{\arraystretch}{1.5}
    \begin{array}{c|c}
      \mathrm{Id}_{n-1}+O(|x'|^2)+O(x_n) & O(x_n) \\
      \hline
      O(x_n) & 1+O(|x'|^2)+O(x_n)
    \end{array}
  \right),
\end{equation}
where $\mathrm{Id}_{n-1}$ is the identity $(n-1) \times (n-1)$
matrix, $O(x_n)$ and $O(|x'|^2)$ denote blocks of matrices with all
elements being $O(x_n)$ as $x_n\to 0$ and $O(|x'|^2)$ as
$|x'|\to 0$, respectively.

\end{enumerate}
\end{proposition}

Let $v\in H^1(\mathcal{C}_{\Omega},t^{1-2s}dz)$ be  a weak solution of
problem \eqref{eqrisoltadaV}, so that the function $\tilde v$ defined
in \eqref{VinfunxdiW} weakly solves \eqref{eqrisoltadaW}. Let us
consider the diffeomorphism introduced in Proposition
\ref{p:diff-deform-boun} and the function
\begin{equation*}
\hat{v}:= \tilde{v} \circ F\in H^1(\mathcal Q_{r_0}, t^{1-2s}dz).
\end{equation*}
From \eqref{eqrisoltadaW} it follows that $\hat v$ is a weak solution to 
\begin{equation}\label{problemaraddrizzato}
\begin{cases}
- \mathop{\rm div}(t^{1-2s} \mathcal  A \nabla \hat{v}) +
t^{1-2s}\hat{\mathbf{b}}(x)\cdot
\nabla _x \hat{v} +t^{1-2s} \hat{a}(x) \hat{v}=0
&\text{in $\mathcal{Q}_{r_0}$} , \\
\mathcal D\nabla _x \hat{v}\cdot\nu=0  &
\text{on $\partial_V \mathcal Q_{r_0}$} , \\
\lim_{t \to 0^+} (t^{1-2s}\mathcal A\nabla\hat{v}\cdot\mathbf{e}) =
\hat{h}(x)
\mathop{\rm Tr}\hat{v}
& \text{on $\mathcal Q_{r_0}'$}, 
\end{cases}
\end{equation}
 where $\mathbf{e}=(0,\cdots,0,-1)$,
$\mathcal Q_{r_0}':=\{(x',x_n)\in B_{r_0}':x_n<0\}$,
\begin{equation}\label{bhat}
  \hat{\mathbf{b}}(x) = \big(\det
  J_F(x))(J_{\Psi}(x))^{-1}\mathbf{b}(\Psi(x)), \quad
  \hat{a}(x)= \big(\det J_F(x)) a(\Psi(x)), 
\end{equation}
and 
\begin{equation}\label{hhat}
  \hat{h}(x)=\big(\det J_F(x)) h(\Psi(x))\in W^{1,p}(\mathcal Q_{r_0}'), 
\end{equation}
as a consequence of \eqref{ipotesisuh}.

Now we perform an even reflection through the hyperplane
$\{x_n=0\}$, in order to extend the equation in the whole half ball
$B^+_{r_0}$.  Let
\begin{equation}\label{Wwidehat}
 w(x',x_n,t):=
\begin{cases}
\hat{v}(x',x_n,t) &\text{if } (x',x_n,t)\in \mathcal Q_{r_0},\\
\hat{v}(x',-x_n,t) &\text{if } (x',x_n,t)\in B^+_{r_0}\setminus \mathcal Q_{r_0},
\end{cases} 
\end{equation}
and
\begin{equation*}
  B:=
    \renewcommand{\arraystretch}{1.5}
  \left(\begin{array}{c|c}B' &\mathbf{0} \\
          \hline 
          \mathbf{0}^T&1 \end{array}\right),
          \quad\text{where}\quad  
  B':=
  \left(\begin{array}{c|c}\mathop{\rm Id}_{n-1} &\mathbf{0} \\
          \hline \mathbf{0}^T&-1 \end{array}\right).
    \end{equation*}
    Thus $w\in H^1(B^+_{r_0}, t^{1-2s}dz)$ is a weak solution to 
\begin{equation}\label{problemadiWtilde}
\begin{cases}
- \mathop{\rm div}(t^{1-2s} A \nabla w) +
t^{1-2s}\widetilde{\mathbf{b}}\cdot
\nabla _x w + t^{1-2s}\tilde{a}(x) w=0 &\text{in $B^+_{r_0}$} , \\
-\lim_{t \to 0^+} t^{1-2s}\zeta\partial_t w = \tilde{h} \mathop{\rm Tr}w
& \text{on $B'_{r_0}$} ,
\end{cases}
\end{equation}
where 
\begin{align}\label{Atilde}
 A(x',x_n)&=
\begin{cases}
\mathcal A(x',x_n) &\text{if } x_n>0,\\
B \mathcal A(x',-x_n)B &\text{if } x_n<0,
\end{cases}\\ 
\notag\widetilde {\mathbf{b}}(x',x_n)&=
\begin{cases}
\hat {\mathbf{b}}(x',x_n) &\text{if } x_n>0,\\
\hat {\mathbf{b}}(x',-x_n)B' &\text{if } x_n<0,
\end{cases}\\
  \notag
 \tilde a(x',x_n)&=
\begin{cases}
\hat a(x',x_n) &\text{if }  x_n>0,\\
\hat a(x',-x_n) &\text{in } x_n<0,
\end{cases}
\end{align}
and
\begin{equation}\label{zetatilde}
 \zeta(x',x_n)=
\begin{cases}
\det J_F(x',x_n) &\text{if }  x_n>0,\\
\det J_F(x',-x_n) &\text{if }  x_n<0.
\end{cases} 
\end{equation}
Moreover
\begin{equation*}
\tilde h(x',x_n)=
\begin{cases}
\hat h(x',x_n) &\text{if }  x_n>0,\\
\hat h(x',-x_n) &\text{if }  x_n<0,
\end{cases} 
\end{equation*}
so that $\tilde{h}\in W^{1,p}(B'_{r_0})$, thanks to \eqref{hhat}.
We note that the boundary condition on $B_{r_0}'$ in
  \eqref{problemadiWtilde} is in fact a weighted Neumann condition,
  since
  \begin{equation*}
    A\nabla w\cdot \mathbf{e}=   \nabla w\cdot
    A(0,\dots,0,-1)=-\zeta\partial_t w\quad \text{on }B_{r_0}'.
  \end{equation*}
By
\eqref{regularitybanda} and \eqref{bhat}, recalling that
$F|_{B_{r_0}}: B_{r_0}\to F(B_{r_0})$ is a diffeomorphism of class
$C^{1,1}$, we thus have that
\begin{equation}\label{nuovaregolarita}
\widetilde{\mathbf{b}}\in L^\infty(B_{r_0}')\quad\text{and}\quad 
    \tilde a \in L^\infty(B_{r_0}').
\end{equation}
Furthermore \eqref{zetahat} and \eqref{zetatilde} imply that
\begin{equation}\label{stimadizeta}
  \zeta(x',x_n)= 1+O(|x'|^2)+O(x_n) \quad
  \text{as $|x'|\to 0$ and $x_n\to 0$}. 
\end{equation}
Also, defining 
\begin{equation}\label{Dtilde}
 D(x',x_n):=
\begin{cases}
\mathcal D(x',x_n) &\text{if } x_n>0,\\
B' \mathcal D(x',-x_n)B' &\text{if } x_n<0,
\end{cases} 
\end{equation}
from \eqref{Ablockmatrix1}, \eqref{Atilde} and \eqref{zetatilde}, it
follows that
\begin{equation}\label{Ablockmatrix}
 A=
  \left(
    \renewcommand{\arraystretch}{1.5}
    \begin{array}{c|c}
  D & \mathbf{0} \\
  \hline
  \mathbf{0}^T & \zeta
    \end{array}
  \right).
\end{equation} 
In particular, 
\begin{equation}\label{Lip}
    \text{$A$ has Lipschitz coefficients}
\end{equation}
because $\mathcal A$ does so and 
  \begin{equation*}
    B' \mathcal D(x',0)B'=\mathcal D(x',0)
  \end{equation*}
  thanks to \eqref{D}, so that no discontinuities appear in the matrix
  as a result of the even reflection across the axis
  $x_n=0$. Thanks to \eqref{det}, \eqref{Ablockmatrix1} 
\eqref{D}, \eqref{zetatilde}, \eqref{Dtilde}, and \eqref{Ablockmatrix},
we can conclude that
\begin{equation}\label{Aperturbazione}
  A (x)- \mathrm{Id}_{n+1} =O(|x|) \quad \text{as $|x|\to 0$}.
\end{equation}
Accordingly, $A$ is uniformly elliptic if $r_0$ is chosen sufficiently
small from the beginning; provided $r_0$ is chosen sufficiently
  small, there is no loss of generality in assuming, for example,
  that for every $z\in \R^{n+1}$ and $x\in B'_{r_0}$
\begin{equation}\label{normadiAlimitata}
\frac{1}{2}|z|^2\leq A(x)z\cdot z\leq 2|z|^2.
\end{equation}
For every $x\in B_{r_0}'$ and $z\in\mathbb{R}^{n+1}$, we
define $dA(x)zz\in\mathbb{R}^{n+1}$  as
\begin{equation}\label{dCtilde}
  dA (x) zz:= \left(\sum_{j,k=1}^{n+1}\frac{\partial a_{jk}}{\partial
      x_1}z_j z_k,
    \dots,\sum_{j,k=1}^{n+1}\frac{\partial a_{jk}}{\partial x_{n}}
    z_j z_k,0\right),
\end{equation} 
with $a_{jk}$ denoting the entries of $A$.  Furthermore, we
consider the functions
\begin{equation}\label{mu}
\mu(z)=\mu(x,t):= \begin{cases}
(A(x) z\cdot z)/|z|^2 &\text{if $z\in B_{r_0}\setminus\{0\}$},\\
1 &\text{if $z=0$},
\end{cases}
\end{equation}
\begin{equation}\label{beta}
  \beta(z)=\beta(x,t):=\frac{A(x)z}{\mu(z)}=\left(\frac{D(x)x}{\mu(z)},
    \frac{\zeta(x) t}{\mu(z)}\right)\quad \text{for every $z\in B_{r_0}$},
\end{equation}
\eqref{Ablockmatrix} has been used, and 
\begin{equation*}
    \beta'(x):=\frac{D(x)x}{\mu(x,0)}\quad \text{for every $x\in B'_{r_0}$}.
\end{equation*}
\begin{proposition}
The following expansions hold
\begin{equation}\label{sviluppomu}
  \mu(z)=1+O(|z|) \quad\text{and}\quad \nabla \mu (z)= O(1) \quad \text{as $|z|\to 0$}.
\end{equation}
  Furthermore,  
\begin{equation}\label{1sumulimitata}
\mu(z)\geq \frac{1}{4}\quad \text{for every $z\in B_{r_0}$},
\end{equation}
\begin{equation}\label{lestimedibeta}
\begin{cases}
\beta(z)= z+O(|z|^2)=O(|z|)\quad \text{as $|z|\to 0$},&\\
J_\beta(z)= \mathrm{Id}_{n+1}+O(|z|)\quad \text{as $|z|\to 0$},&\\
\mathop{\rm div}\beta(z)=n+1+O(|z|)\quad \text{as $|z|\to 0$},&
\end{cases} 
\end{equation}
and 
\begin{equation}\label{lestimedibeta'}
    \begin{cases}
\beta'(x)= x+O(|x|^2)=O(|x|)\quad \text{as $|x|\to 0$},&\\
\mathop{\rm div}\beta'(x)=n+O(|x|)\quad \text{as $|x|\to 0$}.&
\end{cases} 
\end{equation}
\end{proposition}
\begin{proof} 
  See  \cite[Section 2]{DelFelVit22}.
\end{proof}

\subsection{A Pohozaev-type identity}
We derive a Pohozaev-type identity, representing the integral
  counterpart of a Rellich-Ne\v{c}as identity; this will be a key tool
  for establishing a lower bound on the derivative of the Almgren
frequency function in the next subsection.
\begin{proposition}[Pohozaev-type identity]\label{Ppoho}
  Let $w\in H^1(B^+_{r_0}, t^{1-2s}dz)$ be a weak solution to problem
  \eqref{problemadiWtilde}. Then, for a.e. $r\in (0,r_0)$ we have
\begin{align}\label{poho}
&r\int_{\partial^+ B^+_{r}} t^{1-2s} A \nabla w\cdot\nabla w \,dS-
                            2r\int_{\partial^+ B^+_{r}}
                            t^{1-2s}\frac{(A\nabla w \cdot\nu
                            )^2}{\mu} \,dS
                            - r\int_{S_r}\tilde{h}|\mathop{\rm Tr}w|^2\, dS'\\
  \notag
  &+ \int_{B'_r} |\mathop{\rm Tr}w |^2\,(\nabla
    _x\tilde{h}\cdot\beta'
    +\tilde{h}\,\mathop{\rm div}\nolimits_x\beta' )\,dx = \int_{B^+_{r}}
    (t^{1-2s}A \nabla w\cdot\nabla w)
    \mathop{\rm div}\beta\,dz\\
\notag &-  2\int_{B^+_{r}}t^{1-2s}(\beta\cdot\nabla w)
         (\widetilde{\mathbf{b}}\cdot\nabla _x w
         +\tilde{a} w)\,dz-  2\int_{B^+_{r}}t^{1-2s}J_\beta (A\nabla w)  \cdot\nabla w\,dz\\
\notag &+\int_{B^+_{r}}t^{1-2s} (dA\nabla  w \nabla w) \cdot\beta\,dz
         +
         (1-2s)\int_{B^+_{r}}t^{1-2s}\frac{\zeta}{\mu}A\nabla  w\cdot\nabla w\,dz,
\end{align}
where $dS$ and $dS'$ denote the volume elements on the $N$-dimensional
surface $\partial^+ B^+_{r}$ and the $(N-1)$-dimensional surface
$S_r$, respectively, and $\nu$ is the outer normal vector to
  $\partial^+ B^+_{r}$. Moreover, for every $r\in (0,r_0)$,
\begin{align}\label{identity}
    \int_{B^+_r} t^{1-2s} \big(A\nabla  w\cdot\nabla  w&+
                                                     (\widetilde{\mathbf{b}}\cdot
                                                     \nabla_x  w)w + \tilde{a} w^2\big)\,dz\\
\notag    &=\int_{B'_r} \tilde{h} |\mathop{\rm Tr}w|^2\,
            dx
            +\int_{\partial^+ B^+_r} t^{1-2s}(A\nabla  w\cdot\nu) w\,dS.
\end{align}
\end{proposition}
\begin{proof}
The identity \eqref{poho} follows directly from \cite[Proposition 2.3]{FelSic22} with 
\begin{equation*}
  c= \widetilde{\mathbf{b}}\cdot\nabla _x  w+\tilde{a} w\in L^2(B^+_{r_0}, t^{1-2s}dz),
\end{equation*}
taking into account \eqref{nuovaregolarita}.  As for \eqref{identity},
it suffices to apply \cite[Proposition 3.7]{FelSic22} by choosing the
solution $ w\in H^1(B^+_{r_0}, t^{1-2s}dz)$ as the test function
$\phi$.  The proof is thereby complete.
\end{proof}

\begin{remark} The integrals on the first line of \eqref{poho}
    can be understood for a.e. $r\in (0,r_0)$ as $L^1$-functions given
    by the weak derivatives of the $W^{1,1}(0,r_0)$-functions
    $r\mapsto \int_{B_r^+} t^{1-2s} A \nabla w\cdot\nabla w \,dz$,
    $r\mapsto \int_{B^+_{r}} t^{1-2s}\frac{(A\nabla w \cdot\nu
      )^2}{\mu} \,dz$, and
    $r\mapsto \int_{B_r'}\tilde{h}|\mathop{\rm Tr}w|^2\, dx$.
  \end{remark}

We conclude this subsection with an estimate of the term
\begin{equation*}
  \int  _{\partial^+B^+_r}t^{1-2s}A\nabla  w\cdot\nabla w\,dS -
  \int_{S_r}\tilde{h}(x)|\mathop{\rm Tr} w|^2\,dS'
\end{equation*}
by means of \eqref{poho}. For this, we recall from \cite[Lemma
2.6]{FalFel14} the following result.
\begin{lemma}\label{lemmafallfelli}
  There exists $S_{n,s}>0$ such that, for all
  $r>0$ and $w\in H^1(B^+_r, t^{1-2s}dz)$,
    \begin{equation}\label{eq:trem}
\left(\int_{B'_r}|\mathop{\rm
    Tr}w|^{2^\ast_s}dx\right)^{\frac{2}{2^\ast_s}}
\leq S_{n,s}\left(\frac{n-2s}{2r}\int_{\partial^+B^+_r}t^{1-2s}w^2\,dS
  +\int_{B^+_r} t^{1-2s}|\nabla w|^2\,dz\right),
    \end{equation}
    being $2^\ast_s:= 2n/(n-2s)$.
\end{lemma}
The right hand side of \eqref{eq:trem} is well defined for every
  $r>0$ and $w\in H^1(B^+_r, t^{1-2s}dz)$, thanks to the existence of a
  natural linear and continuous trace operator
        \begin{equation}\label{traceoper}
          \gamma_r: H^1(B^+_r, t^{1-2s}dz) \to
          L^2(\partial^+B^+_r,t^{1-2s}dS),
        \end{equation}
        which also happens to be compact.  For
        notational simplicity, in the sequel we will write
        $w$ instead of $\gamma_r(w)$.

\begin{proposition}\label{proputile}
  Let $w\in H^1(B^+_{r_0}, t^{1-2s}dz)$ be a weak solution to problem
  \eqref{problemadiWtilde}. Then
    \begin{align}\label{useful}
      \int  _{\partial^+B^+_r}&t^{1-2s}A\nabla  w\cdot\nabla w\,dS -
                                \int_{S_r}\tilde{h}|\mathop{\rm Tr}
                                w|^2\,dS'
                                =\,2\int_{\partial^+B^+_r}t^{1-2s}
                                \frac{(A\nabla  w\cdot\nu)^2}{\mu}\,dS \\
      \notag
                              &-\frac{2}{r}\int_{B^+_{r}}t^{1-2s}(\beta\cdot\nabla  w)
                                (\widetilde{\mathbf{b}}\cdot\nabla _x
                                w+\tilde{a} w)\,dz
                                + \frac{n-2s}{r}\int_{B^+_{r}}t^{1-2s}
                                A\nabla  w\cdot\nabla  w\,dz \\
      \notag    &+O(r^{-1+\frac{2sp-n}{p}})\left(\int_{B'_r}|
                  \mathop{\rm Tr}
                  w|^{2^\ast_s}\,dx\right)^{\frac{2}{2^\ast_s}}
                  + O(1)\int_{B^+_r}t^{1-2s}|\nabla  w|^2\,dz
    \end{align}
    for a.e. $r\in (0,r_0)$ as $r\to 0^+$.
\end{proposition}
\begin{proof}
  We start by reframing \eqref{poho} as follows:  
\begin{align}\label{comb1}
  \int  _{\partial^+B^+_r}&t^{1-2s}A\nabla w\cdot\nabla w\,dS
                            -  \int_{S_r}\tilde{h}|\mathop{\rm Tr}
                            w|^2\,dS'=
                            \,2\int_{\partial^+B^+_r}t^{1-2s} \frac{(A\nabla  w\cdot\nu)^2}{\mu}\,dS\\
  \notag &-\frac{1}{r} \int_{B'_r} (\mathop{\rm Tr} w)^2 \,
           (\nabla _x\tilde{h}\cdot\beta'+\tilde{h}\,\mathop{\rm
           div}\nolimits_x\beta' )\,dx
           +\frac{1}{r}\int_{B^+_r}(t^{1-2s}A\nabla  w\cdot\nabla w)\mathop{\rm div}\beta\,dz\\
  \notag   &-  \frac{2}{r}\int_{B^+_{r}}t^{1-2s}(\beta\cdot\nabla  w)
             (\widetilde{\mathbf{b}}\cdot\nabla _x  w+\tilde{a}
             w)\,dz-
             \frac{2}{r}\int_{B^+_{r}}t^{1-2s}J_\beta (A\nabla  w)  \cdot\nabla  w\,dz\\
  \notag   &+\frac{1}{r}\int_{B^+_{r}}t^{1-2s} (dA\nabla  w \nabla w)
             \cdot\beta\,dz +
             \frac{1-2s}{r}\int_{B^+_{r}}t^{1-2s}\frac{\zeta}{\mu}A\nabla  w\cdot\nabla  w\,dz 
\end{align}
for a.e. $r\in (0,r_0)$.  Now we take advantage of
\eqref{stimadizeta}, \eqref{Aperturbazione}, \eqref{sviluppomu} and
\eqref{lestimedibeta} to deduce that
    \begin{align}\label{comb2}
      &\frac{1}{r}\int_{B^+_r}(t^{1-2s}A\nabla  w\cdot\nabla
        w)\mathop{\rm div}\beta\,dz-
        \frac{2}{r}\int_{B^+_{r}}J_\beta (t^{1-2s}A\nabla  w)  \cdot\nabla  w\,dz\\
      \notag&+\frac{1}{r}\int_{B^+_{r}}t^{1-2s} (dA\nabla  w \nabla w)
              \cdot\beta\,dz
              + \frac{1-2s}{r}\int_{B^+_{r}}t^{1-2s}\frac{\zeta}{\mu}A\nabla  w\cdot\nabla  w\,dz\\
      \notag & \quad = 
               \frac{n-2s}{r}\int_{B^+_r}t^{1-2s}A\nabla  w\cdot\nabla
               w \,dz+O(1)
               \int_{B^+_r}t^{1-2s}|\nabla  w|^2\,dz 
     \end{align}
     as $r\to 0^+$, being $dA\nabla w\nabla w=O(1)|\nabla w|^2$ as
     $r\to 0^+$ thanks to \eqref{Lip} and \eqref{dCtilde}.  In
     addition, using \eqref{lestimedibeta'}, the fact that
     $\tilde{h}\in W^{1,p}(B'_{r_0})$ and
     $ |\mathop{\rm Tr} w|^2\in L^{2^\ast_s/2}(B'_r)$ by Lemma
     \ref{lemmafallfelli}, we can apply the H\"{o}lder inequality to
     conclude that
    \begin{equation}\label{comb3}
           -\frac{1}{r} \int_{B'_r} |\mathop{\rm Tr} w|^2 \,(\nabla
           _x\tilde{h}\cdot\beta'
           +\tilde{h}\,\mathop{\rm div}_x\beta' )\,dx
        =O(r^{-1+\frac{2sp-n}{p}})\left(\int_{B'_r}|\mathop{\rm Tr}
          w|^{2^\ast_s}\,dx
        \right)^{\frac{2}{2^\ast_s}}
      \end{equation}
      as $r\to 0^+$. The combination of \eqref{comb1}, \eqref{comb2}
      and \eqref{comb3} leads us to \eqref{useful}.
\end{proof}

\subsection{The Almgren frequency function}
In this subsection we introduce the Almgren frequency function
$\mathcal{N}$; then we prove its boundedness close to the origin via
the monotonicity of an appropriate perturbation of it and show the
existence of $\lim_{r\to 0^+}\mathcal{N}(r)$. In particular, the
former information yields certain doubling conditions which, in turn,
are the main ingredients for establishing a strong unique continuation
property; the latter will come into play in the local asymptotics
provided in Section \ref{seclocal}.

To begin with, we fix a non-trivial weak solution
$w\in H^1(B^+_{r_0}, t^{1-2s}dz)$ to problem
\eqref{problemadiWtilde}. For every $r\in (0,r_0)$, we define the
energy function
\begin{equation}\label{E}
  E(r):= r^{2s-n}\biggl( \int_{B^+_r} t^{1-2s} (A\nabla  w\cdot\nabla
  w+
  (\widetilde{\mathbf{b}}\cdot\nabla_x  w) w + \tilde{a}w^2)\,dz -
  \int_{B'_r} \tilde{h} |\mathop{\rm Tr}w|^2\, dx\biggr)
\end{equation}
and the weighted mass 
\begin{equation}\label{H}
    H(r):= r^{2s-n-1}\int_{\partial^+ B^+_r} t^{1-2s}\mu w^2\,dS.
\end{equation}
Henceforth, we simply write $w$ for its trace on $\partial^+B^+_r$ for
notational simplicity.

Our first goal is to provide an estimate for the derivative of $E$ via
the Pohozaev identity \eqref{poho}; to this end, we first establish
some preliminary results.  
\begin{lemma}\label{lemmahardytype}
    For all $r>0$ and $w\in H^1(B^+_r, t^{1-2s}dz)$ we have 
\begin{equation*}
        \int_{B^+_r}t^{1-2s}\frac{|w|^2}{|z|^2}\,dz
        \leq \frac{16}{(n-2s)^2}\biggl(\int_{B^+_r}t^{1-2s}|\nabla
        w|^2\, dz+
        \frac{n-2s}{2r}\int_{\partial^+B^+_r}t^{1-2s}\mu |w|^2\,dS\biggr).
   \end{equation*}
   \end{lemma}
\begin{proof}
  The stated inequality directly follows from \cite[Lemma
  2.4]{FalFel14} and \eqref{1sumulimitata}.
\end{proof} Let us define
  $H^{1}({\mathbb S}^+,\theta_{n+1}^{1-2s}dS)$ as the completion of
  $C^\infty(\overline{{\mathbb S}^+})$ with respect to the norm
\begin{equation*}
\|\psi\|_{H^{1}({\mathbb S}^+,\theta_{n+1}^{1-2s}dS)}=\bigg(
\int_{{\mathbb S}^+}\theta_{n+1}^{1-2s}\big(|\nabla_{{\mathbb
        S}}\psi(\theta)|^2+\psi^2(\theta)\big)\,dS
\bigg)^{\!\!1/2},
\end{equation*}
where  $\nabla_{{\mathbb S}}$ denotes the Riemannian
 gradient on the unit sphere ${\mathbb S}$ endowed
 with the standard metric.

The next lemma provides a trace operator (again denoted as
$\mathop{\rm Tr}$) 
\begin{equation*}
  \mathop{\rm Tr}:    H^1(\mathbb{S}^+, \theta^{1-2s}_{n+1}dS)
  \to L^{2^\ast_s} (\mathbb{S}')
\end{equation*}
from $H^1(\mathbb{S}^+, \theta^{1-2s}_{n+1}dS)$ into
$L^{2^\ast_s} (\mathbb{S}')$, which is not only linear and
  continuous, but also compact, due to the validity of an Ehrling-type
  inequality.  To this aim, it is useful to recall the following
Sobolev trace inequality: 
\begin{equation}\label{Sobolev-trace} c_{n,s} \Vert
  \mathop{\rm Tr}v\Vert _{L^{2^\ast_s}(\R^n)}^2\leq
  \int_{\R^{n+1}_+}t^{1-2s}|\nabla v|^2\, dz\quad\text{for every
    $v\in C^\infty_c(\overline{\R^{n+1}_+})$,}
\end{equation}
for some positive constant $c_{n,s}>0$ depending only on $n$ and $s$,
see e.g. \cite[Equation (21)]{FalFel14}.
\begin{lemma}[Ehrling-type inequality for traces]
  \label{lemmadiimmersione}
  There exists a linear and  bounded trace operator
    $\mathop{\rm Tr}: H^1(\mathbb{S}^+, \theta^{1-2s}_{n+1}dS)\to
    L^{2^\ast_s} (\mathbb{S}')$ satisfying the following Ehrling-type
    inequality: for every $\varepsilon>0$, there exists
  $C_{\varepsilon,n,s}>0$ depending only on $\varepsilon$, $n$ and $s$
  such that 
\begin{equation}\label{sobolevineq}
  \left(\int_{\mathbb{S}'}|\mathop{\rm Tr}\psi|^{2^\ast_s}\,dS'
  \right)^{\frac{2}{2^\ast_s}}\leq
  C_{\varepsilon,n,s}\int_{\mathbb{S}^+} \theta_{n+1}^{1-2s}
  |\psi(\theta)|^2\,dS +\varepsilon \int_{\mathbb{S}^+}
  \theta_{n+1}^{1-2s}
  |\nabla_{\mathbb{S}}\psi(\theta)|^2\,dS
\end{equation}
for every $\psi\in H^1(\mathbb{S}^+,
\theta^{1-2s}_{n+1}dS)$. Furthermore,
$\mathop{\rm Tr}: H^1(\mathbb{S}^+, \theta^{1-2s}_{n+1}dS)\to
L^{2^\ast_s} (\mathbb{S}')$ is compact.
\end{lemma}
\begin{proof}
  By density, it is sufficient to prove inequality \eqref{sobolevineq}
  for every $\psi\in C^\infty(\overline{\mathbb{S}^+})$. For this, let
  $\psi\in C^\infty(\overline{\mathbb{S}^+})$ and
  $f\in C^\infty_c (0,+\infty)$, $f\not\equiv0$. We define
  $f_k(\cdot)= f(\cdot - k)$, with $k \geq 1$ large, to be chosen
    later depending on $\e$.  We introduce the polar
    coordinates
\begin{equation*}
  r=|z|>0 \quad \text{and}\quad \theta=\frac{z}{|z|}\in
  \mathbb{S}^+\quad \text{for every $z\in \R^{n+1}_+$ },
\end{equation*}
     and plug the function 
    $f_k(r)\psi(\theta)$ into \eqref{Sobolev-trace}, thus obtaining 
    \begin{align*}
      c_{n,s} &\left(\int_0^{+\infty}r^{n-1}|f_k(r)|^{2^\ast_s}\,dr
                \right)^{\frac{2}{2^\ast_s}}\left(\int_{\mathbb{S}'}
                |\psi(\theta',\theta_n,0)|^{2^\ast_s}\,dS'
                \right)^{\frac{2}{2^\ast_s}}\\
      \leq &\left(\int_0^{+\infty}
             r^{n+1-2s}|f_k'(r)|^2\,dr\right)\left(\int_{\mathbb{S}^+}
             \theta_{n+1}^{1-2s}|\psi(\theta)|^2\,dS\right)\\
              &\quad +\left(\int_0^{+\infty}
                r^{n-1-2s}|f_k(r)|^2\,dr\right)
                \left(\int_{\mathbb{S}^+} \theta_{n+1}^{1-2s}
                |\nabla _{\mathbb{S}}\psi(\theta)|^2\,dS\right).
    \end{align*}
    We observe that, if $k\geq 1$ is sufficiently large, then
\begin{equation}\label{eq:stquot}
  \frac{\int_0^{+\infty} r^{n-1-2s}|f_k(r)|^2\,dr}
  {\left(\int_0^{+\infty}r^{n-1}|f_k(r)|^{2^\ast_s}\,dr\right)^{\frac{2}{2^\ast_s}}}\leq
  \mathrm{const} \,  k^{-\frac{2s}{n}},
\end{equation}
for some $\mathrm{const}>0$ independent of $k$.  Therefore, to
  prove the desired trace inequality, given $\e>0$, it is sufficient
  to choose $k$ large enough such that the right-hand side of
  \eqref{eq:stquot} is smaller than~$\e$. Finally, we observe that the
  compactness of the trace map is a direct and standard consequence of
  the validity of the Ehrling-type inequality \eqref{sobolevineq}.
\end{proof}
As an immediate corollary of the previous lemma we obtain a
Sobolev-type inequality with boundary terms on the half-spheres
  $\partial^+ B^+_r$. To state this result precisely, we introduce
the following definition.  
\begin{definition}\label{spaziosulbordo}
  For every $r>0$, we define the Sobolev space
  $H^1(\partial^+ B^+_r, t^{1-2s}dS)$ as the completion of
  $C^\infty(\overline{\partial^+ B^+_r})$ with respect to the norm
\begin{equation*}
  \Vert v\Vert _{H^1(\partial^+ B^+_r, t^{1-2s}dS)} :=
  \left(\int_{\partial^+ B^+ _r} t^{1-2s}
    (v^2+|\nabla_{\partial B_r}
    v|^2)\,dS\right)^{1/2},
\end{equation*}
where $\nabla_{\partial B_r} v$ is the Riemannian
  gradient of $v$ on $\partial B_r$.
\end{definition}
We recall that, if $v\in C^\infty(\overline{\partial^+ B^+_r})$
  is the restriction to $\overline{\partial^+ B^+_r}$ of
  some function $\tilde v$ which is smooth in a neighbouhood of
  $\overline{\partial^+ B^+_r}$, then
  \begin{equation*}
    \nabla_{\partial B_r}v(z)=\nabla \tilde
    v(z)-\frac1{r^2}(\nabla\tilde v(z)\cdot z) z\quad\text{for every
    }z\in \partial^+ B^+_r,
  \end{equation*}
i.e. $\nabla_{\partial B_r}v(z)$ is the orthogonal projection of $\nabla
\tilde v(z)$ on the tangent space to $\partial B_r$ at $z$. In
particular
\begin{equation*}
  | \nabla_{\partial B_r}v(z)|^2\leq | \nabla_{\partial B_r}v(z)|^2
  +\frac1{r^2}(\nabla\tilde v(z)\cdot z)^2=|\nabla \tilde
  v(z)|^2 \quad\text{for every
  }z\in \partial^+ B^+_r.
\end{equation*}
Moreover, for every $v\in C^\infty(\overline{\partial^+ B^+_r})$,
\begin{equation}\label{eq:scal-grad}
  \nabla_{\partial B_r} v(z)= \frac{1}{r} \nabla _{\mathbb{S}}
v_r\left(\frac{z}{r}\right)\quad \text{for all $z\in \partial^+B^+_r$},
\end{equation}
where $v_r \in C^\infty(\overline{\mathbb S^+})$ is defined as
$v_r(\theta):= v(r\theta)$ for every $\theta\in \overline{\mathbb{S}^+}$.

We also observe that, if $w\in H^1(B^+_{R}, t^{1-2s}dz)$ for some
  $R>0$, the  coarea formula  and a density argument imply that, for
  a.e. $r\in(0,R)$, $w\big|_{\partial^+B^+_r}\in H^1(\partial^+ B^+_r,
  t^{1-2s}dS)$ with
  \begin{equation}\label{eq:stgr}
    \nabla_{\partial B_r}\big(w\big|_{\partial^+B^+_r}\big)=\nabla w
    -\frac1{r^2}(\nabla w\cdot z) z\quad
\text{and}\quad \big|\nabla_{\partial^+
      B^+_r}\big(w\big|_{\partial B_r}\big)\big|\leq |\nabla w|
    \quad\text{ a.e. in } \partial^+ B^+_r.
  \end{equation}
\begin{lemma}\label{lemmasobolev}
  There exists a positive constant $\bar{S}_{n,s}>0$ depending only on
  $n$ and $s$ such that, for every $r>0$ and
  $v\in H^1(\partial^+B^+_r, t^{1-2s}dS)$,
  \begin{equation*}
    r^{\frac{n}{2s}}\left(\int_{S_r}|\mathop{\rm
        Tr}v|^{2^\ast_s}\,dS'\right)^{\frac{2}{2^\ast_s}}\leq
    \bar{S}_{n,s} \int_{\partial^+ B^+_r}
    t^{1-2s}(r^{-2}|v|^2+|\nabla_{\partial B_r}
    v|^2)\,dS.
    \end{equation*}
\end{lemma}
\begin{proof} 
  We apply Lemma \ref{lemmadiimmersione} with $\varepsilon=1$ to the
  function
  $v_r :=v(r\,\cdot)\in H^1(\mathbb{S}^+,\theta_{n+1}^{1-2s}dS)$ and,
  in view of \eqref{eq:scal-grad}, we obtain
\begin{align*}
  \biggl(\int_{\mathbb{S}'}|\mathop{\rm Tr}
                              v_r(\theta',\theta_n)|^{2^\ast_s}&dS'
                              \biggr)^{\frac{2}{2^\ast_s}}\leq
                              C_{1,n,s}
                              \int_{\mathbb{S}^+}
                              \theta_{n+1}^{1-2s}|v_r(\theta)|^2\,dS
                              + \int_{\mathbb{S}^+}
                              \theta_{n+1}^{1-2s}|\nabla_{\mathbb{S}}
                              v_r(\theta)|^2\,dS\\
  &=  C_{1,n,s}\int_{\mathbb{S}^+}
           \theta_{n+1}^{1-2s}|v(r\theta)|^2\,dS
           + \int_{\mathbb{S}^+}
           \theta_{n+1}^{1-2s}r^2|\nabla_{\partial B_r}
           v(r\theta)|^2\,dS\\
  &=  C_{1,n,s}r^{2s-n-1}\int_{\partial^+ B^+_r} t^{1-2s}|v|^2\,dS +
           r^{2s-n+1}\int_{\partial^+ B^+_r} t^{1-2s}|\nabla_{\partial B_r} v|^2\,dS\\
  &\leq \bar{S}_{n,s}r^{2s-n+1}\left(r^{-2}\int_{\partial^+ B^+_r}
       t^{1-2s}|v|^2\,dS
       + \int_{\partial^+ B^+_r} t^{1-2s}|\nabla_{\partial B_r} v|^2\,dS\right),
\end{align*}
with $\bar{S}_{n,s}:= \max\{1,C_{1,n,s}\}$. 
From this we deduce that
   \begin{align*}
\left(\int_{S_r}|\mathop{\rm Tr}
     v|^{2^\ast_s}\,dS'\right)^{\frac{2}{2^\ast_s}}& =\left(r^{n-1}\int_{\mathbb{S}'}
  |\mathop{\rm Tr} v(r\theta',r\theta_n)|^{2^\ast_s}\,dS'\right)^{\frac{2}{2^\ast_s}}\\
&\leq  \bar{S}_{n,s}r^{\frac{(n-1)(n-2s)}{n}+2s-n+1}\int_{\partial^+
                                                                                          B^+_r} t^{1-2s}(r^{-2}|v|^2+ |\nabla_{\partial B_r} v|^2)\,dS\\
&= \bar{S}_{n,s} r^{\frac{2s}{n}}\int_{\partial^+ B^+_r} t^{1-2s}(r^{-2}|v|^2+ |\nabla_{\partial B_r} v|^2)\,dS,
   \end{align*}
   and  the proof is thereby complete. 
 \end{proof}
 \begin{lemma}\label{lemmastima1}
   Let $g\in L^p(B'_{r_0}) $ for some $p>\frac{n}{2s}$.  Then,
     possibly choosing $r_0$ smaller from the beginning, for every
   $r\in (0,r_0)$ and $v\in H^1(B^+_r, t^{1-2s}dz)$
\begin{align*}
      \int_{B^+_r} t^{1-2s} |\nabla v|^2\,dz &+\frac{n-2s}{2r}\int_{\partial^+
          B^+_r}
      t^{1-2s}\mu|v|^2\,dS+\biggl(\int_{B'_r}|
       \mathop{\rm Tr}v|^{2^\ast_s}\,dx\biggr)^{\frac{2}{2^\ast_s}}\\
      \leq \,C_{n,s} \biggl(&\int_{B^+_r}t^{1-2s}\big(A\nabla v
                              \cdot\nabla
                              v+(\widetilde{\mathbf{b}}\cdot
                              \nabla _x v )v+\tilde{a}(x)v^2\big)\,dz\\
      &-\int_{B'_r}g(x) |\mathop{\rm Tr}v|^2\,dx+\frac{n-2s}{2r}
        \int_{\partial^+B^+_r}t^{1-2s}\mu |v|^2\,dS \biggr) 
    \end{align*}
    for some positive constant $C_{n,s}>0$ depending only on $n$ and
    $s$.
\end{lemma}
\begin{proof} 
We start by observing that by \eqref{normadiAlimitata}  
\begin{equation}\label{Amaggiore}
  \int_{B^+_r} t^{1-2s}A\nabla v\cdot\nabla v\,dz\geq \frac{1}{2}
  \int_{B^+_r} t^{1-2s} |\nabla v|^2\,dz\quad \text{for every $r\in (0,r_0)$}. 
\end{equation}
Moreover, by H\"{o}lder's inequality, the summability of
$g$ and the fact that $|\mathop{\rm Tr}v|^2\in L^{2^\ast_s/2}(B'_r)$
(as a result of Lemma \ref{lemmafallfelli}), we deduce that
\begin{equation}\label{intB'r}
  \int_{B'_r} g(x)|\mathop{\rm Tr}v|^2\,dx \leq
  \tilde{C} r^{\frac{2sp-n}{p}}\Vert g\Vert_{L^p(B'_{r_0})}
  \left(\int_{B'_r}|\mathop{\rm
      Tr}v|^{2^\ast_s}\,dx\right)^{\frac{2}{2^\ast_s}}
\end{equation}
where $\tilde{C}=\tilde{C}(n,s,p)=\omega_n^{\frac{2sp-n}{np}}>0$,
being $\omega_n$ the $n$-dimensional Lebesgue measure of $B'_1$.

On the other hand, by \eqref{nuovaregolarita}, Young's inequality, and
Lemma \ref{lemmahardytype}, we obtain
\begin{align}\label{1}
  \biggl|\int_{B^+_r}t^{1-2s}(\widetilde{\mathbf{b}}\cdot\nabla
  _xv)v\,dz\biggr|
  &\leq \,\Vert  \widetilde{\mathbf{b}} \Vert_{L^\infty(B'_{r_0})}r
    \int_{B_r^+}t^{1-2s} |\nabla _xv|\frac{|v|}{r}\,dz\\
  \notag&\leq \,\Vert  \widetilde{\mathbf{b}}
          \Vert_{L^\infty(B'_{r_0})}
          \frac{r}{2}\left(\int_{B^+_r}t^{1-2s} |\nabla _x v|^2\, dz
          +\int_{B^+_r}t^{1-2s} \frac{v^2}{|z|^2}\,dz\right)\\
  \notag &\leq \,C_1 r\left(\int_{B^+_r}t^{1-2s} |\nabla v|^2\, dz
           +\frac{n-2s}{2r}\int_{\partial^+ B^+_r} t^{1-2s} \mu |v|^2\, dS \right),
\end{align}
being
$C_1:=\frac12
  \|\widetilde{\mathbf{b}}\|_{L^\infty(B'_{r_0})}\big(\frac{16}{(n-2s)^2}+1\big)>0$
independent of $r$, and 
\begin{align}\label{2}
  \biggl|\int_{B^+_r} t^{1-2s} \tilde{a}v^2\,dz\biggr|
  &\leq \Vert \tilde{a}\Vert_{L^\infty(B'_{r_0})} r^2 \int_{B^+_r}
    t^{1-2s}\frac{v^2}{|z|^2}\,dz\\
  \notag&\leq   C_2 r^2\left(\int_{B^+_r}t^{1-2s} |\nabla  v|^2\, dz
          + \frac{n-2s}{2r}\int_{\partial^+ B^+_r} t^{1-2s} \mu |v|^2\, dS \right),
\end{align}
being
$C_2:= \Vert \tilde{a}\Vert_{L^\infty(B'_{r_0})}
\frac{16}{(n-2s)^2}>0$ independent of $r$.

Combining \eqref{Amaggiore}, \eqref{intB'r}, \eqref{1}, and \eqref{2}
finally yields
\begin{align}\label{quasistima}
  \biggl(&\int_{B^+_r}t^{1-2s}\big(A\nabla v\cdot\nabla
           v+(\widetilde{\mathbf{b}}
           \cdot\nabla _x v)v+\tilde{a}(x)v^2\big)\,dz-\int_{B'_r}g(x)
           |\mathop{\rm Tr}v|^2\,dx\\
  \notag&\quad+\frac{n-2s}{2r}\int_{\partial^+B^+_r}t^{1-2s}\mu |v|^2\,dS \biggr)\\
  \notag\geq &\left(\frac{1}{2}-C_1 r-C_2r^2\right)
               \int_{B^+_r}t^{1-2s} |\nabla v|^2\, dz -
               \tilde{C} r^{\frac{2sp-n}{p}}\Vert
               g\Vert_{L^p(B'_{r_0})}
               \left(\int_{B'_r}|\mathop{\rm Tr}v|^{2^\ast_s}\,dx
               \right)^{\frac{2}{2^\ast_s}}\\
  \notag&\quad+ \left(1-C_1 r-C_2 r^2\right)\frac{n-2s}{2r}\int_{\partial^+
          B^+_r}
          t^{1-2s}\mu|v|^2\,dS\\
  \notag\geq &
               \frac{3}{16}\left(\int_{B^+_r}t^{1-2s} |\nabla v|^2\,
               dz
+\frac{n-2s}{2r}\int_{\partial^+
          B^+_r}
          t^{1-2s}\mu|v|^2\,dS\right)\\
  \notag&\quad 
               +
               \left(\frac{1}{64\, S_{n,s}}- \tilde{C}
               r^{\frac{2sp-n}{p}}\Vert g\Vert_{L^p(B'_{r_0})}\right)
       \left(\int_{B'_r}|\mathop{\rm Tr}v|^{2^\ast_s}\,dx\right)^{\frac{2}{2^\ast_s}},
\end{align}
where we used Lemma \ref{lemmafallfelli} and  \eqref{1sumulimitata},
together with the fact that, possibly choosing $r_0$ smaller from the beginning, 
\begin{equation*}
  \frac{1}{2}-C_1 r -C_2 r^2\geq \frac{1}{4} \quad \text{for all $r\in
    (0,r_0)$}.
\end{equation*}
Possibly taking $r_0$ even smaller from the beginning so that
\begin{equation*}
  \frac{1}{64 \, S_{n,s}}- \tilde{C} r^{\frac{2sp-n}{p}}
  \Vert g\Vert_{L^p(B'_{r_0})}\geq   \frac{1}{65 \, S_{n,s}}
  \quad \text{for all $r\in (0,r_0)$},
\end{equation*}
 the conclusion follows from \eqref{quasistima}.
\end{proof}

\begin{lemma}\label{lemmaE'}
  For some fixed $w\in H^1(B^+_{r_0}, t^{1-2s}dz)\setminus\{0\}$
    weakly solving \eqref{problemadiWtilde}, let $E:~\!\!(0,r_0)\to\R$
    be the function defined in \eqref{E}. Then
  $E\in W^{1,1}_{\mathrm{loc}}(0,r_0]$. Furthermore, its distributional
    derivative $E'\in L^{1}_{\mathrm{loc}}(0,r_0]$ satisfies a.e. in
    $(0,r_0)$
\begin{align}\label{E'}
  \quad E'(r) &=r^{2s-n}\left(2\int_{\partial^+B^+_r} t^{1-2s}
           \frac{(A\nabla w\cdot\nu)^2}{\mu}\, dS+
           \int_{\partial ^+
           B^+_r}t^{1-2s}\big((\widetilde{\mathbf{b}}\cdot\nabla_xw)w
           +\tilde{a}w^2\big)\, dS\right)\\
\notag&\quad +O(r^{-1+\varepsilon})\left(E(r)+\frac{n-2s}{2}H(r)\right)
\end{align}
as $r\to 0^+$, with
\begin{equation}\label{vareps}
        \varepsilon:= \min\left\{1,\frac{2sp-n}{p}\right\}>0.
    \end{equation}
\end{lemma}
\begin{proof} 
  By definition, $E$ is the product between the
  $W^{1,\infty}_{\mathrm{loc}}(0,r_0]$-function $r\mapsto r^{2s-n}$
  and the $W^{1,1}(0,r_0)$-function
\begin{equation*}
  r\mapsto \int_{B^+_r} t^{1-2s} \big(A\nabla w\cdot\nabla w
  + (\widetilde{\mathbf{b}}\cdot\nabla_x w)w + \tilde{a}w^2\big)\,dz
  -\int_{B'_r} \tilde{h} |\mathop{\rm Tr}w|^2\, dx,
\end{equation*}
as a result of the coarea formula combined with
\eqref{normadiAlimitata}, \eqref{nuovaregolarita}, the Young
inequality and Lemma \ref{lemmahardytype} as for the first integral,
together with the fact that $\tilde{h}\in W^{1,p}(B'_{r_0})$,
$w^2\in L^{2^\ast_s/2}(B'_r)$ by Lemma \ref{lemmafallfelli} and the
H\"{o}lder inequality with regard to the second integral.  Hence
$E\in W^{1,1}_{\mathrm{loc}}(0,r_0]$ and
\begin{align}\label{toplugin}
    \quad E'(r)
  &= \,(2s-n)r^{2s-n-1}\biggl( \int_{B^+_r} t^{1-2s}
    \big(A\nabla w\cdot\nabla w + (\widetilde{\mathbf{b}}\cdot\nabla_x
    w)w
    + \tilde{a}w^2\big)\,dz\\
  \notag&\qquad\qquad\qquad\qquad\qquad  -\int_{B'_r} \tilde{h}
          |\mathop{\rm Tr}w|^2\, dx\biggr)\\
  \notag&\quad +\,r^{2s-n}\biggl(\int_{\partial^+ B^+_r}t^{1-2s}
          \big(A\nabla w\cdot\nabla w +
          (\widetilde{\mathbf{b}}\cdot\nabla_x w)w
          + \tilde{a}w^2\big)\,dS-\int_{S_r} \tilde{h} w^2\, dS'\biggr)
\end{align}
kMin a distributional sense and a.e. in $(0,r_0)$.

By Proposition \ref{proputile} we have
\begin{align*}
    r^{2s-n}&\biggl(\int_{\partial^+ B^+_r}t^{1-2s}A\nabla
              w\cdot\nabla w\,dS
              -\int_{S_r}\tilde{h}w^2\, dS'\biggr)\\
    &= 2r^{2s-n}\int_{\partial^+B^+_r}t^{1-2s} \frac{(A\nabla
      w\cdot\nu)^2}{\mu}\,dS
      -2 r^{2s-n-1}\int_{B^+_{r}}t^{1-2s}(\beta\cdot\nabla w)
      (\widetilde{\mathbf{b}}\cdot\nabla _x w+\tilde{a}w)\,dz\\
    &\quad + (n-2s)r^{2s-n-1}\int_{B^+_{r}}t^{1-2s}A\nabla w\cdot\nabla w\,dz\\
            &\quad +O(r^{\frac{2sp-n}{p}+2s-n-1})\left(\int_{B'_r}
              |\mathop{\rm Tr}w|^{2^\ast_s}\,dx\right)^{\frac{2}{2^\ast_s}}+
              O(r^{2s-n})\int_{B^+_r}t^{1-2s}|\nabla w|^2\,dz
\end{align*}
a.e. in $(0,r_0)$ as $r\to 0^+$. We now plug the last identity into
\eqref{toplugin} and observe that, similarly to \eqref{comb3},
\begin{equation*}
  -(2s-n)r^{2s-n-1}\int_{B'_r}\tilde{h}|\mathop{\rm Tr}w|^2\,dx
  =O(r^{\frac{2sp-n}{p}+2s-n-1})\left(\int_{B'_r}|\mathop{\rm Tr}
    w|^{2^\ast_s}\,dx
  \right)^{\frac{2}{2^\ast_s}}
  \end{equation*}
  as $r\to 0^+$, thus obtaining
  \begin{align}\label{Eì2}
            E'(r)&= 
        r^{2s-n}\left(2\int_{\partial^+B^+_r} t^{1-2s}\frac{(A\nabla
                   w\cdot\nu)^2}{\mu}\, dS+
                   \int_{\partial ^+
                   B^+_r}t^{1-2s}\big((\widetilde{\mathbf{b}}\cdot\nabla_xw)w
                   +\tilde{a}w^2\big)\, dS\right)\\
    \notag   &\quad +(2s-n)r^{2s-n-1}\left(\int_{B^+_r}t^{1-2s}
               \big((\widetilde{\mathbf{b}}\cdot\nabla_x w)w+\tilde{a}w^2\big)\,dz\right)\\
        \notag&\quad -2r^{2s-n-1}\int_{B^+_{r}}t^{1-2s}(\beta\cdot\nabla
                w)
                (\widetilde{\mathbf{b}}\cdot\nabla _x w+\tilde{a}w)\,dz \\
        \notag &\quad
                 +O(r^{\frac{2sp-n}{p}+2s-n-1})\left(\int_{B'_r}|\mathop{\rm
                 Tr}w|^{2^\ast_s}\,dx
                 \right)^{\frac{2}{2^\ast_s}}+O(r^{2s-n})\int_{B^+_r} t^{1-2s} |\nabla w|^2\,dz
  \end{align}
  a.e. in $(0,r_0)$ as $r\to 0^+$.
  To complete the proof, it remains only to show that the last four
  terms in \eqref{Eì2} are
  $O(r^{-1+\varepsilon})\left(E(r)+\frac{n-2s}{2}H(r)\right)$ as
  $r\to 0^+$ with $\varepsilon>0$ as in \eqref{vareps}.

  To this end, we observe that Lemma \ref{lemmastima1}, \eqref{E}, and \eqref{H} imply 
\begin{align}\label{importante}
    \int_{B^+_r} t^{1-2s} |\nabla w|^2\,dz&\leq
      \int_{B^+_r} t^{1-2s} |\nabla w|^2\,dz
      +\frac{n-2s}{2r}\int_{\partial^+
          B^+_r}
        t^{1-2s}\mu|v|^2\,dS\\
\notag &    \leq
    C_{n,s}r^{n-2s}\left(E(r)+\frac{n-2s}{2}H(r)\right),
  \end{align}
    and consequently, in view of Lemma \ref{lemmahardytype},
    \begin{equation}\label{importante1}
      \int_{B^+_r}t^{1-2s}\frac{|w|^2}{|z|^2}\,dz\leq
      \frac{16\, C_{n,s}}{(n-2s)^2}r^{n-2s}
        \left(E(r)+\frac{n-2s}{2}H(r)\right).
    \end{equation}
From \eqref{importante}, it immediately follows that 
\begin{equation}\label{merg1}
            r^{2s-n}\int_{B^+_r}t^{1-2s}|\nabla w|^2\,dz 
        =O(1)\left(E(r)+\frac{n-2s}{2}H(r)\right)\quad \text{as $r\to 0^+$}.
            \end{equation}
    Moreover, combining 
Lemma \ref{lemmafallfelli}, \eqref{1sumulimitata}, and
\eqref{importante} yields
    \begin{equation}\label{merg2}
            r^{\frac{2sp-n}{p}+2s-n-1}\left(\int_{B'_r} |\mathop{\rm
                Tr}w|^{2^\ast_s}\,dx
            \right)^{\frac{2}{2^\ast_s}}=O(r^{-1+\frac{2sp-n}{p}})\left(E(r)+
              \frac{n-2s}{2}H(r)\right) 
          \end{equation}
          as $r\to 0^+$.
    In addition, by \eqref{1}, \eqref{2}, and \eqref{importante}, we have  
\begin{equation}\label{claim}
    (2s-n)r^{2s-n-1} \biggl(\int_{B^+_r} t^{1-2s}
    \big((\widetilde{\mathbf{b}}\cdot \nabla _x w)w
    +\tilde{a}w^2\big)\,dz\biggr)
    =O(1)\left(E(r)+\frac{n-2s}{2}H(r)\right)
  \end{equation}
  as $r\to 0^+$.  Finally, we claim that
\begin{equation}\label{merg3}
    r^{2s-n-1}\int_{B^+_{r}}t^{1-2s}(\beta\cdot\nabla w)
    (\widetilde{\mathbf{b}}
    \cdot\nabla _x w+\tilde{a}w)\,dz =O(1)\left(E(r)+\frac{n-2s}{2}H(r)\right)
  \end{equation}
  as $r\to 0^+$. Indeed, from thanks to \eqref{nuovaregolarita},
  \eqref{lestimedibeta}, \eqref{importante}, \eqref{importante1}, and 
  Young's inequality, we have
  \begin{align*}
   \biggl|\int_{B^+_{r}}t^{1-2s}(\beta\cdot\nabla w)
  (\widetilde{\mathbf{b}}\cdot\nabla _x w
  +\tilde{a}w)\,dz\biggr|&\leq \mathrm{const}\cdot r
                           \left(\int_{B^+_r}t^{1-2s} |\nabla w|^2\,
                           dz +
                           \int_{B^+_r} t^{1-2s}\frac{|w|^2}{|z|^2}\,dz\right) \\
                         &\leq \mathrm{const}\cdot r^{n-2s+1}\left(E(r)+\frac{n-2s}{2}H(r)\right) 
  \end{align*}
  for some $\mathrm{const}>0$ (independent of $r$, and whose value may
  change from line to line), taking if necessary $r_0$ smaller from
  the beginning. Claim \eqref{merg3} is thereby proved.  By combining
  \eqref{Eì2}, \eqref{merg1}, \eqref{merg2}, \eqref{claim} and
  \eqref{merg3}, the proof of \eqref{E'} is complete.
\end{proof}

\begin{lemma}\label{lemmasuH}
  For some fixed $w\in H^1(B^+_{r_0}, t^{1-2s}dz)\setminus\{0\}$
  weakly solving \eqref{problemadiWtilde}, let $H:~\!\!(0,r_0)\to\R$
  be the function defined in \eqref{H}.  Then
  \begin{enumerate}
    \item[$\mathrm{(i)}$] $H(r)>0$ for every $r\in (0,r_0)$;
    \item[$\mathrm{(ii)}$]$H\in W^{1,1}_{\mathrm{loc}}(0,r_0]$ and,
      for a.e. $r\in (0,r_0)$,
    \begin{equation}\label{eqH'}
      H'(r)=2r^{2s-n-1}\int_{\partial^+ B^+_r} t^{1-2s}
      (A\nabla w\cdot\nu)w\,dS+O(H(r)) \quad \text{as $r\to 0^+$};
    \end{equation}
\item[$\mathrm{(iii)}$] for a.e. $r\in (0,r_0)$ 
        \begin{equation}\label{relazioneH'E}
            H'(r)=\frac{2}{r}E(r)+ O(H(r)) \quad \text{as $r\to 0^+$}.
        \end{equation}
    
\end{enumerate}
    \end{lemma}
    \begin{proof}
      To prove (i) we argue by contradiction assuming that
      $H(\bar{r})=0$ for some $\bar{r}\in (0,r_0)$. Due to
      \eqref{1sumulimitata}, this in particular
      yields 
\begin{equation}\label{nonecostante}
  w\equiv 0 \quad \text{on $\partial^+ B^+_{\bar{r}}$}.
\end{equation}
Hence, from \eqref{identity} and Lemma \ref{lemmastima1} we deduce
\begin{align*}
  0&= \int_{B^+_{\bar{r}}} t^{1-2s} \big(
     A\nabla  w\cdot\nabla  w+
     (\widetilde{\mathbf{b}}\cdot\nabla_x  w)w
     + \tilde{a}(x) w^2\big)\,dz-\int_{B'_{\bar{r}}} \tilde{h}
     |\mathop{\rm Tr}w|^2\, dx\\
   &\geq \frac{1}{C_{n,s}}\int_{B^+_{\bar{r}}} t^{1-2s} |\nabla w|^2\,dz.
\end{align*}
From this and \eqref{nonecostante}, it follows that $w\equiv 0$ on the
whole $B^+_{\bar{r}}$.  By invoking classical unique continuation
principles for second order elliptic operators with locally bounded
coefficients (see, e.g., \cite{Wolff}), we conclude that $w\equiv 0$
on $B^+_{r_0}\cap \{t>\delta\}$ for any $\delta>0$, so that
$w\equiv 0$ on $B^+_{r_0}$. This leads to the desired contradiction.

We now proceed to the proof of (ii). We observe that
$H\in L^1_{\mathrm{loc}}(0,r_0]$, being $H$ the product between the
$L^\infty_{\mathrm{loc}}(0,r_0]$-function $r\mapsto r^{2s-n-1}$ and
the function
$r\mapsto \int_{\partial^+ B^+_r} t^{1-2s}\mu(z) w^2\, dS$ which
belongs to $L^1(0,r_0)$ by \eqref{sviluppomu}. In order to prove
\eqref{eqH'}, we first show that, in a distributional sense and
a.e. in $(0,r_0)$,
\begin{equation}\label{auxH'}
  H'(r)= 2r^{2s-n-1}\int_{\partial ^+ B^+_r} t^{1-2s} \mu
  w\frac{\partial w}{\partial\nu}\,dS +O(H(r))\quad \text{as $r\to 0^+$}.
\end{equation}
Indeed, for every $\varphi\in C^\infty_c(0,r_0)$, by the Divergence Theorem
\begin{align*}
    \int_0^{r_0} H(r) \varphi'(r)\,dr
    &    =\int_{B^+_{r_0}}  t^{1-2s}|z|^{2s-n-1}\mu(z) w^2(z) \varphi'(|z|)\,dz\\
    &=  \int_{B^+_{r_0}}  t^{1-2s}|z|^{2s-n-2}\mu(z) w^2(z) \nabla
      \tilde{\varphi}
      (z)\cdot z\,dz\\
    &= -\int_{B^+_{r_0}}
      t^{1-2s}\tilde{\varphi}(z)|z|^{2s-n-2}(2\mu(z)w(z)\nabla
      w(z)+w^2(z)\nabla \mu(z)) \cdot z\,dz\\
    &=- \int_0^{r_0} r^{2s-n-1}\left(\int_{\partial^+
      B^+_r}t^{1-2s}\big(2\mu w
      \partial_\nu w + w^2\nabla\mu\cdot\nu\big)\,dS\right)\varphi(r)\,dr,
\end{align*}
where $\tilde{\varphi}(z):=\varphi(|z|)$. We conclude that 
\begin{equation}\label{aux2}
  H'(r)= r^{2s-n-1}\int_{\partial^+ B^+_r}t^{1-2s}
  \big(2\mu w\partial_\nu w + w^2\nabla\mu\cdot\nu\big)\,dS
\end{equation}
in a distributional sense. Since
$w\in H^1(B^+_r,t^{1-2s}dz)$, by \eqref{sviluppomu} we have
\begin{equation*}
  r\mapsto \int_{\partial^+ B^+_r}t^{1-2s}\big(2\mu w \partial_\nu w
  + w^2\nabla\mu\cdot\nu\big)\,dS
\end{equation*}
is a $L^1(0,{r_0})$-function, so that
$H\in W_{\mathrm{loc}}^{1,1}(0,{r_0}]$ as a consequence of
\eqref{aux2}. In particular, \eqref{aux2} also holds a.e. in
$(0,{r_0})$. From \eqref{sviluppomu} it follows that
\begin{equation*}
  r^{2s-n-1}\int_{\partial^+ B^+_r} t^{1-2s}  w^2\nabla
  \mu\cdot\nu\,dS=O(H(r))
  \quad \text{as $r\to 0^+$},
\end{equation*}
thus implying \eqref{auxH'} in a distributional sense and a.e. in
$(0,{r_0})$ in view of \eqref{aux2}.
  
To prove \eqref{eqH'}, we first observe that, by \eqref{beta},
\begin{multline}\label{secondaformula}
    2r^{2s-n-1}\int_{\partial^+ B^+_r} t^{1-2s}\mu w\frac{\partial w}{\partial\nu} \,dS 
    = 2r^{2s-n-1}\int_{\partial^+ B^+_r} t^{1-2s}\mu w\nabla w
      \cdot\frac{z}{|z|} \,dS \\
    =2r^{2s-n-1} \int_{\partial^+ B^+_r} t^{1-2s} (A\nabla 
    w\cdot\nu)w\,dS-r^{2s-n-1} \int_{\partial^+ B^+_r} t^{1-2s}
    \mu \nabla (w^2)\cdot\frac{\beta-z}{|z|}\,dS,
  \end{multline}
for a.e. $r\in (0,{r_0})$. 

In particular \eqref{mu} and \eqref{beta} imply that
$ (\beta-z)\cdot z=0$. Thus, by applying the Divergence Theorem
followed by the coarea formula, we
deduce that, for a.e.  $r\in (0,{r_0})$,
\begin{equation}\label{pluginfo}
  0= \int_{\partial^+ B^+_r} \mathop{\rm div}\left(t^{1-2s}
    \frac{\mu(\beta-z)}{|z|}\right)w^2\,dS +
  \int_{\partial^+ B^+_r}  t^{1-2s} \mu \nabla (w^2)\cdot \frac{\beta-z}{|z|} \,dS. 
\end{equation}
We observe that
\begin{align*}
  \mathop{\rm div}\left(t^{1-2s} \frac{\mu(\beta-z)}{|z|}\right)
  &= (1-2s) \frac{\zeta-\mu}{|z|}t^{1-2s} +
    t^{1-2s}\nabla \mu\cdot \frac{\beta -z}{|z|}\\
  &\quad  + t^{1-2s}\mu\frac{ \mathop{\rm div}(\beta-z)}{|z|}
    - t^{1-2s}\mu (\beta-z)\cdot\frac{z}{|z|^3},
\end{align*}
so that
$\mathop{\rm div}(t^{1-2s} \frac{\mu(\beta-z)}{|z|}) = t^{1-2s}O(1)$
as $|z|\to 0$, in view of \eqref{stimadizeta}, \eqref{sviluppomu}, and
\eqref{lestimedibeta}. Inserting this into \eqref{pluginfo}, we obtain
\begin{equation}\label{toplug2}
  \int_{\partial^+ B^+_r}  t^{1-2s} \mu \nabla (w^2)\cdot
    \frac{\beta-z}{|z|}\,dS= r^{n+1-2s} O(H(r)) \quad \text{as $r\to 0^+$}. 
\end{equation}
By combining \eqref{toplug2} and \eqref{secondaformula} with
\eqref{auxH'}, we obtain \eqref{eqH'}.

Finally, by \eqref{identity} and \eqref{eqH'} 
        \begin{equation*}
          \int_{B^+_r} t^{1-2s} \big(A\nabla w\cdot\nabla w+
          (\widetilde{\mathbf{b}}\cdot\nabla_x w)w +\tilde{a}w^2\big)\,dz 
            -\int_{B'_r} \tilde{h}|\mathop{\rm Tr} w|^2\,
            dx=\frac{H'(r)}{2r^{2s-n-1}}
            + O\left(\frac{H(r)}{r^{2s-n-1}}\right)
        \end{equation*}
        as $r\to 0^+$. The above estimate directly yields
        \eqref{relazioneH'E}, because the left-hand side is equal to
        $r^{n-2s}E(r)$, see \eqref{E}.
    \end{proof}
Thanks to Lemma \ref{lemmasuH}, the Almgren frequency function 
\begin{equation}\label{N}
\mathcal N:(0,r_0)\to\R,\quad \mathcal{N}(r):= \frac{E(r)}{H(r)},
\end{equation}
is well-defined. Furthermore,
$\mathcal{N}\in W^{1,1}_{\mathrm{loc}}(0,r_0]$ as a consequence of
Lemma \ref{lemmaE'} and Lemma \ref{lemmasuH}.  In the following lemma,
we prove that $\mathcal{N}$ is bounded from below.
\begin{lemma}
    For every $r\in (0,r_0)$
\begin{equation}\label{Nmaggiore}
       \mathcal{N}(r)+\frac{n-2s}{2}\geq \frac{n-2s}{2 C_{n,s}}>0.
    \end{equation}
\end{lemma}
\begin{proof}
From Lemma \ref{lemmastima1}, the definitions of $E$ and $H$ (see \eqref{E} and \eqref{H},
respectively), and Lemma \ref{lemmasuH}, it follows that 
  \begin{equation}
    \label{eq:stH-E+H}
   0< H(r)\leq \frac{2
      C_{n,s}}{n-2s}\left(E(r)+\frac{n-2s}2H(r)\right)\quad\text{for
      every }r\in(0,r_0).
  \end{equation}
  Hence,  the
    definition of $\mathcal{N}$ given in \eqref{N} directly implies \eqref{Nmaggiore}.
\end{proof}

Let us consider now the set
\begin{equation*}
    \Xi := \{r\in (0,r_0) : \mathcal{N}'(r)\leq 0\}.
\end{equation*}
\begin{lemma}\label{lemmaN'}
If $0$ is a limit point of $\Xi$, then we have, for a.e. $r\in (0,r_0)$,
    \begin{equation}\label{N'maggiore}
  \mathcal{N}'(r)\geq O(r^{-1+\varepsilon})\left(\mathcal{N}(r)+\frac{N-2s}{2}\right)
    \end{equation}
    as $r\to 0^+$, being $\varepsilon>0$  defined as in \eqref{vareps}.
\end{lemma}
\begin{proof} 
  We observe that \eqref{N'maggiore} is trivially satisfied a.e. in
  $(0,r_0)\setminus\Xi$, since $\mathcal{N}'\geq0$ there.  Hence, we
  just need to prove \eqref{N'maggiore} a.e. in $\Xi$.  To this aim,
  we preliminarily show that
\begin{equation}\label{E'nuova}
  E'(r) = 2r^{2s-n} \int_{\partial^+ B^+_r}t^{1-2s}
  \frac{(A\nabla w\cdot\nu)^2}{\mu}\,dS +
  O(r^{-1+\varepsilon})\left(E(r)+\frac{n-2s}{2}H(r)\right)
\end{equation}
as $r\to 0^+$, $r\in \Xi$.
In light of \eqref{E'}, proving \eqref{E'nuova} reduces to showing that
\begin{equation}\label{claimdrift}
  r^{2s-n}\int_{\partial^+ B^+_r}t^{1-2s} \big(
  (\widetilde{\mathbf{b}}\cdot\nabla_x w)w+\tilde{a} w^2\big)\,dS
  =O(1) \left(E(r)+\frac{n-2s}{2}H(r)\right)
\end{equation}
as $r\to 0^+$, $r\in \Xi$. We define
\begin{equation*}
  D(r):= 2r^{2s-n} \int_{\partial^+ B^+_r} t^{1-2s}
  \frac{(A\nabla w\cdot\nu)^2}{\mu}\,dS,
    \end{equation*}
    so that, by \eqref{useful},  
    \begin{align*}
         r^{2s-n+1} &\int_{\partial^+ B^+_r} t^{1-2s} A\nabla w\cdot\nabla w\,dS \\
          =& r^{2s-n+1}\int_{S_r} \tilde{ h} |\mathop{\rm Tr}w|^2\,dS'
          +rD(r)-2r^{2s-n}\int_{B^+_r} t^{1-2s} (\beta\cdot\nabla w)
          (\widetilde{\mathbf{b}}\cdot\nabla_x w +\tilde{a}w)\,dz\\
          &+(n-2s) r^{2s-n}\int_{B^+_r}t^{1-2s}A\nabla w\cdot \nabla
          w\, dz + O(r^{\frac{2sp-n}{p}+2s-n})\left(\int_{B'_r}
            |\mathop{\rm Tr}w|^{2^\ast_s}dx\right)^{\frac{2}{2^\ast_s}}\\
          &+ O(r^{2s-n+1}) \int_{B^+_r} t^{1-2s}|\nabla w|^2\,dz
    \end{align*}
    as $r\to0^+$.  In view of \eqref{normadiAlimitata}, \eqref{merg1},
    \eqref{merg2} and \eqref{merg3}, this implies 
\begin{multline}\label{assembling1}
  r^{2s-n+1} \int_{\partial^+ B^+_r} t^{1-2s} A\nabla w\cdot\nabla
  w\,dS \\ = r^{2s-n+1}\int_{S_r} \tilde{ h} |\mathop{\rm
    Tr}w|^2\,dS'+rD(r) + O(1) \left(E(r)+\frac{n-2s}{2}H(r)\right),
\end{multline}
as $r\to0^+$.
By a scaling argument and classical trace theorems, we have that, for every $r\in(0,r_0)$,
  \begin{equation*}
    \|\tilde h\|_{L^p(S_r)}\leq \mathop{\rm const}r^{-1/p}\|\tilde h\|_{W^{1,p}(B'_{r_0})}
  \end{equation*}
  for some $\mathop{\rm const}>0$
  independent of $r$. Hence, by the H\"{o}lder inequality,
\eqref{eq:stgr}, and Lemma \ref{lemmasobolev}, we have, for
a.e. $r\in(0,r_0)$, 
\begin{align}\label{assembling2}
 &\quad r^{2s-n+1}\int_{S_r} \tilde{ h} |\mathop{\rm Tr}w|^2\,dS' \\
 \notag &\leq  \big(n\omega_{n}\big)^{\frac{2sp-n}{np}} r^{2s-n+1+
         \frac{(2sp-n)(n-1)}{pn}} \Vert \tilde{h}\Vert_{L^p(S_r)}
         \left(\int_{S_r}|\mathop{\rm Tr}w|^{2^\ast_s}\,ds'\right)^{\frac{2}{2^\ast_s}}\\
  \notag    &  \leq  \big(n\omega_{n}\big)^{\frac{2sp-n}{np}}
               \bar{S}_{n,s}r^{2s-n+1
               +\frac{(2sp-n)(n-1)}{pn}+\frac{2s}{n}}\Vert
               \tilde{h}\Vert_{L^p(S_r)}
               \int_{\partial^+ B^+_r} t^{1-2s}(r^{-2}|w|^2+ |\nabla w|^2)\,dS\\
  \notag  &     \leq  \,\mathop
  {\rm const}\bigg(r^{\frac{2sp-n}{p}}
               H(r)+
               r^{\frac{2sp-n}{p}+2s-n+1}
               \int_{\partial^+ B^+_r}t^{1-2s} A\nabla w\cdot\nabla w\,dS\bigg),
\end{align}
for some $\mathrm{const}>0$, where we also  used
\eqref{1sumulimitata} and  \eqref{Amaggiore}, being $\omega_{n}$ the
$n$-dimensional Lebesgue measure of $B_1'$. Assembling
\eqref{assembling1} with \eqref{assembling2}, and if necessary
choosing $r_0$ smaller from the beginning in order to ensure that
\begin{equation*}
  1-\mathrm{const}\, r^{\frac{2sp-n}{p}}\geq \frac{1}{2}\quad \text{for every $r\in (0,r_0)$},
\end{equation*}
we obtain, for a.e. $r\in(0,r_0)$,
\begin{equation}\label{stimaterminedibordoconA}
  \frac{r^{2s-n+1}}{2}\int_{\partial^+ B^+_r}t^{1-2s}A\nabla w\cdot\nabla w\,dS
  \leq rD(r) + O(1)\left(E(r)+\frac{n-2s}{2}H(r)\right)
\end{equation}
as $r\to 0^+$, where also estimate \eqref{eq:stH-E+H} was used.
From \eqref{nuovaregolarita}, \eqref{normadiAlimitata},
\eqref{1sumulimitata}, \eqref{eq:stH-E+H}, and 
\eqref{stimaterminedibordoconA}, it

therefore follows that
    \begin{align}\label{terminedidrift}
      r^{2s-n}&\int_{\partial^+ B^+_r}t^{1-2s} \big(
                (\widetilde{\mathbf{b}}\cdot\nabla_x w)w+\tilde{a}(x) w^2\big)\,dS \\
      \notag  = & \,O(1)\left(r^{2s-n+1}\int_{\partial^+B^+_r}
                  t^{1-2s}|\nabla w|^2\,dS\right)^{\frac{1}{2}}\sqrt{H(r)}+O(r) H(r)\\
      \notag = &\,O(1)\left(r^{2s-n+1}\int_{\partial^+B^+_r} t^{1-2s}
                 A\nabla w\cdot\nabla w\,dS\right)^{\frac{1}{2}}\sqrt{H(r)}+O(r) H(r)\\
    \notag    = &
                  \,O(1)\left(\sqrt{rD(r)H(r)}+E(r)+\frac{n-2s}{2}H(r)\right)
    \end{align}
    as $r\to0^+$.  We insert estimate \eqref{terminedidrift} into
    \eqref{E'}, thus obtaining, also in view of \eqref{eq:stH-E+H},
\begin{align*}
  D(r)
  &=E'(r) +O(1)\sqrt{rD(r)H(r)}+O(r^{-1+\varepsilon})\left(E(r)+\frac{n-2s}{2}H(r)\right)\\
  &\leq \, E'(r) +\frac{D(r)}{2}+O(r^{-1+\varepsilon})\left(E(r)+\frac{n-2s}{2}H(r)\right),
\end{align*}
so that
\begin{equation}\label{Dminore}
    D(r)\leq 2E'(r) + O(r^{-1+\varepsilon})\left(E(r)+\frac{n-2s}{2}H(r)\right)
\end{equation}
as $r\to0^+$. Since $E'(r)H(r)\leq E(r)H'(r)$ for every
$r\in\Xi$, from \eqref{eq:stH-E+H}, \eqref{relazioneH'E}, and
\eqref{Dminore}, it follows that, as $r\to 0^+$, $r\in \Xi$,
\begin{align*}
  rD(r)H(r) \leq& \,2r E'(r)H(r) + O(r^{\varepsilon})
                  \left(E(r)+\frac{n-2s}{2}H(r)\right)^2\\
  \leq &\, 2r E(r)H'(r)+ O(r^{\varepsilon}) \left(E(r)+\frac{n-2s}{2}
         H(r)\right)^2\\
  =&\, 2r E(r)\left(\frac{2}{r}E(r)+O(H(r))\right)+ O(r^{\varepsilon})
     \left(E(r)+\frac{n-2s}{2}H(r)\right)^2\\
  =& \,4 E(r)^2 + O(r)E(r)H(r)+ O(r^{\varepsilon}) \left(E(r)+\frac{n-2s}{2}H(r)\right)^2\\
  = &\, O(1)\left(E(r)+\frac{n-2s}{2}H(r)\right)^2
\end{align*}
which, in view of \eqref{eq:stH-E+H}, yields 
\begin{equation*}
  \sqrt{rD(r)H(r)} =O(1)\left(E(r)+\frac{n-2s}{2}H(r)\right)
  \quad \text{as $r\to 0^+$, $r\in \Xi$}.
\end{equation*}
Substituting this into \eqref{terminedidrift} yields the desired
estimate \eqref{claimdrift} and, consequently, \eqref{E'nuova}.

We are now ready to prove \eqref{N'maggiore}.  By direct computations
and thanks to \eqref{H}, \eqref{eqH'}, \eqref{relazioneH'E},
\eqref{N}, and \eqref{E'nuova}, we have, as $r\to 0^+$, $r\in \Xi$,
    \begin{multline}\label{N'prima}
        \mathcal{N}'(r)= \frac{E'(r)H(r)-\frac{r}{2}(H'(r))^2+O(r) H(r)H'(r)}{H^2(r)}\\
        = 2r^{4s-2n-1}\frac{\left(\int_{\partial^+ B^+_r} t^{1-2s}
            \frac{(A\nabla
              w\cdot\nu)^2}{\mu}\,dS\right)\left(\int_{\partial^+
              B^+_r}
            t^{1-2s} \mu w^2\,dS\right)-\left(\int_{\partial^+ B^+_r}
            t^{1-2s}(A\nabla w\cdot\nu)w\,dS\right)^2}{H^2(r)}\\
        +O(r) + \frac{O(r^{2s-n})}{H(r)}\int_{\partial^+ B^+_r}
        t^{1-2s}
        (A\nabla w\cdot \nu)w\,dS + O(r^{-1+\varepsilon}) \left(\mathcal{N}(r)+\frac{n-2s}{2}\right).
      \end{multline}
   The Cauchy-Schwarz inequality applied to vectors in $L^2(\partial^+
   B^+_r, t^{1-2s}dS)$ implies that
    \begin{multline}\label{C-S}
      \left(\int_{\partial^+ B^+_r} t^{1-2s}\frac{(A\nabla
          w\cdot\nu)^2}{\mu}\,dS\right)
      \left(\int_{\partial^+ B^+_r} t^{1-2s} \mu w^2\,dS\right)\\
      -\left(\int_{\partial^+ B^+_r} t^{1-2s}(A\nabla
        w\cdot\nu)w\,dS\right)^2 \geq 0
      \end{multline}
      for a.e. $r\in (0,r_0)$; furthermore, as a consequence of
      \eqref{eqH'} and \eqref{relazioneH'E}, we have
    \begin{equation}\label{necess}
        \frac{r^{2s-n}}{H(r)}\int_{\partial^+ B^+_r} t^{1-2s}(A\nabla w\cdot \nu)w\,dS= \mathcal{N}(r) + O(r)\quad \text{as $r\to 0^+$}. 
    \end{equation} 
    Combining \eqref{N'prima} with \eqref{C-S} and \eqref{necess}, and
    recalling \eqref{Nmaggiore}, we finally obtain
    \eqref{N'maggiore}.~\end{proof}

\begin{lemma}\label{l:twocases}
There exists a positive constant $C_3>0$ such that 
    \begin{equation}\label{Nmindicost}
        \mathcal{N}(r) \leq C_3\quad \text{for every $r\in (0,r_0)$}.
    \end{equation}
    Moreover,
    \begin{equation}\label{Nammettelimite}
      \text{the limit of
        $\mathcal{N}$ as $r\to 0^+$ exists and is finite}.
  \end{equation}
\end{lemma}
\begin{proof}
  To establish \eqref{Nmindicost} and \eqref{Nammettelimite}, we
  distinguish two cases.

  \emph{Case 1.} If $0$ is not a limit point of $\Xi$ then
  $\mathcal{N}'(r)\geq 0$ in a right neighbourhood of 0, where
    $\mathcal N$ is therefore increasing and bounded from above. The
    continuity of $\mathcal N$ in $(0,r_0]$ then implies its boundedness
    on the whole interval $(0,r_0)$, thus yielding
    \eqref{Nmindicost}.  Moreover, the monotonicity of $\mathcal{N}$
  in a right neighbourhood of $0$ and \eqref{Nmaggiore} imply
  \eqref{Nammettelimite}.

  \emph{Case 2.} If $0$ is a limit point of
  $\Xi$, then we are under assumptions of Lemma \ref{lemmaN'} and thus
    \begin{equation}\label{dividoN}
      \left(\mathcal{N}(\tau)+\frac{n-2s}{2}\right)'
      \geq - c \tau^{-1+\varepsilon}\left(\mathcal{N}(\tau)+\frac{n-2s}{2}\right)
    \end{equation}
    for some positive constant $c>0$ and for every $\tau$ in a
      sufficiently small right neighbourhood of $0$. Estimate 
      \eqref{dividoN}
      directly implies that the function
  \begin{equation*}
        r\mapsto e^{\frac{cr^\e}\e}\left(\mathcal N(r)+\frac{n-2s}2\right)
      \end{equation*}
      is increasing in a right neighbourhood of 0. Since it is
      positive by \eqref{Nmaggiore}, it has a finite limit as $r\to0^+$. This directly
      implies \eqref{Nammettelimite}; hence, by the continuity of
      $\mathcal N$ in $(0,r_0]$, \eqref{Nmindicost} also holds.
\end{proof}

Henceforth, we adopt the following notation:
\begin{equation}\label{zeta}
   \gamma:= \lim_{r\to 0^+}\mathcal{N}(r). 
\end{equation}
We observe that
  \begin{equation}
    \label{eq:gamma-ge}
    \gamma\geq -\frac{n-2s}2
  \end{equation}
  in view of \eqref{Nmaggiore}.

As a consequence of \eqref{Nmindicost} and the relation between $H'$
and $E$ provided in Lemma \ref{lemmasuH}, the following results
follow.

\begin{lemma}
    There exists $K_1>0$ such that 
    \begin{equation}\label{Hminore}
        H(r)\leq K_1 r^{2\gamma}\quad \text{for all $r\in (0,r_0)$}.
    \end{equation}
    For every fixed $\delta>0$ there exists
    $K_2=K_2(\delta)>0$ such that
    \begin{equation}\label{Hmag}
        H(r)\geq K_2 r^{2\gamma+\delta}\quad \text{for all $r\in (0,r_0)$}.
    \end{equation}
\end{lemma}
\begin{proof}
  To prove \eqref{Hminore}, we first observe that there exists some
  $r_1\in(0,r_0)$ such that, for all
  $\rho\in (0,r_1)$,
    \begin{equation*}
      \mathcal{N}(\rho)-\gamma=
      \int_0^\rho \mathcal{N}'(\tau)\,d\tau \geq -
      c \left(C_3+\frac{n-2s}{2}\right) \frac{\rho^{\varepsilon}}{\varepsilon},
    \end{equation*}
    where we have used $\mathcal{N}\in W^{1,1}_{\mathrm{loc}}(0,r_0]$
    along with \eqref{Nammettelimite} and \eqref{zeta} to obtain the
    first identity, and \eqref{Nmindicost} and \eqref{dividoN} to
    derive the second inequality (we note that \eqref{dividoN}
      trivially holds also in Case 1 of Lemma \ref{l:twocases}).
      Combined with \eqref{relazioneH'E} and the strict positivity of
    $H$ (Lemma \ref{lemmasuH}-(i)), this yields
    \begin{equation*}
      \frac{H'(\rho)}{H(\rho)}= \frac{2\mathcal{N}(\rho)}{\rho}
      +O(1) \geq 2\gamma\rho^{-1}-2c
      \left(C_3+\frac{n-2s}{2}\right) \frac{\rho^{-1+\varepsilon}}{\varepsilon}-c_1
    \end{equation*}
    for some $c_1>0$ and for all $\rho\in (0,r_1)$ (possibly
      after taking a smaller $r_1$).
     Integrating the above estimate
      between $r$ and $r_1$ and recalling that $H$ is continuous on
      $(0,r_0]$, we directly obtain \eqref{Hminore}.

    To prove \eqref{Hmag}, we note from \eqref{Nammettelimite} and
    \eqref{zeta} that, for each fixed $\delta>0$,
    \begin{equation*}
      \mathcal{N}(\rho) -\gamma\leq \frac{\delta}{2}\quad
      \text{for all $\rho\in (0,\rho_\delta)$},
    \end{equation*}
    for some $\rho_\delta>0$. In view of \eqref{relazioneH'E} and
    the strict positivity of $H$ (Lemma \ref{lemmasuH}-(i)), this
    implies that there exist
    some $c_2>0$ and $r_2\leq \rho_\delta$ such that
    \begin{equation*}
      \frac{H'(\rho)}{H(\rho)}=
      \frac{2\mathcal{N}(\rho)}{\rho}+O(1) \leq \frac{2\gamma+\delta}{\rho}+c_2
    \end{equation*}
    for all $\rho\in (0,r_2)$.
    An integration with respect to
      $\rho$ of the above inequality over $(r,r_2)$ and the continuity
      of $H$ yield \eqref{Hmag}.
\end{proof}

\begin{lemma}\label{limiteHfinito}
    The limit $\lim_{r\to 0^+}H(r)r^{-2\gamma}$ exists and is finite. 
\end{lemma}
\begin{proof}
  By \eqref{Hminore}, the proof of the lemma reduces to establishing
  the existence of the limit. By \eqref{relazioneH'E} and
  \eqref{N}, for a.e. $\rho\in (0,r_0)$ we have 
    \begin{align*}
      \frac{d}{d\rho}\left(H(\rho)\rho^{-2\gamma}\right)
      &= \,2E(\rho)\rho^{-2\gamma-1}+O(H(\rho))\rho^{-2\gamma}
        -2\gamma H(\rho) \rho^{-2\gamma-1} \\
        &=2H(\rho)\rho^{-2\gamma-1} ( \mathcal{N}(\rho)+O(\rho)-\gamma)\\
        &= 2H(\rho)\rho^{-2\gamma-1} \int_0^\rho
          \left(\mathcal{N}'(\tau)
          + c \tau^{-1+\varepsilon}\left(C_3+\tfrac{n-2s}{2}\right)\right)\,d\tau\\
      &-2cH(\rho)\rho^{-2\gamma-1}\int_0^\rho \tau^{-1+\varepsilon}
        \left(C_3+\tfrac{n-2s}{2}\right)\,d\tau + 2 H(\rho)\rho^{-2\gamma}O(1)
    \end{align*}
    as $\rho\to 0^+$. Letting  $\sigma\in (0,r_0)$, we
      integrate with respect to $\rho$ between $r\in(0,\sigma)$ and
      $\sigma$, thus obtaining
    \begin{multline}\label{precedente}
           \frac{H(\sigma)}{ \sigma^{2\gamma}}-\frac{H(r)}{r^{2\gamma
             }}
        = 2\int_r^{\sigma} 
        H(\rho)\rho^{-2\gamma-1}
G(\rho)\,d\rho\\
        -2c\,\varepsilon^{-1}\left(C_3+\tfrac{n-2s}{2}\right)\int_r^{\sigma}
        H(\rho)\rho^{-2\gamma-1+\varepsilon}\,d\rho+
        2\int_r^{\sigma} H(\rho)\rho^{-2\gamma}O(1)\,d\rho.
      \end{multline}
      where
      $G(\rho):= \int_0^\rho \left(\mathcal{N}'(\tau)+ c
        \tau^{-1+\varepsilon}\left(C_3+\tfrac{n-2s}{2}\right)\right)
      \,d\tau$.  From \eqref{dividoN} and \eqref{Nmindicost} it
      follows that, if $\sigma$ is chosen sufficiently small,
      $G(\rho)\geq0$ for every $\rho\in(0,\sigma)$. Hence the limit as
      $r\to 0^+$ of the first integral on the right-hand side of
      \eqref{precedente} exists. Moreover, in view of \eqref{Hminore},
      the functions $\rho\mapsto H(\rho)\rho^{-2\gamma-1+\varepsilon}$
      and $\rho\mapsto H(\rho)\rho^{-2\gamma}O(1)$ are integrable in
      $(0,\sigma)$, so that both the second and third integrals have
      finite limits as $r\to 0^+$. Hence, by \eqref{precedente}, we
      can conclude that the limit of $H(r)r^{-2\gamma}$ as $r\to 0^+$
      exists, as desired.
\end{proof}

\section{Neumann eigenvalues on the half-sphere under a symmetry condition}\label{sect3}
In this section we are interested in characterizing the eigenvalues of
the spherical eigenvalue problem
\begin{equation}\label{prob-eigenvalues}
\begin{cases}
  -\mathop{\rm{div}_{\mathbb{S}}}(\theta_{n+1}^{1-2s}\nabla_{\mathbb{S}}Y)=
  \mu \,\theta_{n+1}^{1-2s} \,Y &\text{in } \mathbb{S}^+,\\[5pt]
  \lim_{\theta_{n+1} \to 0^+}\theta_{n+1}^{1-2s}\,\nabla_{\mathbb{S}}
  Y\cdot\mathbf e=0
  &\text{on } \mathbb{S}',\\[5pt]
  Y\in H_{\rm even}^1(\mathbb{S}^+,\theta_{n+1}^{1-2s}dS),
\end{cases}	
\end{equation}
where $\mathbf e=(0,\dots, 0,-1)$ is the outer normal vector to
$\mathbb{S}^+$ on $\mathbb{S}'$ and
\begin{equation*}
  H_{\rm even}^1(\mathbb{S}^+,\theta_{n+1}^{1-2s}dS):=\\
  \{\Psi \in
  H^1(\mathbb{S}^+,\theta_{n+1}^{1-2s}dS):\Psi(\theta',\theta_n,\theta_{n+1})
  =\Psi(\theta',-\theta_n,\theta_{n+1})\},
\end{equation*}
with $H^1(\mathbb{S}^+,\theta_{n+1}^{1-2s}dS)$ denoting the space
introduced in Definition \ref{spaziosulbordo} (for $r=1$). It is easy
to verify that $H_{\rm even}^1(\mathbb{S}^+,\theta_{n+1}^{1-2s}dS)$ is
a closed subspace of $ H^1(\mathbb{S}^+,\theta_{n+1}^{1-2s}dS)$.

\begin{definition}\label{defdiautofunzione}
  We say that a function
  $Y \in H_{\rm even}^1(\mathbb{S}^+,\theta_{n+1}^{1-2s}dS)$ is an
  eigenfunction of \eqref{prob-eigenvalues} if $Y\not\equiv0$ and
  there exists $\mu\in\R$ (called eigenvalue) such that, for all
  $\Psi \in H_{\rm even}^1(\mathbb{S}^+,\theta_{n+1}^{1-2s}dS)$,
\begin{equation}\label{eq-egienvlulues}
  \int_{\mathbb{S}^+} \theta_{n+1}^{1-2s}\, \nabla_{\mathbb{S}}Y \cdot
  \nabla_{\mathbb{S}} \Psi \, dS
  = \mu \int_{\mathbb{S}^+} \theta_{n+1}^{1-2s}  Y \Psi \, dS.
\end{equation}
\end{definition}
By classical spectral theory, the eigenvalues of problem
\eqref{prob-eigenvalues} form an increasing sequence of nonnegative
real numbers $\{\mu_m\}_{m\in\mathbb{N}}$ that diverges to
$+\infty$. We explicitly characterize the sequence
$\{\mu_m\}_{m\in\mathbb{N}}$, showing that, for all
$m \in \mathbb{N}$,
\begin{equation}\label{eigenvalues}
\mu_m=
\begin{cases}
m(m+n-2s), &\text{if } n\geq2, \\
2m(2m+n-2s), &\text{if } n=1.
\end{cases}
\end{equation}

\begin{proposition}\label{p:eigenvalues}
  The following statements hold true.
  \begin{itemize}
\item[\rm (i)]   All eigenvalues of problem \eqref{prob-eigenvalues} are given by
  \eqref{eigenvalues}.
\item[\rm (ii)]  If $Y$ is an eigenfunction associated to the
    eigenvalue $\mu_m$, then there exists a homogeneous polynomial
    $P:\R^{n+1}\to\R$ of degree $m$, such that
    \begin{itemize}
      \item[\rm a)] $P$ is even with
    respect to both variables $x_n$ and $t$,
    i.e.
    \begin{equation*}
P(x',x_n,t)=P(x',-x_n,t)= P(x',x_n,-t)
\end{equation*}
for every $(x',x_n,t)\in \R^{n+1}$;
    \item[\rm b)] $P(z)=P(x',x_n,t)=|z|^m Y\big(\tfrac{z}{|z|}\big)$ for
      every $z\in \R^n\times(0,+\infty)$;
    \item[\rm c)] $P$ solves
      \begin{equation}\label{eq:hom}
\mathop{\rm div}(|t|^{1-2s}\nabla P)=0 \quad\text{in }\R^{n+1}.  
      \end{equation}
    \end{itemize}
\end{itemize}
\end{proposition}
\begin{proof}
We assume that $\mu$ is an eigenvalue of \eqref{prob-eigenvalues}. 
Hence there exists an eigenfunction $Y$ of \eqref{prob-eigenvalues}
associated with $\mu$, that is a nontrivial solution of
\eqref{prob-eigenvalues} in the sense specified in Definition
\ref{defdiautofunzione}.
A straightforward computation shows that this is equivalent to
  the function 
\begin{equation}\label{eq:U}
U(z):=|z|^\gamma Y\left(\frac{z}{|z|}\right), \quad z \in \R^{n+1}, 
\end{equation}
with 
\begin{equation}\label{zetaesplicito}
\gamma:=-\frac{n-2s}{2}+\sqrt{\left(\frac{n-2s}2{}\right)^2 +\mu},
\end{equation}
belonging to $H^1_{\mathrm{loc}}(\overline{\R_+^{n+1}},t^{1-2s}dz)$, being even
with respect to $x_n$, and weakly solving
\begin{equation}\label{prob-div-0}
\begin{cases}
\mathop{\rm div}(t^{1-2s}\nabla U)=0 &\text{in }\R_+^{n+1},\\
\lim_{t \to 0^+}t^{1-2s} \partial_tU=0 &\text{on } \R^n.
\end{cases}	
\end{equation}
Therefore, there exists a nontrivial solution $U$ of
\eqref{prob-div-0}, which is even with respect to $x_n$ and
homogeneous of degree $\gamma$; in addition
$U\in C^{\infty}(\overline{B_1^+})$ as a consequence of the regularity
result in \cite[Theorem 1.1]{SirTerVit21a}.  Then there exists
$m \in \mathbb{N}$ such that $\gamma=m$; thus we have that
$\mu=m(m+n-2s)$ by virtue of \eqref{zetaesplicito}. We notice that the
case $m=0$ is included here since, in that case, $\mu=0$ and 0 is an
eigenvalue with associated constant (and nontrivial) eigenfunctions.

If $n=1$, we may also prove that the odd integer homogeneities
are not allowed; that is, if $m\in \mathbb{N}$, then
$(2m+1)(2m+1+n-2s)$ is not an eigenvalue of \eqref{prob-eigenvalues}.
To this goal, we argue by contradiction assuming that
$(2m+1)(2m+1+n-2s)$ is an eigenvalue of \eqref{prob-eigenvalues}. If
$\Psi$ is an eigenfunction of \eqref{prob-eigenvalues} associated with
$(2m+1)(2m+1+n-2s)$, then the function defined as
\begin{equation*}
U(z)=|z|^\gamma\Psi\left(\frac{z}{|z|}\right), \quad z=(x,t)\in \mathbb{R}^2_+,
\end{equation*}
with 
\begin{equation*}
\gamma=-\frac{n-2s}{2}+\sqrt{\left(\frac{n-2s}{2}\right)^2+(2m+1)(2m+1+n-2s)}=2m+1
\end{equation*}
is a nontrivial solution to \eqref{prob-div-0}, even with respect to
$x$. Hence, the even reflection $\widetilde{U}$ of $U$ with respect to
$t$, given by $\widetilde{U}(x,t):=U(x,|t|)$, is a solution of
$\mathop{\rm div}(|t|^{1-2s}\nabla \widetilde{U})=0$ in
$\mathbb{R}^{2}$, and in addition
$\widetilde{U}\in C^\infty(\mathbb{R}^{2})$ thanks to \cite[Theorem
1.1]{SirTerVit21a}. Moreover, from \cite[Lemma A.1]{DelFelSic23} it
follows that $\widetilde{U} $ is a homogeneous polynomial of degree
$2m+1$, namely
\begin{equation*}
\widetilde{U}(x,t)= \sum_{k=0}^{2m+1}a_k x^{k} t^{2m+1-k}.
\end{equation*}
We have that $a_k=0$ if $k$ is odd since $\widetilde{U}$ is even with
respect to $x$. Furthermore, $a_k=0$ when $k$ is even, since
$\widetilde{U}$ is also even with respect to $t$. Consequently,
$\tilde{U}$ is trivial, leading to a contradiction.

Now we prove that the numbers given in \eqref{eigenvalues} are
eigenvalues of \eqref{prob-eigenvalues}: to this end, we need to show
that, for any fixed $m \in \mathbb{N}$, there exists an eigenfunction
of \eqref{prob-eigenvalues} associated with $m(m+n-2s)$ if $n\geq2$,
and an eigenfunction associated with $2m(2m+n-2s)$ if $n=1$. This is
equivalent to find, for any fixed $m \in \mathbb{N}$, a nontrivial
solution to \eqref{prob-div-0} which is even with respect to $x_n$ and
homogeneous of degree $m$ if $n\geq2$, and of degree $2m$ if $n=1$.
To this purpose, we observe that the equation in \eqref{prob-div-0}
can be written as 
\begin{equation}\label{eq:eginvluaes-computation:1}
\Delta U +\frac{1-2s}{t}\partial_tU=0. 
\end{equation}
We first focus on the case $n=1$.  Let $m\in \mathbb{N}$, and consider
the following homogeneous polynomial of degree $2m$, even with respect
to $x_n=x$,
\begin{equation*}
  U(x,t):=\sum_{k=0}^{m} a_k x^{2k} t^{2m-2k},
\end{equation*} 
with $a_0, \dots, a_{m} \in \mathbb{R}$. One can directly verify that
$U$ is a solution of \eqref{eq:eginvluaes-computation:1} if and only
if
\begin{equation*}
  a_{k+1}=-\frac{2(m-k)(m-k-s)}{(k+1)(2k+1)} a_{k} \quad
    \text{for all } k \in \{0,\dots,m-1\}.
\end{equation*}
Thus, choosing $a_0=1$, we obtain a nontrivial solution to
\eqref{prob-div-0} which is even with respect to $x_n=x$ and
homogeneous of degree $2m$.

If $n\geq2$, we first observe that, for any $m \in \mathbb{N}$, there
exists a nontrivial homogeneous harmonic polynomial $P$ of degree $m$
in the variables $x_1,\dots, x_{n},$, which is even with respect to
$x_n$. Indeed, in dimension $n=2$, we consider the polynomial
$Q(y_1,y_2)=\mathrm{Re}((y_1+iy_2)^m)$, which is a
homogeneous harmonic polynomial of degree $m$. Moreover, $Q$ is even
with respect to the variable $y_2$, since
\begin{equation*}
  Q(y_1,-y_2)=\mathrm{Re}((y_1-iy_2)^m)=\mathrm{Re}(\overline{(y_1-iy_2)^m})
  =\mathrm{Re}((y_1+iy_2)^m)=Q(y_1,y_2).
\end{equation*}
Then, in any dimension $n\geq2$, one can consider
$P(x_1,\dots,x_n)=Q(x_{n-1},x_n)$, which is a nontrivial homogeneous
harmonic polynomial of degree $m$, even with respect to $x_n$.  We
conclude that the function $U(x_1,\dots,x_{n},t):= P(x_1,\dots,x_{n})$
is a nontrivial solution to \eqref{prob-div-0}, which is even with
respect to $x_n$ and homogeneous of degree $m$. Statement (i) is
thereby proved.

We finally observe that \cite[Lemma B.2]{DelFelVit22}, see also
\cite[Lemma A.1]{DelFelSic23}, together with the above analysis,
implies that any function $U$ constructed as in \eqref{eq:U} from an
eigenfunction $Y$ is necessarily a polynomial solving \eqref{eq:hom},
thus proving (ii).
\end{proof}

\section{Local Asymptotics and strong unique continuation
  property}\label{seclocal}

\subsection{The blow-up analysis}
In the present subsection, we perform a blow-up analysis in order to
classify the limit profiles of suitably normalized rescaled
solutions. Let us fix a non-trivial weak solution
$w\in H^1(B^+_{r_0}, t^{1-2s}dz)$ to problem \eqref{problemadiWtilde},
which is even with respect to the variable $x_n$,
  i.e.
\begin{equation}\label{eq:w-pari}
  w(x',x_n,t)= w(x',-x_n,t) \quad\text{for a.e.
    $(x',x_n,t)\in B_{r_0}^+$}.
\end{equation}
We recall that the solution $w$ to \eqref{problemadiWtilde},
constructed in Section \ref{sez2} from a solution
$v\in H^1(\mathcal{C}_{\Omega},t^{1-2s}dz)$ to \eqref{eqrisoltadaV}
via a change of variables, a diffeomorphic transformation, and an
even reflection, satisfies \eqref{eq:w-pari}, see \eqref{Wwidehat}.

We consider the family $\{w^\lambda\}_{\lambda\in (0,r_0]}$ defined by
\begin{equation}\label{wlambda}
    w^\lambda(z):=\frac{w(\lambda z)}{\sqrt{H(\lambda)}},\quad z\in B^+_{r_0/\lambda}.
\end{equation}
In view of \eqref{problemadiWtilde}, for
every $\lambda\in (0,r_0]$, $w^\lambda$ weakly solves
\begin{equation}\label{eqriscalata}
    \begin{cases}
      -\mathop{\rm div}\left(t^{1-2s}A(\lambda\,\cdot)\nabla
        w^\lambda\right)
      +t^{1-2s}\lambda\widetilde{\mathbf{b}}(\lambda\cdot)\cdot \nabla
      _x w^\lambda + t^{1-2s}\lambda^2\tilde{a}(\lambda\cdot)
      w^\lambda= 0&\mathrm{in \ } B_{{r_0}/{\lambda}}^+,\\
      - \lim_{t\to0^+}t^{1-2s}\zeta(\lambda\cdot)\partial_t
        w^\lambda=\lambda^{2s} \tilde h(\lambda \cdot) \mathop{\rm
        Tr}w^\lambda &\mathrm{on \ } B'_{r_0/\lambda},
\end{cases}
\end{equation}
that is, for every
$\varphi\in C_c^\infty(B^+_{r_0/\lambda}\cup B'_{r_0/\lambda})$,
\begin{multline}\label{formdeblambda}
  \int_{B^+_{r_0/\lambda}} t^{1-2s}A(\lambda z) \nabla
  w^\lambda(z)\cdot \nabla\varphi(z)\,dz +\lambda
  \int_{B^+_{r_0/\lambda}} t^{1-2s}(\widetilde{\mathbf{b}}
  (\lambda x)\cdot \nabla _x w^\lambda(z) )\varphi(z)\,dz \\
  +\lambda^2 \int_{B^+_{r_0/\lambda}} t^{1-2s}\tilde{a}(\lambda x)
  w^\lambda(z)\varphi(z)\,dz= \lambda^{2s}\int_{B'_{r_0/\lambda}}
  \tilde{h}(\lambda x) \mathop{\rm Tr} w^\lambda(x)\mathop{\rm Tr}
  \varphi(x)\,dx.
\end{multline}
We first observe that the family $\{w^\lambda\}_{\lambda\in (0,r_0]}$
is uniformly bounded in $H^1(B^+_1, t^{1-2s}dz)$.

\begin{lemma}\label{lafamiglialimitata}
  There exists $L>0$ such that
  $\Vert w^\lambda\Vert _{H^1(B^+_1, t^{1-2s}dz)}\leq L$ for every
  $\lambda\in (0,r_0]$.
\end{lemma}
\begin{proof}
  By \eqref{wlambda}, a change of variable, and Lemma \ref{lemmastima1}, we
  have 
    \begin{equation}\label{gradlimitato}
      \int_{B^+_1} t^{1-2s} |\nabla w^\lambda|^2\,dz =
      \frac{\lambda^{2s-n}}
      {H(\lambda)} \int_{B^+_\lambda} t^{1-2s}|\nabla w|^2\, dz \leq
      C_{n,s}
      \left(\mathcal{N}(\lambda)+\frac{n-2s}{2}\right),
    \end{equation}
    which yields the uniform boundedness of
    $\{\nabla w^\lambda\}_{\lambda\in (0,r_0]}$ in
    $L^2(B^+_1, t^{1-2s}dz)$ in view of \eqref{Nmindicost}. In order
    to prove the uniform boundedness of
    $\{w^\lambda\}_{\lambda\in (0,r_0]}$ in $L^2(B^+_1, t^{1-2s}dz)$,
    we observe that
    \begin{equation}\label{intlambda1}
      \int_{\mathbb{S}^+} \theta^{1-2s}_{n+1}\mu(\lambda\cdot)
      |w^\lambda|^2\,dS= \frac{\lambda^{2s-n-1}}{H(\lambda)}
      \int_{\partial^+ B^+_\lambda} t^{1-2s}\mu |w|^2\,dS= 1.
    \end{equation}
    Then the thesis follows from \eqref{gradlimitato},
    \eqref{intlambda1}, \eqref{1sumulimitata}, and \cite[Lemma
      2.4]{FalFel14}.
\end{proof}
\begin{lemma}\label{lemmaconvsullasfera}
  For every sequence $\lambda_{\ell}\to 0^+$, there exist a subsequence
  $\lambda_{\ell_k}\to 0^+$ and a function
  $g\in L^2(\mathbb{S}^+, \theta^{1-2s}_{n+1}dS)$ such that
    \begin{equation}\label{eq:condebnor}
      A(\lambda_{\ell_k}\cdot)\nabla w^{\lambda_{\ell_k}}\cdot\nu
      \rightharpoonup g \quad \text{weakly in $L^2(\mathbb{S}^+,
        \theta^{1-2s}_{n+1}dS)$ as $k\to \infty$}.
    \end{equation}
\end{lemma}
\begin{proof}
  We first claim that
\begin{equation}\label{claimlambda}
  \text{$\{\nabla w^{\lambda}\}_{\lambda\in (0,r_0]}$ is bounded in
    $L^2(\mathbb{S}^+,
    \theta^{1-2s}_{n+1}dS)$.}
\end{equation}
In order to prove \eqref{claimlambda}, we observe that
$\{\nabla _x w^{\lambda}\}_{\lambda\in (0,r_0]}$ is bounded in
$H^1(B^+_1, t^{1-2s}dz)$ and
$\{t^{1-2s}\partial_t w^{\lambda}\}_{\lambda\in (0,r_0]}$ is bounded
in $H^1(B^+_1, t^{2s-1}dz)$ by the regularity result in \cite[Theorem
2.1]{FelSic22} and Lemma \ref{lafamiglialimitata}.  Therefore, claim
\eqref{claimlambda} is a consequence of the continuity of the trace
operators in \eqref{traceoper} with $r=1$ and
 \begin{equation*}
   \int_{\mathbb{S}^+}\theta_{n+1}^{1-2s} |\nabla
   w^{\lambda}|^2\,dS=\int_{\mathbb{S}^+}
   \theta_{n+1}^{1-2s} |\nabla_x w^{\lambda}|^2\,dS +
   \int_{\mathbb{S}^+}\theta_{n+1}^{2s-1}
   \left|\theta_{n+1}^{1-2s}\partial_t w^{\lambda}\right|^2\,dS.
 \end{equation*} 
 In view of \eqref{normadiAlimitata}, \eqref{claimlambda} implies that
 also
 $\{A(\lambda \cdot)\nabla w^{\lambda}\cdot\nu\}_{\lambda\in (0,r_0]}$
 is bounded in $L^2(\mathbb{S}^+, \theta^{1-2s}_{n+1}dS)$.
This directly yields the conclusion.
\end{proof}
\begin{lemma}\label{lemmablowup1}
  Let $\gamma$ be as in \eqref{zeta}. For any sequence
  $\lambda_\ell\to 0^+$, there exist a subsequence $\lambda_{\ell_k}\to 0^+$
  and a homogeneous function $\bar{w}\in H^1(B^+_1, t^{1-2s}dz)$ of
  degree $\gamma$, i.e. satisfying
    \begin{equation*}
      \bar{w} (z)=|z|^\gamma\bar{w}\left(\frac{z}{|z|}\right)\quad
      \text{for all $z\in B^+_1$}, 
    \end{equation*}
    such that 
\begin{equation}\label{wlambdaconvforte}
  w^{\lambda_{\ell_k}}\to \bar{w}\quad \text{in $H^1(B^+_1, t^{1-2s}dz)$ as $k\to \infty$}. 
    \end{equation}
    In addition, $\gamma(\gamma+n-2s)$ is an eigenvalue of
      problem \eqref{prob-eigenvalues} and
    $\psi:= \bar{w}|_{\mathbb{S}^+}$ is an associated
    $L^2(\mathbb{S}^+,\theta_{n+1}^{1-2s}dS)$-normalized
    eigenfunction.
\end{lemma}
\begin{proof}
  Thanks to Lemma \ref{lafamiglialimitata}, for any sequence
  $\lambda_\ell\to 0^+$,
  there exist a subsequence $\lambda_{\ell_k}\to 0^+$
  and a function $\bar{w}\in H^1(B^+_1, t^{1-2s}dz)$ such that
\begin{equation}\label{convdebo}
  w^{\lambda_{\ell_k}} \rightharpoonup \bar{w} \quad
  \text{in $H^1(B^+_1, t^{1-2s}dz)$ as $k\to \infty$}. 
\end{equation}
We first observe that $\bar{w}\not\equiv0$: indeed by the compactness of the
trace operator \eqref{traceoper} with $r=1$ and \eqref{convdebo}, we
have
\begin{equation}\label{convfortetracce}
  w^{\lambda_{\ell_k}} \to \bar{w} \quad
  \text{strongly in $L^2(\mathbb{S}^+, \theta_{n+1}^{1-2s}dS)$ as $k\to \infty$}; 
\end{equation}
combining \eqref{sviluppomu},
\eqref{intlambda1}, and \eqref{convfortetracce}, we conclude that
\begin{equation}\label{intbarw1}
\int_{\mathbb{S}^+}\theta_{n+1}^{1-2s}|\bar{w}|^2\,dS=1,  
\end{equation}
hence $\bar{w}\not\equiv0$.

We now prove that $\bar{w}$ weakly solves 
\begin{equation}\label{problemarisoltodawbar}
      \begin{cases}
  -\mathop{\rm div}(t^{1-2s}\nabla \bar{w})= 0&\mathrm{in \ } B_{1}^+,\\
  \lim_{t\to0^+} t^{1-2s}\partial_ t\bar{w}=0 &\mathrm{on \ } B'_{1},
\end{cases}
\end{equation}
i.e. 
\begin{equation}\label{formdeblimite}
    \int_{B^+_1} t^{1-2s}\nabla \bar{w}\cdot\nabla\varphi\,dz=0
  \end{equation}
  for every $\varphi\in C^\infty_c(B^+_1\cup B'_1)$.  To this aim, we
  first notice that, if $\varphi\in C^\infty_c(B^+_1\cup B'_1)$, then
  $\varphi\in C_c^\infty(B^+_{r_0/{\lambda_{\ell_k}}}\cup
  B'_{r_0/\lambda_{\ell_k}})$ for sufficiently large $k$. Hence
  \eqref{formdeblambda} is satisfied for sufficiently large $k$, i.e.
\begin{multline}\label{devofarneilimite}
  \int_{B^+_1} t^{1-2s}A(\lambda_{\ell_k} z) \nabla
  w^{\lambda_{\ell_k}}(z)\cdot
  \nabla\varphi(z)\,dz +\lambda_{\ell_k} \int_{B^+_1} t^{1-2s}
  (\widetilde{\mathbf{b}}(\lambda_{\ell_k} x)\cdot \nabla _x
  w^{\lambda_{\ell_k}}(z) )
  \varphi(z)\,dz \\
  +\lambda_{\ell_k}^2 \int_{B^+_1} t^{1-2s}\tilde{a}(\lambda_{\ell_k} x)
  w^{\lambda_{\ell_k}}(z)\varphi(z)\,dz= 
  \lambda_{\ell_k}^{2s}\int_{B'_1}\tilde{h}(\lambda_{\ell_k} x) \mathop{\rm
    Tr} w^{\lambda_{n\ell_k}}(x)\mathop{\rm Tr} \varphi(x)\,dx
\end{multline}
for sufficiently large $k$.  By \eqref{Aperturbazione} and  the
Cauchy–Schwarz inequality, we have
\begin{multline*}
  \int_{B^+_1} t^{1-2s}A(\lambda_{\ell_k} z) \nabla
  w^{\lambda_{\ell_k}}(z)\cdot \nabla\varphi(z)\,dz =
  \int_{B^+_1} t^{1-2s}\nabla w^{\lambda_{\ell_k}}(z)\cdot \nabla\varphi(z)\,dz\\
  +O(\lambda_{\ell_k}) \left(\int_{B^+_1} t^{1-2s}|\nabla
    w^{\lambda_{\ell_k}}(z)|^2\,dz\right)^{\frac{1}{2}}\left(\int_{B^+_1}
    t^{1-2s}
    |\nabla \varphi(z)|^2\,dz\right)^{\frac{1}{2}},
\end{multline*}
hence, by Lemma \ref{lafamiglialimitata} and
\eqref{convdebo} we conclude that
\begin{equation}\label{lim1}
  \int_{B^+_1} t^{1-2s}A(\lambda_{\ell_k} z) \nabla
  w^{\lambda_{\ell_k}}(z)\cdot \nabla\varphi(z)\,dz
  \to
  \int_{B^+_1} t^{1-2s}\nabla \bar{w}(z)\cdot\nabla\varphi(z)\,dz \quad \text{as $k\to \infty$.}
\end{equation}
Moreover, by \eqref{nuovaregolarita} and Lemma \ref{lafamiglialimitata}, 
\begin{align}
  \label{lim2}&    \lambda_{\ell_k} \int_{B^+_1}
    t^{1-2s}(\widetilde{\mathbf{b}}(\lambda_{\ell_k}
    x)\cdot \nabla _x w^{\lambda_{\ell_k}}(z) )\varphi(z)\,dz
    +\lambda_{\ell_k}^2  \int_{B^+_1}
    t^{1-2s}\tilde{a}(\lambda_{\ell_k} x)
    w^{\lambda_{\ell_k}}(z)\varphi(z)\,dz\\
\notag=\,&O(\lambda_{\ell_k})\left(\int_{B^+_1} t^{1-2s}|
  \varphi|^2\,dz\right)^{\frac{1}{2}}
\bigg(\bigg(\int_{B^+_1} t^{1-2s}|\nabla
w^{\lambda_{\ell_k}}|^2\,dz\bigg)^{\frac{1}{2}}
           +\bigg(\int_{B^+_1} t^{1-2s}
           |  w^{\lambda_{\ell_k}}|^2\,dz\bigg)^{\frac{1}{2}}\bigg)\\
  \notag=\,&o(1)\quad\text{as }k\to\infty.
\end{align}
Finally, applying the H\"{o}lder inequality to
$\tilde{h}(\lambda_{\ell_k}\cdot)\in L^{p}(B'_1)$ and
$\mathop{\rm Tr}w^{\lambda_{\ell_k}},\mathop{\rm Tr}\varphi\in
L^{2^\ast_s}(B'_1)$, and using Lemma \ref{lemmafallfelli},
\eqref{1sumulimitata}, and \eqref{intlambda1}, we obtain
\begin{multline*}
  \lambda_{\ell_k}^{2s}\biggl|\int_{B'_1}\tilde{h}(\lambda_{\ell_k} x)
  \mathop{\rm Tr} w^{\lambda_{\ell_k}}(x)\mathop{\rm Tr}
  \varphi(x)\,dx\biggr|
  \leq S_{n,s}\lambda_{\ell_k}^{2s} \,
  \omega_n^{\frac{2sp-n}{np}}\Vert \tilde{h}(\lambda_{\ell_k}\cdot)\Vert_{L^{p}(B'_1)} \\
  \times\left(2(n-2s)\int_{\mathbb{S}^+}\theta_{n+1}^{1-2s}\mu(\lambda_{\ell_k}\cdot)
    |w^{\lambda_{\ell_k}}|^2\,dS+\int_{B^+_1}
    t^{1-2s}|\nabla w^{\lambda_{\ell_k}}|^2\,dz\right)^{\frac{1}{2}}\\
  \times
  \left(\frac{n-2s}{2}\int_{\partial^+B^+_1}t^{1-2s}\varphi^2\,dS
    +\int_{B^+_1} t^{1-2s}|\nabla \varphi|^2\,dz\right)^{\frac{1}{2}}
  \\
  =S_{n,s}\lambda_{\ell_k}^{2s-\frac{n}{p}}
  \,\omega_n^{\frac{2sp-n}{np}}\Vert
  \tilde{h}\Vert_{L^{p}(B'_{\lambda_{\ell_k}})} \left(2(n-2s)+
    \int_{B^+_1} t^{1-2s}|\nabla
    w^{\lambda_{\ell_k}}|^2\,dz\right)^{\frac{1}{2}}\\
  \times
  \left(\frac{n-2s}{2}\int_{\partial^+B^+_1}t^{1-2s}\varphi^2\,dS
    +\int_{B^+_1} t^{1-2s}|\nabla \varphi|^2\,dz\right)^{\frac{1}{2}}
  ,
\end{multline*}
$\omega_n$ being the $n$-dimensional Lebesgue measure of $B'_1$; hence, Lemma
\ref{lafamiglialimitata} and the fact that $2s-\frac{n}{p}>0$ yield
\begin{equation}
  \label{lim3}
\lim_{k\to\infty }
  \lambda_{\ell_k}^{2s}\int_{B'_1}\tilde{h}(\lambda_{\ell_k} x) 
\mathop{\rm Tr} w^{\lambda_{\ell_k}}(x)\mathop{\rm Tr} \varphi(x)\,dx=0
\end{equation}
Combining \eqref{lim1}, \eqref{lim2}, and \eqref{lim3}, and taking the
limit as $k\to\infty$ in \eqref{devofarneilimite}, we obtain
\eqref{formdeblimite}.

The next step is to show that the convergence in \eqref{convdebo} is
actually strong, meaning that \eqref{wlambdaconvforte} holds.
Multiplying \eqref{eqriscalata} by any fixed
$\varphi\in H^1(B^+_1,t^{1-2s}dz)$ and integrating by parts (see
  \cite[Proposition 3.7]{FelSic22}), we obtain that, for sufficiently
large $k$,
\begin{multline}\label{anchecolpezzodibordo}
  \int_{B^+_1} t^{1-2s}A(\lambda_{\ell_k} z) \nabla
  w^{\lambda_{\ell_k}}(z)\cdot \nabla\varphi(z)\,dz +\lambda_{\ell_k}
  \int_{B^+_1} t^{1-2s} (\widetilde{\mathbf{b}}(\lambda_{\ell_k}
  x)\cdot \nabla _x w^{\lambda_{\ell_k}}(z))
  \varphi(z)\,dz \\
  +\lambda_{\ell_k}^2 \int_{B^+_1} t^{1-2s}\tilde{a}(\lambda_{\ell_k}
  x)
  w^{\lambda_{\ell_k}}(z)\varphi(z)\,dz\\
  = \lambda_{\ell_k}^{2s}\int_{B'_1}\tilde{h}(\lambda_{\ell_k} x)
  \mathop{\rm Tr} w^{\lambda_{\ell_k}}(x)\mathop{\rm Tr}
  \varphi(x)\,dx + \int_{\mathbb{S}^+}
  \theta_{n+1}^{1-2s}(A(\lambda_{\ell_k} z) \nabla
  w^{\lambda_{\ell_k}}\cdot\nu) \varphi\, ds.
\end{multline}
We observe that the limits in \eqref{lim1}--\eqref{lim3} still hold
for any $\varphi\in H^1(B^+_1,t^{1-2s}dz)$. Thus, also invoking Lemma
\ref{lemmaconvsullasfera}, we find that, up to a further subsequence
still denoted by $\lambda_{\ell_k}$, taking the limit as
$k\to \infty$ in \eqref{anchecolpezzodibordo} yields
\begin{equation}\label{g}
  \int_{B^+_1} t^{1-2s} \nabla \bar{w}\cdot\nabla \varphi\,dz=
  \int_{\mathbb{S}^+}
  \theta_{n+1}^{1-2s}g\varphi\,dS,
\end{equation}
for some function $g\in L^2(\mathbb{S}^+, \theta^{1-2s}_{n+1}dS)$
satisfying \eqref{eq:condebnor}. On the other hand, choosing
$\varphi=w^{\lambda_{\ell_k}}$ in \eqref{anchecolpezzodibordo}, for
sufficiently large $k$ we have
\begin{multline}\label{testataconleistessa}
  \int_{B^+_1} t^{1-2s}A(\lambda_{\ell_k} \cdot) \nabla
  w^{\lambda_{\ell_k}}\cdot \nabla w^{\lambda_{\ell_k}}\,dz
  +\lambda_{\ell_k} \int_{B^+_1} t^{1-2s}(\widetilde{\mathbf{b}}
  (\lambda_{\ell_k} \cdot)\cdot \nabla _x w^{\lambda_{\ell_k}})
  w^{\lambda_{\ell_k}}\,dz \\
  +\lambda_{\ell_k}^2 \int_{B^+_1} t^{1-2s}\tilde{a}(\lambda_{\ell_k}
  \cdot)
  |w^{\lambda_{\ell_k}}|^2\,dz\\
  = \lambda_{\ell_k}^{2s} \int_{B'_1}\tilde{h}(\lambda_{\ell_k}
  \cdot) |\mathop{\rm Tr} w^{\lambda_{\ell_k}}|^2\,dx +
  \int_{\mathbb{S}^+} \theta_{n+1}^{1-2s}(A(\lambda_{\ell_k} \cdot)
  \nabla w^{\lambda_{\ell_k}}\cdot\nu) w^{\lambda_{\ell_k}}\, dS.
         \end{multline}
         By proceeding as in the proof of \eqref{lim2} and \eqref{lim3}
           with $\varphi=w^{\lambda_{\ell_k}}$ and taking 
           Lemma \ref{lafamiglialimitata} into account, we can show
           that
           \begin{equation}\label{eq:dt0}
             \lim_{k\to\infty}\bigg(\lambda_{\ell_k} \int_{B^+_1} t^{1-2s}(\widetilde{\mathbf{b}}
  (\lambda_{\ell_k} \cdot)\cdot \nabla _x w^{\lambda_{\ell_k}})
  w^{\lambda_{\ell_k}}\,dz 
  +\lambda_{\ell_k}^2 \int_{B^+_1} t^{1-2s}\tilde{a}(\lambda_{\ell_k}
  \cdot)
  |w^{\lambda_{\ell_k}}|^2\,dz\bigg)=0
           \end{equation}
           and
\begin{equation}\label{eq:ut0}
  \lim_{k\to\infty}\bigg(\lambda_{\ell_k}^{2s} \int_{B'_1}\tilde{h}(\lambda_{\ell_k}
  \cdot) |\mathop{\rm Tr} w^{\lambda_{\ell_k}}|^2\,dx\bigg) =0.
\end{equation}  
Hence, from \eqref{testataconleistessa} and \eqref{eq:condebnor} it
follows that
\begin{align}\label{dausaresubito}
  \lim_{k\to \infty}\int_{B^+_1} t^{1-2s}A(\lambda_{\ell_k} \cdot)
  \nabla w^{\lambda_{\ell_k}}\cdot \nabla
  w^{\lambda_{\ell_k}}\,dz&=\lim_{k\to\infty}
                            \int_{\mathbb{S}^+}
                            \theta_{n+1}^{1-2s}(A(\lambda_{\ell_k}
                            \cdot)
                            \nabla w^{\lambda_{\ell_k}}\cdot\nu) w^{\lambda_{\ell_k}}\, dS\\
  \notag    &= \int_{\mathbb{S}^+}\theta_{n+1}^{1-2s} g\bar{w}\,dS,
\end{align}
where we also used \eqref{convfortetracce}.  Therefore, by combining
\eqref{Aperturbazione}, Lemma \ref{lafamiglialimitata},
\eqref{dausaresubito}, and \eqref{g} with $\varphi=\bar{w}$, we
finally obtain that
\begin{align*}
  \lim_{k\to \infty}&\int_{B^+_1} t^{1-2s}|\nabla w^{\lambda_{\ell_k}}|^2\,dz\\
                    &=\lim_{k\to \infty}\biggl(\int_{B^+_1} t^{1-2s}
                      A(\lambda_{\ell_k} \cdot) \nabla
                      w^{\lambda_{\ell_k}}\cdot
                      \nabla
                      w^{\lambda_{\ell_k}}\,dz+O(\lambda_{\ell_k})
                      \int_{B^+_1} t^{1-2s} |\nabla w^{\lambda_{\ell_k}}|^2\,dz\biggr)\\
                    &= \int_{\mathbb{S}^+}\theta_{n+1}^{1-2s} g\bar{w}\,dS =
                      \int_{B^+_1}
                      t^{1-2s} |\nabla \bar{w}|^2\,dz.
\end{align*}
This, together with the weak convergence of $\nabla w^{\lambda_{\ell_k}}$
to $\nabla \bar{w}$ in $L^2(B^+_1, t^{1-2s}dz)$ by virtue of
\eqref{convdebo}, implies that
$\nabla w^{\lambda_{\ell_k}} \to \nabla \bar{w}$ strongly in
$L^2(B^+_1, t^{1-2s}dz)$. By \cite[Lemma 2.4]{FalFel14} and
\eqref{convfortetracce}, this also yields the convergence
$w^{\lambda_{\ell_k}} \to \bar{w}$ in $L^2(B^+_1, t^{1-2s}dz)$.
Statement \eqref{wlambdaconvforte} is thereby proved.

To complete the proof of lemma, it remains to show that $\bar{w}$ is
homogeneous of degree $\gamma$. For this purpose, we consider its
Almgren frequency function, which is defined as follows in view of
\eqref{problemarisoltodawbar}.  For every $r\in (0,1]$, let
\begin{equation*}
    E_{\bar{w}}(r):= r^{2s-n}\int_{B^+_r} t^{1-2s}|\nabla \bar{w}|^2\,dz,
\end{equation*}
and 
\begin{equation*}
    H_{\bar{w}}(r):= r^{2s-n-1}\int_{\partial ^+ B^+_r} t^{1-2s}\bar{w}^2\,dS. 
\end{equation*}
By arguing as in the proof of Lemma \ref{lemmasuH}-(i) with
$A=\mathrm{Id}_{n+1}$, $\tilde{\mathbf{b}}=\mathbf{0}$, $\tilde{a}=0$,
$\tilde h=0$, and $\mu=1$, one can show that $H_{\bar{w}}(r)>0$
for every $r\in (0,1]$.  Thus the function
\begin{equation}\label{Nwbar}
  \mathcal{N}_{\bar{w}}(r):=
  \frac{E_{\bar{w}}(r)}{H_{\bar{w}}(r)}\quad
  \text{for all $r\in (0,1]$}
\end{equation}
is well-defined. For each fixed $k$, we also consider the Almgren 
frequency function associated with $w^{\lambda_{\ell_k}}$: 
for every $r\in (0,1]$, let 
\begin{align*}
  &E_k(r)\\
  &:= r^{2s-n} \!\int_{B^+_r}\!t^{1-2s}\bigg(\!A(\lambda_{\ell_k}
    \cdot)
    \nabla w^{\lambda_{\ell_k}}\!\cdot\! \nabla w^{\lambda_{\ell_k}} 
    + \left(\lambda_{\ell_k} \tilde{\mathbf{b}}(\lambda_{\ell_k}\cdot)
    \!\cdot\! \nabla_x w^{\lambda_{\ell_k}} + \lambda_{\ell_k}^2
    \tilde{a}(\lambda_{\ell_k}\cdot)w^{\lambda_{\ell_k}}\!\right)
    w^{\lambda_{\ell_k}}\! \bigg) dz\\
  &\quad -
    \lambda_{\ell_k}^{2s}\int_{B'_r}\tilde{h}(\lambda_{\ell_k}\cdot)
    |\mathop{\rm Tr}w^{\lambda_{\ell_k}}|^2\,dx
    \end{align*}
and 
\begin{equation*}
  H_k(r):=
  r^{2s-n-1}\int_{\partial^+ B^+_r} t^{1-2s}\mu(\lambda_{\ell_k}\cdot)|w^{\lambda_{\ell_k}}|^2\,dS.
\end{equation*}
For $k$ sufficiently large, $H_k(r)>0$ by Lemma
\ref{lemmasuH}-(i). Then, the frequency
\begin{equation*}
  \mathcal{N}_k(r):= \frac{E_k(r)}{H_k(r)},\quad r\in (0,1],
\end{equation*}
is well defined. Moreover, by direct computation,
\begin{equation*}
  E_k(r)=   \frac{E(\lambda_{\ell_k}r)}{H(\lambda_{\ell_k})}, 
  \quad H_k(r)=   \frac{H(\lambda_{\ell_k}r)}{H(\lambda_{\ell_k})}, 
\end{equation*}
and hence 
\begin{equation}\label{Nkugualea}
  \mathcal{N}_k(r)= \mathcal{N}(\lambda_{\ell_k}r).  
\end{equation}
From \eqref{Aperturbazione}, Lemma \ref{lafamiglialimitata},
\eqref{wlambdaconvforte}, \eqref{eq:dt0}, and \eqref{eq:ut0} it
follows that
\begin{equation}\label{Ekconvergea}
    \lim_{k\to\infty}E_k(r)= E_{\bar{w}}(r).
\end{equation}
Furthermore,  combining \eqref{sviluppomu} and \eqref{convfortetracce} yields
\begin{equation}\label{Hkconvergea}
    \lim_{k\to\infty}H_k(r)= H_{\bar{w}}(r).
\end{equation}
Thus, from \eqref{zeta}, \eqref{Nkugualea}, \eqref{Ekconvergea}, and
\eqref{Hkconvergea}, it follows that $\mathcal{N}_{\bar{w}}$ is
constant; more precisely, for every $r\in (0,1]$,
\begin{equation}\label{Nwbarcostante}
  \mathcal{N}_{\bar{w}}(r)=\lim_{k\to \infty}
  \mathcal{N}_k(r)=\lim_{k\to \infty}\mathcal{N}(\lambda_{\ell_k}r)=\gamma. 
\end{equation}
Consequently, $\mathcal{N}_{\bar{w}}'(r)=0$ for a.e. $r\in(0,1]$.  On
the other hand, arguing as in the proof of Lemma \ref{lemmaN'} with
$A=\mathrm{Id}_{n+1}$, $\tilde{\mathbf{b}}=\mathbf{0}$, $\tilde{a}=0$,
and $\mu=1$, we can prove that
\begin{equation*}
    \mathcal{N}_{\bar{w}}'(r)= 2r\frac{\left(\int_{\partial^+ B^+_r} 
    t^{1-2s}|\partial_\nu \bar{w}|^2\,dS\right)\left(
    \int_{\partial^+ B^+_r} t^{1-2s} \bar{w}^2\,dS\right)-
    \left(\int_{\partial^+ B^+_r} t^{1-2s}\partial_\nu 
    \bar{w} \,\bar{w}\,dS\right)^2}{\left(\int_{\partial^+ B^+_r}
    t^{1-2s}\bar{w}^2\,dS\right)^2}.
\end{equation*}
Hence 
\begin{equation*}
    \left(\int_{\partial^+ B^+_r} 
    t^{1-2s}|\partial_\nu \bar{w}|^2\,dS\right)\!\left(
    \int_{\partial^+ B^+_r} t^{1-2s} \bar{w}^2\,dS\right)\!-\!
    \left(\int_{\partial^+ B^+_r} t^{1-2s}\partial_\nu 
    \bar{w} \,\bar{w}\,dS\right)^{\!2}\!=0\text{ for a.e. }r\in(0,1).
\end{equation*}
In the Hilbert space $L^2(\partial^+ B^+_r, t^{1-2s}ds)$, equality in
the Cauchy-Schwarz inequality holds only if the vectors are linearly
dependent. Therefore, for a.e. $r\in (0,1)$ and each
$\theta\in \mathbb{S}^+$,
\begin{equation}\label{integrarerispettoar}
    \partial _\nu \bar{w}(r\theta)= \eta(r) \bar{w}(r\theta),
\end{equation}
for some function $\eta=\eta(r)$ defined a.e. in $(0,1)$. Multiplying
the above identity by $t^{1-2s}\bar{w}(z)$ and integrating over
$\partial^+ B^+_r$, we obtain 
\begin{equation}\label{eq:au1}
    \int_{\partial^+ B^+_r}t^{1-2s}\bar{w}\partial _\nu \bar{w}\,dS=\eta(r)
    \int_{\partial^+ B^+_r}t^{1-2s}\bar{w}^2\,dS=\eta(r)
    r^{n+1-2s}H_{\bar{w}}(r)
\end{equation}
for a.e. $r\in(0,1)$. Moreover, by \eqref{aux2} applied to $\bar w$
(hence, with $\mu\equiv 1$ and $\nabla \mu\equiv0$), we have
  \begin{equation}\label{eq:au2}
  H_{\bar w}'(r)= 2r^{2s-n-1}\int_{\partial^+ B^+_r}t^{1-2s}
  \bar w\partial_\nu \bar w \,dS,
\end{equation}
whereas \eqref{identity} applied to $\bar w$  (hence, with 
 $A=\mathrm{Id}_{n+1}$, $\tilde{\mathbf{b}}=\mathbf{0}$, $\tilde{a}=0$, and 
$\tilde h=0$) yields 
\begin{equation}\label{eq:au3}
\int_{\partial^+ B^+_r}t^{1-2s}
  \bar w\partial_\nu \bar w \,dS=r^{n-2s}E_{\bar w}(r)
\end{equation}
for a.e. $r\in (0,1)$. Combining \eqref{eq:au1}, \eqref{eq:au2}, and
\eqref{eq:au3} with \eqref{Nwbar} and \eqref{Nwbarcostante}, we
conclude that
\begin{equation*}
  \eta(r)= \frac{H'_{\bar{w}}(r)}{2 H_{\bar{w}}(r)}=
  \frac{\mathcal{N}_{\bar{w}}(r)}{r}=\frac{\gamma}{r}\quad\text{for a.e. }r\in(0,1).
\end{equation*}
 Substituting this into \eqref{integrarerispettoar} and
integrating over $(r,1)$ for any $r\in (0,1)$, we obtain that
\begin{equation}\label{weomogenea}
    \bar{w}(r\theta)=r^\gamma\psi(\theta)
\end{equation}
being $\psi:= \bar{w}|_{\mathbb{S}^+}$. By \eqref{intbarw1} it follows
that
\begin{equation}\label{normalizzata}
    \int_{\mathbb{S}^+}\theta_{n+1}^{1-2s}|\psi|^2\,dS=1,
\end{equation}
hence $\psi\not\equiv0$ on $\mathbb{S}^+$. We observe that
$\psi\in H_{\rm even}^1(\mathbb{S}^+,\theta_{n+1}^{1-2s}dS)$, since
$w^{\lambda_{n_k}}$ is even with respect to the variable $x_n$ by
\eqref{eq:w-pari} and \eqref{wlambda}. Thus, substituting
\eqref{weomogenea} into \eqref{problemarisoltodawbar}, we find that
$\gamma(\gamma+n-2s)$ is an eigenvalue of problem
\eqref{prob-eigenvalues} with $\psi$ as an associated eigenfunction.
\end{proof}

\begin{corollary}\label{cor}
  Let $\gamma$ be as in \eqref{zeta}.  Then there exists
  $m_0\in\mathbb{N}$ (even if $n=1$) such that
    \begin{equation*}
        \gamma=m_0.
    \end{equation*}
\end{corollary}
\begin{proof}
  By Lemma \ref{lemmablowup1}, $\gamma(\gamma+n-2s)$ is an eigenvalue
  of \eqref{prob-eigenvalues} in the sense of Definition
  \ref{defdiautofunzione}.
Hence, by Proposition \ref{p:eigenvalues}, there exists
$m_0\in\mathbb{N}$ (which is even if $n=1$) such that
\begin{equation*}
  \gamma(\gamma+n-2s)=m_0(m_0+n-2s).
\end{equation*}
The conclusion follows from \eqref{eq:gamma-ge} and the fact that the
function $t\mapsto t(t+n-2s)$  is injective on $\big[-\frac{n-2s}2,+\infty\big)$.
\end{proof}
The next step is to show that
$\lim_{r\to 0^+}r^{-2\gamma} H(r)=\lim_{r\to 0^+}r^{-2m_0}$ is
strictly positive (it is already known to be finite thanks to Lemma
\ref{limiteHfinito}).  To this end, we derive an asymptotic expansion
of the Fourier coefficients of $w\in H^1(B^+_{r_0}, t^{1-2s}dz)$.

Let $\{\mu_m\}_{m\in\mathbb{N}}$ be the eigenvalues of problem
  \eqref{prob-eigenvalues}, explicity characterized in
  \eqref{eigenvalues} by Proposition~\ref{p:eigenvalues}.  For every
$m\in\mathbb{N}$, let $M_m$ be the multiplicity of the eigenvalue
$\mu_m$. By the Spectral Theorem, there exists an orthonormal
  basis $\{Y_{m,k}\}_{m\in\mathbb{N},\,k=1,2,\dots,M_m}$ of
  \begin{equation*}
  L_{\rm even}^2(\mathbb{S}^+,\theta_{n+1}^{1-2s}dS):=\\
  \{\Psi \in
  L^2(\mathbb{S}^+,\theta_{n+1}^{1-2s}dS):\Psi(\theta',\theta_n,\theta_{n+1})
  =\Psi(\theta',-\theta_n,\theta_{n+1})\},
\end{equation*}
such that, for every $m\in\N$, $\{Y_{m,k}\}_{k=1,2,\dots,M_m}$ is
a basis of the eigenspace associated with the eigenvalue $\mu_m$.

For every $m\in\N$, $k=1,2,\dots,M_m$, and $\lambda\in (0,r_0]$, we
define
\begin{equation}\label{coefdiF}
  \varphi_{m,k}(\lambda):=
  \int_{\mathbb{S}^+} \theta_{n+1}^{1-2s}w(\lambda\theta) Y_{m,k}(\theta)\,dS.
\end{equation}
For future reference, we also introduce, for every $m\in\N$,
$k=1,2,\dots,M_m$, and $\lambda\in (0,r_0)$,
\begin{align}\label{Upsilon}
  &\Upsilon_{m,k}(\lambda):= -  \int_{B^+_\lambda}
                           t^{1-2s}(A-\mathrm{Id}_{n+1})
                           \nabla w\cdot \frac{\nabla _{\mathbb{S}} Y_{m,k}
                           \big(\tfrac{z}{|z|}\big)}{|z|}
                           \,dz\\
  \notag &\ -\int_{B^+_\lambda} t^{1-2s}\big(\widetilde{\mathbf{b}}(x)
           \cdot\nabla_x w + \tilde{a}(x)w\big)Y_{m,k}
           \big(\tfrac{z}{|z|}\big)\,dz\\
  \notag &\ + \int_{\partial ^+ B^+_\lambda} t^{1-2s}(A-\mathrm{Id}_{n+1})
           \nabla w\cdot \frac{z}{|z|}
           Y_{m,k}\big(\tfrac{z}{|z|}\big)\,dS+\int_{B'_\lambda}
           \tilde{h}(x)\mathop{\rm Tr}w(x)\,
           \mathop{\rm Tr}Y_{m,k}\big(\tfrac{x}{|x|}\big)\,dx.
\end{align}
It is straightforward to verify that $\Upsilon_{m,k}\in L^1(0,r_0)$.
\begin{proposition}
  Let $m_0\in\mathbb{N}$ be given by Corollary \ref{cor} and
  $\varepsilon>0$ by \eqref{vareps}.  Let $\varphi_{m_0,k}$ be defined in
  \eqref{coefdiF} for every $k=1,2,\dots,M_{m_0}$. Then, for every $r\in(0,r_0]$,
    \begin{multline}\label{asintoticacoeff}
      \varphi_{m_0,k}(\lambda)= \lambda^{m_0}
      \left(\frac{\varphi_{m_0,k}(r)}{r^{m_0}}+
        \frac{m_0 r^{-2m_0-n+2s}}{2m_0+n-2s}\int_0^r
        \tau^{m_0-1}\Upsilon_{m_0,k}(\tau)d\tau\right) \\
      +\frac{m_0+n-2s}{2m_0+n-2s}\lambda^{m_0}
      \int_\lambda^r \tau^{-m_0-n-1+2s}\Upsilon_{m_0,k}(\tau)d\tau
      + O(\lambda^{m_0+\varepsilon})
    \end{multline}
    as $\lambda\to 0^+$.
    \end{proposition}
\begin{proof}
  Let $1\leq k\leq M_{m_0}$.  Testing \eqref{problemadiWtilde} by
  $\phi(|z|) |z|^{-n-1+2s} Y_{m_0,k}\big(\frac{z}{|z|}\big)$ for any
  $\phi\in \mathcal{D}(0,r_0)$, and recalling that $Y_{m_0,k}$
  satisfies \eqref{eq-egienvlulues} with
  $\mu=\mu_{m_0}=m_0(m_0+n-2s)$, we obtain that $\varphi_{m_0,k}$
  satisfies
  \begin{equation}\label{primaeqcoeff}
    - \varphi_{m_0,k}''-\frac{n+1-2s}{\lambda}\varphi_{m_0,k}' + 
    \frac{\mu_{m_0}}{\lambda^2}\varphi_{m_0,k}=Z_{m_0,k}, 
    \end{equation}
    in a distributional sense on $(0,r_0)$, where $Z_{m_0,k}$ is the
    distribution on $(0,r_0)$ defined as 
\begin{align*}
           \phantom{a}_{\mathcal{D}'(0,r_0)}\langle Z_{m_0,k},
        \phi\rangle_{\mathcal{D}(0,r_0)}&:= -\int_{B^+_{r_0}}
        t^{1-2s}(A-\mathrm{Id}_{n+1})\nabla w\cdot
        \nabla\left(\phi(|z|)|z|^{-n-1+2s} Y_{m_0,k}\big(
                          \tfrac{z}{|z|}\big)\right)\,dz\\
                        &\quad  -\int_{B^+_{r_0}} t^{1-2s}
                          (\widetilde{\mathbf{b}}\cdot\nabla_x w +
                          \tilde{a}w) \phi(|z|)|z|^{-n-1+2s}
                          Y_{m_0,k}\big(\tfrac{z}{|z|}\big)\,dz\\
           &\quad +
                          \int_{B_{r_0}'}\tilde{h}(x) \mathop{\rm
                          Tr}w(x)\mathop{\rm
                          Tr}Y_{m_0,k}\big(\tfrac{x}{|x|}\big)
                          \phi(|x|) |x|^{-n-1+2s}\,dx
     \end{align*}
  for every $\phi\in \mathcal{D}(0,r_0)$, i.e.
  \begin{align}\label{eq:carZ}
   &\phantom{a}_{\mathcal{D}'(0,r_0)}\langle Z_{m_0,k},
      \phi\rangle_{\mathcal{D}(0,r_0)}\\
    \notag&\quad \quad :=
      -\int_0^{r_0}(r^{-n-1+2s}\phi)'(r)\left(
      \int_{\partial^+B^+_{r}}
      t^{1-2s}(A-\mathrm{Id}_{n+1})\nabla w\cdot\frac{z}{|z|}Y_{m_0,k}\big(
      \tfrac{z}{|z|}\big)\,dS\right)\,dr\\
    \notag&\quad\quad  \quad 
      -\int_0^{r_0}r^{-n-1+2s}\phi(r)\left(
      \int_{\partial^+B^+_{r}}
      t^{1-2s}(A-\mathrm{Id}_{n+1})\nabla
      w\cdot
      \frac{\nabla _{\mathbb{S}}Y_{m_0,k}}{|z|}\,dS\right)\,dr\\
    \notag&\quad\quad\quad   -\int_0^{r_0}
      \left(\int_{\partial^+B^+_{r}}t^{1-2s}(\widetilde{\mathbf{b}}\cdot\nabla_x w +
      \tilde{a}w)
      Y_{m_0,k}\big(
      \tfrac{z}{|z|}\big)\,dS
      \right) r^{-n-1+2s}\phi(r)\,dr\\
    \notag&\quad\quad\quad   +\int_0^{r_0}
      \left(\int_{S_r}\tilde{h}(x) \mathop{\rm
      Tr}w(x)\mathop{\rm
      Tr}Y_{m_0,k}\big(\tfrac{x}{|x|}\big)
      \,dS'
      \right) r^{-n-1+2s}\phi(r)\,dr
  \end{align}
  for every $\phi\in
  \mathcal{D}(0,r_0)$. From \eqref{eq:carZ} it directly follows that,
in a distributional sense on $(0,r_0)$,
    \begin{equation*}\label{directcomp}
      Z_{m_0,k}(\lambda)=\lambda^{-n-1+2s}\Upsilon_{m_0,k}'(\lambda),
    \end{equation*}
    being $\Upsilon_{m_0,k}$ defined as in \eqref{Upsilon}.  This
    allows us to rewrite \eqref{primaeqcoeff} as
    \begin{equation*}
      -(\lambda^{2m_0+n+1-2s}(\lambda^{-m_0}\varphi_{m_0,k}(\lambda))')'=
      \lambda^{m_0}\Upsilon_{m_0,k}'(\lambda).
    \end{equation*}
    For any $r\in (0,r_0]$, an integration over $(\lambda,r)$ yields
    \begin{align*}
       (\lambda^{-m_0}\varphi_{m_0,k}(\lambda))'
       & =-\lambda^{-m_0-n-1+2s}\Upsilon_{m_0,k}(\lambda) \\
        &\quad-m_0\lambda^{-2m_0-n-1+2s}\bigg(c_{m_0,k}(r)+
          \int_\lambda^r\tau^{m_0-1}\Upsilon_{m_0,k}(\tau)\,d\tau\bigg)
    \end{align*}
    for some constant $c_{m_0,k}(r)\in\R$ depending only on  $m_0$, $k$ and $r$.
Integrating again over $(\lambda,r)$, we obtain 
\begin{align}\label{lausopercontraddire}
  \varphi_{m_0,k}(\lambda) =\,
  &\lambda^{m_0}\frac{\varphi_{m_0,k}(r)}{r^{m_0}}
    - \lambda^{m_0} c_{m_0,k}(r)\frac{m_0 r^{-2m_0-n+2s}}{2m_0+n-2s}\\
  \notag  &+ m_0
            \frac{\lambda^{-m_0-n+2s}}{2m_0+n-2s}\left(c_{m_0,k}(r)
            +\int_\lambda^r \tau^{m_0-1}\Upsilon_{m_0,k}(\tau)\,d\tau\right)\\
  \notag&+ \frac{\lambda^{m_0}(m_0+n-2s)}{2m_0+n-2s}\int_\lambda^r
          \tau^{-m_0-n-1+2s}\Upsilon_{m_0,k}(\tau)\,d\tau.
\end{align}
We claim that 
\begin{equation}\label{show}
    \int_0^{r_0} \tau^{-m_0-n-1+2s}|\Upsilon_{m_0,k}(\tau)|\,d\tau< +\infty.
\end{equation}
To prove \eqref{show}, we first notice that 
\begin{align}\label{put1}
  &\tau^{-m_0-n-1+2s}\biggl|\int_{B^+_\tau}
                      t^{1-2s}(A-\mathrm{Id}_{n+1})\nabla w
                      \cdot \frac{\nabla _{\mathbb{S}}
                      Y_{m_0,k}
                      \big(\frac{z}{|z|}\big)}{|z|} \,dz\biggr|\\
  \notag     &\quad=O(\tau^{-m_0-n+2s})\left(\int_{B^+_\tau}
               t^{1-2s}|\nabla w|^2\,dz
               \right)^{1/2}\left(\int_{B^+_\tau} t^{1-2s}
               \frac{|\nabla _{\mathbb{S}}
               Y_{m_0,k}(\frac{z}{|z|})|^2}{|z|^2}\,dz\right)^{1/2}\\
  \notag     &\quad=O(\tau^{-m_0})\sqrt{H(\tau)}\left(\int_{B^+_1}
               t^{1-2s}|\nabla w^\tau|^2\,dz
               \right)^{1/2}\left(\int_{B^+_1} t^{1-2s} \frac{|\nabla
               _{\mathbb{S}}
               Y_{m_0,k}(\frac{z}{|z|})|^2}{|z|^2}\,dz\right)^{1/2}\\
  \notag&\quad=O(1) \quad \text{as $\tau \to 0^+$},
\end{align}
where we used \eqref{Aperturbazione}, H\"{o}lder's inequality,
\eqref{gradlimitato}, \eqref{Hminore}, and Lemma
\ref{lafamiglialimitata},
as well as the facts that $n>2s$ and
  $Y_{m_0,k} \in H^1(\mathbb{S}^+,\theta_{n+1}^{1-2s}dS)$.

Furthermore, by \eqref{nuovaregolarita}, H\"{o}lder's inequality, the
fact that $Y_{m_0,k} \in L^2(\mathbb{S}^+,\theta_{n+1}^{1-2s}dS)$,
\eqref{gradlimitato}, Lemma \ref{lafamiglialimitata}, and
\eqref{Hminore}, 
\begin{align}\label{put2}
  &\tau^{-m_0-n-1+2s}\biggl|\int_{B^+_\tau} t^{1-2s}
  \big(\widetilde{\mathbf{b}}(x)\cdot\nabla_x w +\tilde{a}(x)w\big)
    Y_{m_0,k}\big(\tfrac{z}{|z|}\big)\,dz\biggr|\\
                                         \notag
                                         &\quad=\,O(\tau^{-m_0-n-1+2s})\left(\int_{B^+_\tau}t^{1-2s}\left|Y_{m_0,k}\big(\tfrac{z}{|z|}\big)\right|^2dz\right)^{1/2}\times\\
                                         \notag &\quad\quad \quad
                                                  \times\left(
                                                  \left(\int_{B^+_\tau}t^{1-2s}
                                                  |\nabla
                                                  w|^2\,dz\right)^{1/2}
                                                  +
                                                  \left(\int_{B^+_\tau}t^{1-2s}|w|^2\,dz
                                                  \right)^{1/2}\right)\\
                                         \notag
                                         &\quad=O(\tau^{-m_0})\sqrt{H(\tau)}\left(
                                           \left(\int_{B^+_1}t^{1-2s}
                                           |\nabla
                                           w^\tau|^2\,dz\right)^{1/2}
                                           +
                                           \tau\left(\int_{B^+_1}t^{1-2s}
                                           |w^\tau|^2\,dz\right)^{1/2}\right)\\
                                         \notag &\quad=O(1)\quad
                                                  \text{as
                                                  $\tau \to 0^+$.}
\end{align}
By \eqref{Aperturbazione}, an integration by parts, H\"{o}lder's
inequality, Lemma \ref{lafamiglialimitata} and
\eqref{Hminore}, we have that, for every  $r\in (0,r_0]$,
\begin{align}\label{put3}
          \int_0^r & \tau^{-m_0-n-1+2s}\biggl|\int_{\partial ^+
                     B^+_\tau} t^{1-2s}
                     (A-\mathrm{Id}_{n+1})\nabla w\cdot \frac{z}{|z|}
                     Y_{m_0,k}\big(\tfrac{z}{|z|}\big)\,dS\biggr|\,d\tau\\
                   &\notag \leq \mathop{\rm const} \int_0^r
                     \tau^{-m_0-n+2s}\left(\int_{\partial ^+
                     B^+_\tau}t^{1-2s}|\nabla w|\big|
                     Y_{m_0,k}\big(\tfrac{z}{|z|}\big)\big|\,dS\right) d\tau\\
        &\notag =\mathop{\rm const}  \biggl( r^{-m_0-n+2s}\int_{B^+_r}
          t^{1-2s}
          |\nabla w|\big|Y_{m_0,k}\big(\tfrac{z}{|z|}\big)\big|\,dz\\
        &\notag \quad+(m_0+n-2s)\int_0^r \tau^{-m_0-n-1+2s}
          \left(\int_{B^+_\tau}
          t^{1-2s}|\nabla
          w|\big|Y_{m_0,k}\big(\tfrac{z}{|z|}\big)\big|\,dz
          \right)\,d\tau\biggr)\\
        &\notag \leq \mathop{\rm const} \left(
          r^{-m_0+1}\sqrt{H(\tau)} + \int_0^r
          \tau^{-m_0}\sqrt{H(\tau)}\,d\tau\right)\leq   \mathop{\rm const} \, r,
\end{align}
for some constants $\mathop{\rm const} >0$ that may vary from
  line to line and are independent of $r$.

Finally, after a change of variables, we apply H\"{o}lder's inequality
to $\tilde{h}(\tau\cdot)\in L^{p}(B'_1)$ and
$\mathop{\rm Tr}w^\tau,\mathop{\rm Tr}Y_{m_0,k}\in
L^{2^\ast_s}(B'_1)$; then, in view of \eqref{Hminore}, Lemma
\ref{lemmafallfelli}, \eqref{1sumulimitata}, \eqref{intlambda1}, and
Lemma \ref{lafamiglialimitata}, we obtain that
\begin{align}\label{put4}
  &     \tau^{-m_0-n-1+2s}\biggl|\int_{B'_\tau}
    \tilde{h}(x)\mathop{\rm Tr}w\,
    \mathop{\rm Tr}Y_{m_0,k}\big(\tfrac{x}{|x|}\big)\,dx\biggr|\\
  \notag &=\tau^{-m_0-1+2s} \sqrt{H(\tau)} \biggl|\int_{B'_1}
           \tilde{h}(\tau x)
           \mathop{\rm Tr}w^\tau(x)\, \mathop{\rm Tr}Y_{m_0,k}\big(\tfrac{x}{|x|}\big)\,dx\biggr|\\
  \notag&\leq\frac{\mathop{\rm const}}{\tau^{1-2s}} \Vert
          \tilde{h}(\tau\cdot)\Vert_{L^{p}(B'_1)}
          \left(2(n-2s)\int_{\mathbb{S}^+}\theta_{n+1}^{1-2s}\mu(\tau\cdot)
          |w^{\tau}|^2\,dS+\int_{B^+_1} t^{1-2s}|\nabla w^{\tau}|^2\,dz\right)^{\frac{1}{2}} \\
  \notag&= \mathop{\rm const}\tau^{-1+2s-\frac{n}{p}} \Vert
          \tilde{h}
          \Vert_{L^{p}(B'_{\tau})} \bigg(2(n-2s)+\int_{B^+_1}
          t^{1-2s}|\nabla w^{\tau}|^2\,dz
          \bigg)^{\!\frac{1}{2}}
          \leq\mathop{\rm const}\tau^{-1+\frac{2sp-n}{p}}.
\end{align}
Combining \eqref{put1}, \eqref{put2}, \eqref{put3}, and\eqref{put4} yields, for
every $r\in (0,r_0]$,
\begin{equation}\label{stimaUpsilon}
  \int_0^r \tau^{-m_0-n-1+2s}|\Upsilon_{m_0,k}(\tau)|\,d\tau \leq
  \mathop{\rm const}
  \left(r+\int_0^r \tau^{-1+\frac{2sp-n}{p}}\,d\tau \right)\leq \mathop{\rm const} r^{\varepsilon},
\end{equation}
for some constants $\mathop{\rm const} >0$ independent of $r$,
  where $\e$ is defined in \eqref{vareps}. In particular, estimate
\eqref{stimaUpsilon} implies \eqref{show}.

We next prove that, for every $r\in (0,r_0]$,
\begin{equation}\label{zero}
    c_{m_0,k}(r) + \int_0^r \tau^{m_0-1} \Upsilon_{m_0,k}(\tau)\,d\tau = 0.
  \end{equation}
 To
  this aim, we argue by contradiction: if \eqref{zero} is not true,
  then \eqref{lausopercontraddire} and \eqref{show} imply that 
 \begin{equation*} \varphi_{m_0,k}(\lambda) \sim m_0
      \frac{\lambda^{-m_0-n+2s}}{2m_0+n-2s}\left(c_{m_0,k}(r)
        +\int_0^r \tau^{m_0-1}\Upsilon_{m_0,k}(\tau)\,d\tau\right)
\end{equation*}
as $\lambda\to 0^+$. This, combined with the Parseval identity, yields  
\begin{align*}
  \int_{B^+_{r_0}}t^{1-2s}\frac{w^2}{|z|^2}\,dz
  &=\int_0^{r_0}\lambda^{n-1-2s} \left(\int_{\mathbb{S}^+}
    \theta_{n+1}^{1-2s}|w(\lambda\theta)|^2\,dS\right) d\lambda\\
  &\geq \int_0^{r_0}\lambda^{n-1-2s}|\varphi_{m_0,k}(\lambda)|^2\,d\lambda 
    =+\infty,
\end{align*}
thus giving rise to a contradiction, because the left-hand side is
finite as a result of \cite[Lemma 2.4]{FalFel14}, see also Lemma
  \ref{lemmahardytype}.  Claim \eqref{zero} is thereby proved.

To complete the proof, in view of \eqref{lausopercontraddire} and
\eqref{zero} it is enough to observe that
\begin{align*}
 \lambda^{-m_0-n+2s}\biggl|c_{m_0,k}(r)+\int_\lambda^r
                      \tau^{m_0-1}\Upsilon_{m_0,k}(\tau)\,d\tau\biggr|
  &=
                      \lambda^{-m_0-n+2s}\int_0^\lambda
    \tau^{m_0-1}\Upsilon_{m_0,k}(\tau)\,d\tau\\
&  =O(\lambda^{m_0+\varepsilon})\quad \text{as $\lambda\to 0^+$},
\end{align*}
as a consequence of \eqref{stimaUpsilon}.
\end{proof}
The following result finally allows us to identify the exact order of
vanishing of $H(\lambda)$ as $\lambda\to 0^+$.
\begin{proposition}\label{lim>0}  Let $m_0\in\mathbb{N}$ be as in Corollary \ref{cor}. Then 
\begin{equation*}
    \lim_{\lambda\to 0^+}\frac{H(\lambda)}{\lambda^{2 m_0}}>0.
\end{equation*} 
\end{proposition}
\begin{proof}
 By  \eqref{coefdiF}, the Parseval identity, \eqref{sviluppomu}, and \eqref{H}, we have that
    \begin{align*}
      \sum_{m=0}^\infty \sum_{k=1}^{M_m} |\varphi_{m,k}(\lambda)|^2&=
                                                                     \int_{\mathbb{S}^+} \theta^{1-2s}_{n+1}      |w(\lambda\theta)|^2\,dS
                                                                     =(1+O(\lambda)) \int_{\mathbb{S}^+} \theta^{1-2s}_{n+1}     \mu(\lambda\theta)   |w(\lambda\theta)|^2\,dS\\
                                                                   &=
                                                                     (1+O(\lambda)) H(\lambda)\quad\text{as }\lambda\to0^+.
    \end{align*}
  If we assume by contradiction that $\lim_{\lambda \to
    0^+}\frac{H(\lambda)}{\lambda^{2m_0}}=0$, this implies
\begin{equation*}
    \lim_{\lambda\to 0^+}\frac{\varphi_{m_0,k}(\lambda)}{\lambda^{m_0}}=0 \quad \text{for every $k=1,2,\dots,M_{m_0}$}, 
\end{equation*}
which, in view of  \eqref{asintoticacoeff} and \eqref{stimaUpsilon},
yields
\begin{align*}
  \frac{\varphi_{m_0,k}(r)}{r^{m_0}}+& \frac{m_0
                                       r^{-2m_0-n+2s}}{2m_0+n-2s}\int_0^r
                                       \tau^{m_0-1}\Upsilon_{m_0,k}(\tau)d\tau \\
  \notag& + \frac{m_0+n-2s}{2m_0+n-2s}\int_0^r \tau^{-m_0-n-1+2s}\Upsilon_{m_0,k}(\tau)d\tau=0
\end{align*}
for every $k=1,2,\dots,M_{m_0}$ and $r\in (0,r_0]$.  Plugging the above
identity into \eqref{asintoticacoeff} yields, for every
$k=1,2,\dots,M_{m_0}$,
\begin{equation*}
  \varphi_{m_0,k}(\lambda) =
  -\frac{m_0+n-2s}{2m_0+n-2s}\lambda^{m_0}\int_0^\lambda
  \tau^{-m_0-n-1+2s}\Upsilon_{m_0,k}(\tau)d\tau + O(\lambda^{m_0+\varepsilon}) 
\end{equation*}
as $\lambda\to 0^+$. From this and \eqref{stimaUpsilon} it follows
that $\varphi_{m_0,k}(\lambda)=O(\lambda^{m_0+\varepsilon}) $ as
$\lambda\to 0^+$, i.e., by
\eqref{coefdiF}, 
\begin{equation*}
     \int_{\mathbb{S}^+} \theta_{n+1}^{1-2s}w(\lambda\theta)
     Y_{m_0,k}(\theta)\,dS =O(\lambda^{m_0+\varepsilon})
     \quad \text{as $\lambda\to 0^+$}
\end{equation*} 
for every $k=1,2,\dots,M_{m_0}$. The above estimate, together with \eqref{wlambda} and
\eqref{Hmag} for $\delta=\varepsilon$, implies that 
\begin{equation}\label{eq:contr}
    \int_{\mathbb{S}^+} \theta_{n+1}^{1-2s}w^\lambda(\theta)
    Y_{m_0,k}(\theta)\,dS = O(\lambda^{\varepsilon/2})
    \quad \text{as $\lambda\to 0^+$},
\end{equation}
for every $k=1,2,\dots,M_{m_0}$.

By Lemma \ref{lemmablowup1} and the continuity of the trace map
  \eqref{traceoper}, for any sequence $\lambda_\ell\to 0^+$, there
  exist a subsequence $\lambda_{\ell_k}\to 0^+$ and an
  $L^2$-normalized eigenfunction $\psi$ of problem
  \eqref{prob-eigenvalues} associated with the eigenvalue $\mu_{m_0}$
  such that $w^{\lambda_{\ell_k}} \to \psi$ in
  $L^2(\mathbb{S}^+,\theta_{n+1}^{1-2s}ds)$ as $k\to\infty$.
Thus, together with \eqref{normalizzata}, we obtain that
\begin{equation*}
    \int_{\mathbb{S}^+}
    \theta_{n+1}^{1-2s}w^{\lambda_{\ell_k}}(\theta)
    \psi(\theta)\,dS\to \int_{\mathbb{S}^+} \theta_{n+1}^{1-2s}
    |\psi(\theta)|^2\,dS=1\quad
    \text{as $k\to\infty$},
  \end{equation*}
  as $\{Y_{m_0,k}\}_{k=1,2,\dots,M_{m_0}}$ is a basis for the
  eigenspace corresponding to the eigenvalue $\mu_{m_0}$.  This
    contradicts \eqref{eq:contr}, thereby completing the proof.
\end{proof}

\subsection{Local asymptotics and strong unique continuation
  principle}

In view of Proposition \ref{lim>0}, the preliminary blow-up result
  obtained in Lemma \ref{lemmablowup1} takes the following sharper
  form.
\begin{theorem}\label{thm:preparatorio}
  Let $w\in H^1(B^+_{r_0}, t^{1-2s}dz)$ be a non-trivial weak
  solution to problem \eqref{problemadiWtilde}. Let $m_0\in\mathbb{N}$
  be as in Corollary \ref{cor}, let $M_{m_0}$ be the multiplicity of
  the eigenvalue $\mu_{m_0}$, and let
  $\{Y_{m_0,k}\}_{k=1,2,\dots,M_{m_0}}$ be an associated orthonormal
  eigenbasis. Then there exists a nonzero vector
  $(\beta_1,\dots, \beta_{M_{m_0}})\in \R^{M_{m_0}}$ such that
    \begin{equation*}
        \frac{w(\lambda z)}{\lambda^{m_0}}\to |z|^{m_0}\sum
        _{k=1}^{M_{m_0}}
        \beta_k Y_{m_0,k}\bigg(\frac{z}{|z|}\bigg)
        \quad \text{in $H^1(B^+_1, t^{1-2s}dz)$  as $\lambda\to 0^+$}.
    \end{equation*}
\end{theorem}
\begin{proof}
  By \eqref{limiteHfinito}, Lemma \ref{lemmablowup1}, and Proposition
  \ref{lim>0}, for any sequence $\lambda_\ell\to 0^+$ there exist a
  subsequence $\lambda_{\ell_j}\to 0^+$ and a nonzero vector
  $(\beta_1,\dots, \beta_{M_{m_0}})\in \R^{M_{m_0}}$ such that
    \begin{equation}\label{convergenzasottos}
        \frac{w(\lambda_{\ell_j}z)}{\lambda_{\ell_j}^{m_0}} \to
        |z|^{m_0}
        \sum_{k=1}^{M_{m_0}}\beta_k
        Y_{m_0,k}\bigg(\frac{z}{|z|}\bigg)\quad
        \text{in $H^1(B^+_1, t^{1-2s}dz)$  as $j\to\infty$}.
    \end{equation}
    To prove that the above convergence holds as $\lambda\to 0^+$, we
    will show that the real numbers $\beta_1,\dots,\beta_{M_{m_0}}$,
    and consequently the entire limit profile, depend on neither
    $\lambda_\ell$ nor $\lambda_{\ell_j}$. Indeed, from
    \eqref{asintoticacoeff} and \eqref{show} it follows that, for every
    $k=1,2,\dots, M_{m_0}$ and $r\in (0,r_0]$,
    \begin{align}\label{betak1}
      \lim_{j\to
      \infty}\frac{\varphi_{m_0,k}(\lambda_{\ell_j})}{\lambda_{\ell_j}^{m_0}}
      &=\frac{\varphi_{m_0,k}(r)}{r^{m_0}}+ \frac{m_0
        r^{-2m_0-n+2s}}{2m_0+n-2s}
        \int_0^r \tau^{m_0-1}\Upsilon_{m_0,k}(\tau)\,d\tau \\
      \notag&\quad + \frac{m_0+n-2s}{2m_0+n-2s}\int_0^r
              \tau^{-m_0-n-1+2s}
              \Upsilon_{m_0,k}(\tau)\,d\tau.
    \end{align}
    On the other hand, by the definition of $\varphi_{m_0,k}$ given in
    \eqref{coefdiF}, we have
    \begin{align}\label{betak2}
      \lim_{j\to\infty}\frac{\varphi_{m_0,k}(\lambda_{\ell_j})}{\lambda_{\ell_j}^{m_0}}
      &=\lim_{j\to \infty}\int_{\mathbb{S}^+}\theta_{n+1}^{1-2s}
        \frac{w(\lambda_{\ell_j}\theta)}{\lambda_{\ell_j}^{m_0}} Y_{m_0,k}(\theta)\,dS \\
    \notag&=
            \sum_{i=1}^{M_{m_0}}\beta_i\int_{\mathbb{S}^+}\theta_{n+1}^{1-2s}
            Y_{m_0,i}(\theta)Y_{m_0,k}(\theta)\,dS= \beta_k,
    \end{align}
    where we have used \eqref{convergenzasottos} and the continuity of
    the trace map \eqref{traceoper}. Combining \eqref{betak1} and
    \eqref{betak2}, we complete the proof.
\end{proof}
We are finally ready to prove Theorem \ref{thm:1}. 
\begin{proof}[Proof of Theorem \ref{thm:1}]
  Up to rigid motions, it is not restrictive to assume that
  $x_0=0\in\partial \Omega$ and \eqref{eq:rigid-mot},
  \eqref{eq:rigid-mot2}, and \eqref{condizionisug} are satisfied. In
  this case, the outer normal vector to $\partial\Omega$ at $x_0=0$ is
  ${\mathbf N}_0=\nu(x_0)=(0,\dots,0,1)\in\R^n$, so that, for any
  $\delta>0$, $C_\delta:=\{x\in\R^n \, : \, x_n<-\delta|x'|\}$.

  Let $v\in H^1(\mathcal{C}_\Omega,t^{1-2s}dz)$ be a nontrivial weak
  solution to problem \eqref{eqrisoltadaV}. Then the function
  $\tilde{v}(x,t):= e^{-\alpha d(x)} v(x,t)\in
  H^1(\mathcal{C}_\Omega\cap B_1,t^{1-2s}dz)$ is a nontrivial weak
  solution to \eqref{eqrisoltadaW}.

Let us fix $\delta>0$. We observe that 
\begin{equation*}
    ((C_\delta\times(0,T))\cap B_1 \subset
    \mathcal{C}_{\frac{\Omega}{\lambda}}
    \cap B_1\quad \text{for sufficiently small $\lambda>0$}. 
\end{equation*}
As a result, the function
$\tilde v_\lambda:= \tilde v(\lambda\cdot)\in
H^1((C_\delta\times(0,T))\cap B_1, t^{1-2s}dz)$ for sufficiently
small $\lambda>0$.  Let $F$ be the diffeomorphism introduced in
  Proposition \ref{p:diff-deform-boun}, $\hat{v}:= \tilde{v}\circ F$,
and let $w$ be defined as in \eqref{Wwidehat}. Then,
for every $z\in  (C_\delta\times(0,T))\cap B_1$,
\begin{equation}\label{vtildelambda}
    \tilde v_\lambda(z) = w_\lambda\left(\frac{F^{-1}(\lambda
        z)}{\lambda}\right),
\end{equation}
being $w_\lambda:= w(\lambda\cdot)$, and consequently
\begin{equation}\label{nablavtildelambda}
  \nabla \tilde v_\lambda(z)= \nabla w_\lambda
  \left(\frac{F^{-1}(\lambda z)}{\lambda}\right) J_{F^{-1}}(\lambda z). 
\end{equation}
At this point, we exploit the asymptotics in
\eqref{sviluppoF} and \eqref{sviluppoF-2} to deduce that 
\begin{equation*}
  \biggl| \frac{F^{-1}(\lambda z)}{\lambda} -z\biggr| \to 0
  \quad \text{and}\quad\Vert J_{F^{-1}}(\lambda z)  - \mathrm{Id}_{n+1}\Vert \to 0
\end{equation*}
as $\lambda\to 0^+$, uniformly with respect to $z$.  By Theorem
\ref{thm:preparatorio} and Proposition \ref{p:eigenvalues}-(ii) there
exist $m_0\in\mathbb{N}$ and a nontrivial homogeneous polynomial
$P(z)=P(x',x_n,t)$ of degree $m_0$ (even with respect to both $x_n$
and $t$) such that $\mathop{\rm div}(|t|^{1-2s}\nabla P)=0$ in
$\R^{n+1}$ and
\begin{equation}\label{together}
  \frac{w(\lambda z)}{\lambda^{m_0}}\to P(z)
  \quad \text{as $\lambda\to 0^+$ in
    $H^1(B_R^+, t^{1-2s}dz)$ for every $R>0$}.
\end{equation}
Combining \eqref{vtildelambda}, \eqref{nablavtildelambda}, and
\eqref{together}, we deduce that
\begin{equation}\label{quasi}
  \frac{\tilde v(\lambda z)}{\lambda^{m_0}}\to P(z)
  \quad \text{as $\lambda\to 0^+$ in $H^1((C_\delta\times(0,T))\cap B_1, t^{1-2s}dz)$}.
\end{equation}
To complete the proof, we observe that, since
  $0\in\partial\Omega$, $d(\lambda x)\leq |\lambda x|\leq \lambda$ for
  every $x\in B_1'$ and $\lambda>0$, so that
  $(e^{\alpha d(\lambda x)}-1)^2\leq (e^{\alpha\lambda}-1)^2$;
  moreover, $|\nabla d|\leq 1$.  Hence,
  by \eqref{quasi} 
\begin{align}\label{quasi1}
  &\int_{(C_\delta\times(0,T))\cap B_1}t^{1-2s} \biggl|\frac{\tilde
    v(\lambda z)}{\lambda^{m_0}} - \frac{ v(\lambda
    z)}{\lambda^{m_0}}\biggr|^2\,dz\\
  \notag &\quad\quad=
           \int_{(C_\delta\times(0,T))\cap B_1}t^{1-2s}(e^{\alpha
           d(\lambda x)}-1)^2
           \biggl|\frac{\tilde
           v(\lambda z)}{\lambda^{m_0}}\biggr|^2\,dz\\
  \notag&\quad\quad\leq (e^{\alpha\lambda}-1)^2
          \int_{(C_\delta\times(0,T))\cap B_1}t^{1-2s}
          \biggl|\frac{\tilde
          v(\lambda z)}{\lambda^{m_0}}\biggr|^2\,dz\to 0 
\end{align}
and
\begin{align}\label{quasi2}
  \int_{(C_\delta\times(0,T))\cap B_1}
  &t^{1-2s} \biggl|\nabla \left(\frac{\tilde
    v(\lambda
    z)}{\lambda^{m_0}}
    \right)- \nabla
    \left(\frac{ v(\lambda
    z)}{\lambda^{m_0}}\right)\biggr|^2\,dz
  \\\notag&\leq
            2 \int_{(C_\delta\times(0,T))\cap B_1}t^{1-2s}(e^{\alpha
            d(\lambda x)}-1)^2
            \biggl|\nabla\left(\frac{\tilde v(\lambda
            z)}{\lambda^{m_0}}\right)\biggr|^2\,dz
  \\
  \notag  &\quad +2\lambda^2\alpha^2
            \int_{(C_\delta\times(0,T))\cap B_1}t^{1-2s}e^{2\alpha
            d(\lambda x)}|\nabla d(\lambda x)|^2
            \biggl|\frac{\tilde
            v(\lambda z)}{\lambda^{m_0}}\biggr|^2\,dz
            \\
  \notag&\leq
          2 (e^{\alpha\lambda}-1)^2\int_{(C_\delta\times(0,T))\cap
          B_1}t^{1-2s} \biggl|\nabla\left(\frac{ \tilde v(\lambda z)}{\lambda^{m_0}}\right)\biggr|^2 \,dz\\
  \notag&\quad + 2\lambda^2\alpha^2e^{2\alpha\lambda}
          \int_{(C_\delta\times(0,T))\cap B_1}t^{1-2s}
          \biggl|\frac{\tilde v(\lambda z)}{\lambda^{m_0}}\biggr|^2\,dz\to 0 
\end{align}
as $\lambda\to 0^+$.  The convergence of
$\lambda^{-m_0}v(\lambda\cdot)$ to $P$ in
$H^1((C_\delta\times(0,T))\cap B_1, t^{1-2s}dz)$ then follows
from \eqref{quasi}, \eqref{quasi1}, and \eqref{quasi2}.

To prove the convergence in
$C^{0,\tau}_{\mathrm{loc}}((\overline{
  C_\delta}\setminus\{0\})\times[0,\infty))$ for some $\tau>0$,
we first observe that, letting
$\hat v_\lambda:=\lambda^{-m_0}v(\lambda\cdot)$, for every
$R,\delta>0$, there exists $\bar\lambda_\delta>0$ such that the family
of functions $\{\hat v_\lambda\}_{\lambda\in(0, \bar\lambda_\delta)}$
is bounded in
$H^1((C_{\delta/3}\times(0,\infty))\cap B_{3R},
t^{1-2s}dz)$. Moreover, each $\hat v_\lambda$ weakly solves
\begin{equation*}
\begin{cases}
  -\mathop{\rm div}(t^{1-2s} \nabla \hat v_\lambda)=0
  &\text{in } \mathcal{C}_{\frac1\lambda\Omega}, \\[4pt]
  -\lim_{t \to 0^+} t^{1-2s}\partial _t \hat v_\lambda =
  \lambda^{2s}h(\lambda x) \mathop{\rm Tr}\hat v_\lambda & \text{on }
  \frac1\lambda\Omega.
\end{cases}
\end{equation*} 
By \eqref{ipotesisuh}, the family
$\{\lambda^{2s}h(\lambda \cdot)\}_{\lambda\in(0, \bar\lambda_\delta)}$
is bounded in $L^p (C_{\delta/3}\cap B_{3R}')$. By \cite[Proposition
2.6, (i)]{JLX} we have that
$\{\hat v_\lambda\}_{\lambda\in(0, \bar\lambda_\delta)}$ is bounded in
$L^\infty(((C_{\delta/2}\setminus B'_{r/2})\times(0,\infty))\cap
B_{2R})$ for any $r\in(0,R)$. In this regard, we point out that
\cite[Proposition 2.6, (i)]{JLX} would yield an estimate with a
constant depending on the $L^p$ norm of
$\lambda^{2s}h(\lambda \cdot)$; however, by analyzing the proof, one
can easily see that in our case such a constant is uniformly bounded
in $\lambda$ since
$\{\lambda^{2s}h(\lambda \cdot)\}_{\lambda\in(0, \bar\lambda_\delta)}$
is bounded in $L^p (C_{\delta/3}\cap B_{3R}')$. Hence,
$\{\lambda^{2s}h(\lambda x) \mathop{\rm Tr}\hat
v_\lambda\}_{\lambda\in(0, \bar\lambda_\delta)}$ is bounded in
$L^p ((C_{\delta/2}\setminus B'_{r/2})\cap B_{2R}')$, so that from
part (iii) of \cite[Proposition 2.6]{JLX} we deduce that
$\{\hat v_\lambda\}_{\lambda\in(0, \bar\lambda_\delta)}$ is bounded in
$C^{0,\rho}(((\overline{C_{\delta}}\setminus B_r)\times[0,\infty))\cap
\overline{B_R})$ for any $r\in(0,R)$ and some $\rho\in(0,1)$. This
yields the conclusion in view of the compact embedding of $C^{0,\rho}$
into $C^{0,\tau}$ for all $0<\tau<\rho$.
\end{proof}
\begin{remark} \label{rem:genarale} Since the analysis conducted so
  far has in fact been based on equation \eqref{eqrisoltadaW} and
  works for any equation of that form provided that $a$ and
  $\mathbf{b}$ are in $L^\infty$, we can allow a drift term and a
  zero-order term in equation \eqref{eqrisoltadaV} right from the
  start, meaning that Theorem \ref{thm:1} and Corollary
  \ref{cor:unique}-(i) hold for any equation of form
    \begin{equation}\label{eq:general}
      -\mathop{\rm div}(t^{1-2s}\nabla v) + t^{1-2s}
      \mathbf{b}(x)\cdot\nabla_x v + t^{1-2s}a(x) v =0,
    \end{equation}
    with
    \begin{equation*}
      \mathbf{b}\in L^{\infty}(\Omega,\R^n)
      \quad\text{and}\quad a\in L^{\infty}(\Omega).
    \end{equation*}        
  \end{remark}
  
  \section{Application to the Neumann and Robin spectral fractional
    Laplacians}\label{cinque} In this section we provide a rigorous
  definition of the Robin spectral fractional Laplacian
  $(-\Delta_{R,\alpha})^s$ of order $s\in (0,1)$ and coefficient
  $\alpha\geq0$, appearing in \eqref{prob:nonloc}. In particular,
  Subsection \ref{subsec:neumann} is devoted to the case $\alpha=0$,
  corresponding to the Neumann spectral fractional Laplacian, while 
  Subsection \ref{subsec:robin} addresses the case $\alpha>0$. For the
  sake of simplicity, we denote $(-\Delta_{R,0})^s$ by
  $(-\Delta_N)^s$.

  \subsection{Extension for the Neumann spectral
    fractional Laplacian}\label{subsec:neumann} Let us consider the
  classical eigenvalue problem for the Neumann Laplacian in a bounded
  open domain $\Omega\subset \R^n$ of class $C^{1,1}$, i.e.
  \begin{equation}\label{prob:nonloc:1}
    \begin{cases}
        -\Delta \phi=\lambda \phi &\mathrm{in \ }\Omega,\\
        \partial_\nu \phi=0 &\mathrm{on \ }\partial\Omega.
      \end{cases}
    \end{equation}
    Classic spectral theory ensures the existence of
    a diverging sequence $\{\lambda_k\}_{k\geq0}$ of nonnegative
    eigenvalues
    of problem \eqref{prob:nonloc:1}, with
    $0=\lambda_0<\lambda_1\leq\lambda_2\leq\cdots\leq\lambda_k\to+\infty$
    counted with their multiplicities. For each $k\geq0$, we can
    choose an eigenfunction $\phi_{k}$ associated with the eigenvalue
    $\lambda_k$ such that $\{\phi_{k}\}_{k\geq0}$ forms an orthonormal basis
    for $L^2(\Omega)$.
    Note that the multiplicity
    of $\lambda_0$ is 1, with $\phi_0\equiv \frac{1}{\sqrt{|\Omega|}}$.
    For a given $s\in(0,1)$, let us define the following linear space:
    \begin{equation*}
    \mathcal H^s(\Omega)=\left\{v\in L^2(\Omega)  :  \int_\Omega
      v(x)\,dx=0\text{ and } \sum_{k=1}^{\infty}\lambda^s_k\langle v,
      \phi_{k}\rangle_{L^2(\Omega)}^2<+\infty\right\},
\end{equation*}
endowed with the scalar product
\begin{equation*}
  \langle v,w\rangle_{\mathcal H^s(\Omega)}=
  \sum_{k=1}^{\infty}\lambda^s_k\langle
  v,\phi_{k}\rangle_{L^2(\Omega)}
  \langle w,\phi_{k}\rangle_{L^2(\Omega)}.
\end{equation*}
We observe that $ \langle \cdot,\cdot\rangle_{\mathcal H^s(\Omega)}$
is indeed a scalar on $\mathcal H^s(\Omega)$ due to the zero--mean
condition imposed on $\mathcal H^s(\Omega)$, which implies
orthogonality to the eigenfunctions
$\{\phi_{k}\}_{k\geq1}$. Furthermore, $\mathcal H^s(\Omega)$ with the
scalar product $\langle \cdot,\cdot\rangle_{\mathcal H^s(\Omega)}$ is
a Hilbert space.

We consider the operator
\begin{equation*}
  \mathcal L^s: \mathcal H^s(\Omega)\to(\mathcal H^s(\Omega))^*,
\end{equation*}
mapping each $v\in \mathcal H^s(\Omega)$ to $\mathcal L^s v$, where $\mathcal L^s v$
belongs to the dual space $(\mathcal H^s(\Omega))^*$ and 
acts as
\begin{equation*}
  \prescript{\mkern-\thinmuskip}{(\mathcal H^s(\Omega))^*}{\langle}
  \mathcal L^s v,w\rangle_{\mathcal H^s(\Omega)}=\langle
  v,w\rangle_{\mathcal H^s(\Omega)}
  \quad\text{for all } w\in \mathcal H^s(\Omega).
\end{equation*}
Following \cite{CapDavDupSir11,Sic24}, see also
\cite{MonPelVer13,Vol16}, given $v\in \mathcal H^s(\Omega)$, we define
its extension on the semi-infinite cylinder
$\mathcal C_\Omega^\infty=\Omega\times (0,+\infty)$ as
\begin{equation}\label{eq:def-extV}
  E_N(v): \mathcal C_\Omega^\infty\to\R,\quad
  E_N(v)(x,t):= V(x,t)=\sum_{k=1}^\infty v_kh_k(t)\phi_k(x),
\end{equation}
where $v_k=\langle v,\phi_k\rangle_{L^2(\Omega)}$ and $h_k$ is the unique solution of the ODE
\begin{equation*}
    \begin{cases}
        h''+\frac{1-2s}{t}h'-\lambda_kh=0 &\mathrm{in \ }(0,+\infty),\\
        h(0)=1,\\
        h(t)\to0\quad \text{as }t\to+\infty.
    \end{cases}
\end{equation*}
As observed in \cite{CapDavDupSir11}, for
every $k\geq1$ $h_k$ satisfies
\begin{equation*}
    -\lim_{t\to0^+}t^{1-2s}h'_k(t)=\kappa_s\lambda_k^s,
\end{equation*}
where $\kappa_s>0$ is a positive constant depending only on $s$.

The extended function $V$ defined in \eqref{eq:def-extV} belongs to
the space
\begin{equation*}
  \mathcal H^1(\mathcal C_\Omega^\infty,t^{1-2s}dz):=\left\{W\in H^1(\mathcal
    C_\Omega^\infty,t^{1-2s}dz)
    :  \int_\Omega W(x,t)dx=0 \quad \text{for every }t>0\right\},
\end{equation*}
satisfies
\begin{equation*}
  \int_{\mathcal C_\Omega^\infty}t^{1-2s}|\nabla V|^2\,dz=\kappa_s
  \left\langle v,v\right\rangle_{\mathcal H^s(\Omega)}
\end{equation*}
and
\begin{equation*}
  \mathop{\rm Tr}V=v
\end{equation*}
 where ${\mathop{\rm Tr}}$ is
 the trace map in \eqref{passoalletracce}.

 This in particular implies
that $\mathcal H^s(\Omega)\subseteq H^s(\Omega)$ with continuous
embedding. Furthermore,
for every $\Phi \in H^1(\mathcal C_\Omega^\infty,t^{1-2s}dz)$
  \begin{equation*}
    \mathop{\rm Tr}\Phi- \textstyle{\dashint}_\Omega \mathop{\rm Tr}\Phi\in
    \mathcal H^s(\Omega)
\end{equation*}
and
\begin{equation*}
  \kappa_s
  \left\langle \mathop{\rm Tr}\Phi- \textstyle{\dashint}_\Omega
    \mathop{\rm Tr}\Phi, \mathop{\rm Tr}\Phi-
    \textstyle{\dashint}_\Omega \mathop{\rm
      Tr}\Phi\right\rangle_{\mathcal H^s(\Omega)}\leq
  \int_{\mathcal C_\Omega^\infty}t^{1-2s}|\nabla \Phi|^2\,dz. 
\end{equation*}
In particular, we have
\begin{equation*}
  \mathcal H^s(\Omega)=\big\{\mathop{\rm Tr}\Phi-
  \textstyle{\dashint}_\Omega
  \mathop{\rm Tr}\Phi: \Phi \in H^1(\mathcal C_\Omega^\infty,t^{1-2s}dz)\big\}.
\end{equation*}
It is possible to show that, if $v\in \mathcal H^s(\Omega)$ and
$V=E_N(v)$ is its extension as in \eqref{eq:def-extV}, then
\begin{equation*}
  \int_{\mathcal C_\Omega^\infty}t^{1-2s}\nabla V\cdot\nabla\Phi=\kappa_s
  \left\langle v,\mathop{\rm Tr}\Phi-{\textstyle{\dashint}}_\Omega
    \mathop{\rm Tr}\Phi\right\rangle_{\mathcal H^s(\Omega)}
  \quad \text{for all $\Phi\in H^1(\mathcal C_\Omega^\infty,t^{1-2s}dz)$},
\end{equation*}
so that $V$ weakly solves
\begin{equation*}
    \begin{cases}
        -\mathop{\rm div}(t^{1-2s}\nabla V)=0 &\text{in }\mathcal C_\Omega^\infty,\\
        \nabla_x V\cdot \nu=0 &\text{on }\partial_L\mathcal
        C_\Omega^\infty:=
        \partial\Omega\times [0,+\infty),\\
        \mathop{\rm Tr}V=v &\text{on }\Omega,\\
        -\lim_{t\to0^+}t^{1-2s}\partial_tV=\kappa_s\, \mathcal L^sv
        &\text{on }\Omega.
    \end{cases}
\end{equation*}
Let us now consider the Hilbert space
\begin{equation*}
  H^s_{N}(\Omega):=\left\{v\in L^2(\Omega)  :
  \sum_{k=1}^{\infty}\lambda^s_k\langle v,\phi_{k}\rangle_{L^2(\Omega)}^2<+\infty\right\},
\end{equation*}
equipped with the scalar product
\begin{equation*}
    \langle v,w\rangle_{H_N^s(\Omega)}=\langle
    v,w\rangle_{L^2(\Omega)}+
    \langle v,w\rangle_{\mathcal H^s(\Omega)}. 
\end{equation*}
We define the Neumann spectral fractional Laplacian as the operator
\begin{equation*}
    (-\Delta_N)^s: H^s_N(\Omega)\to(H^s_N(\Omega))^*
\end{equation*}
acting on any $v\in H^s_N(\Omega)$ as
\begin{align*}
  \prescript{\mkern-\thinmuskip}{(H^s_N(\Omega))^*}{\langle}
  (-\Delta_N)^s v,
  w\rangle_{H^s_N(\Omega)}&=\sum_{k=1}^{\infty}\lambda^s_k\langle
                            v,\phi_{k}
                            \rangle_{L^2(\Omega)}\langle w,\phi_{k}\rangle_{L^2(\Omega)}\\
                          &=
                            \left\langle v-\textstyle{\dashint}_\Omega v,w-
                            \textstyle{\dashint}_\Omega w\right\rangle_{\mathcal H^s(\Omega)}
                            \quad\text{for all } w\in H^s_N(\Omega).
\end{align*}
By a weak solution to problem \eqref{prob:nonloc} with $\alpha=0$ we
mean a function $u\in H^s_N(\Omega)$ such that
\begin{equation}\label{form:deb:1}
   \prescript{\mkern-\thinmuskip}{(H^s_N(\Omega))^*}{\langle}
  (-\Delta_N)^s u,
  \psi\rangle_{H^s_N(\Omega)}=
  \int_{\Omega}\xi(x)u(x)\psi(x)
  \,dx\quad \text{for all $\psi\in H^s_N(\Omega)$}. 
\end{equation}
Notice that, given $u\in H^s_N(\Omega)$, the function
$v=u-\dashint_\Omega u$ belongs to $\mathcal H^s(\Omega)$. Moreover, 
\begin{equation*}
  (-\Delta_N)^su=\mathcal L^s\bigg(u-\dashint_\Omega u\bigg),
\end{equation*}
with the equality understood in $(\mathcal H^s(\Omega))^*$, since all
the eigenfunctions $\phi_k$ with $k\geq 1$ are orthogonal to constant
functions.  We define an extension for $u\in H^s_N(\Omega)$ as
\begin{equation}\label{eq:exteU}
    U(x,t):=\dashint_\Omega u+E_N\bigg(u-\dashint_\Omega u\bigg)(x,t).
\end{equation}
We observe that
$U-\dashint_\Omega u \in H^1(\mathcal{C}_\Omega^\infty,t^{1-2s}dz)$, so that
$U\in H^1(\Omega\times(0,R),t^{1-2s}dz)$ for every $R>0$; moreover, it
weakly solves \begin{equation*}
    \begin{cases}
        -\mathop{\rm div}(t^{1-2s}\nabla U)=0 &\mathrm{in \ }\mathcal
        C_\Omega^\infty,\\
        \nabla_x U\cdot \nu=0 &\mathrm{on \ }\partial_L\mathcal C_\Omega^\infty,\\
        \mathop{\rm Tr}U=u &\mathrm{on \ }\Omega,\\
        -\lim_{t\to0^+}t^{1-2s}\partial_tU=\kappa_s(-\Delta_N)^su &\mathrm{on \ }\Omega.
    \end{cases}
\end{equation*}
This means that 
    \begin{align*}
      \int_{\mathcal C_\Omega^\infty}t^{1-2s}\nabla U\cdot\nabla\Psi\,dz&=\kappa_s
      \left\langle u-{\textstyle{\dashint}}_\Omega u,
          \mathop{\rm Tr}\Psi-{\textstyle{\dashint}}_\Omega
          \mathop{\rm Tr}\Psi\right\rangle_{\mathcal H^s(\Omega)}\\
      &=\kappa_s\prescript{\mkern-\thinmuskip}{(H^s_N(\Omega))^*}{\langle}
  (-\Delta_N)^s u,
  \mathop{\rm Tr}\Psi\rangle_{H^s_N(\Omega)}
      \quad \text{for all $\Psi\in  H^1(\mathcal C_\Omega^\infty,t^{1-2s}dz)$}.
    \end{align*}
In particular, if $u\in H^s_N(\Omega)$  is a weak solution to
  problem \eqref{prob:nonloc} with $\alpha=0$, then its extension $U$
  defined in \eqref{eq:exteU} satisfes
  \begin{equation}\label{eq:weak-neu}
    \int_{\mathcal C_\Omega^\infty}t^{1-2s}\nabla U\cdot\nabla\Psi\,dz
    =\kappa_s\int_{\Omega}\xi(x)u(x)\mathop{\rm Tr}\Psi(x)
  \,dx\quad \text{for all  $\Psi\in  H^1(\mathcal C_\Omega^\infty,t^{1-2s}dz)$}.
  \end{equation}

  \subsection{Extension for the Robin spectral fractional
    Laplacian}\label{subsec:robin}
    Let us consider the following eigenvalue problem for the Laplacian
    subject to Robin boundary condition with parameter $\alpha>0$:
\begin{equation*}
    \begin{cases}
      -\Delta \phi=\lambda \phi &\mathrm{in \ }\Omega,\\
      \partial_\nu \phi+\alpha \,\phi=0 &\mathrm{on \ }\partial\Omega.
    \end{cases}
\end{equation*}
Here $\Omega$ is a bounded domain in $\R^n$ of class $C^{1,1}$.
Classic spectral theory
ensures the existence of a diverging sequence of positive
eigenvalues $\{\lambda_{k,\alpha}\}_{k\geq1}$ with
\begin{equation*}
  0<\lambda_{1,\alpha}<\lambda_{2,\alpha}\leq\lambda_{3,\alpha}
  \leq\cdots\leq\lambda_{k,\alpha}\to+\infty
\end{equation*}
counted with their multiplicities.  For every $k\geq1$, we can select
an eigenfunction $\phi_{k,\alpha}$ corresponding to the eigenvalue
$\lambda_{k,\alpha}$ so that $\{\phi_{k,\alpha}\}_{k\geq1}$ forms an
orthonormal basis for $L^2(\Omega)$. Note that the eigenvalue
$\lambda_{1,\alpha}$ is simple.

For some given $s\in(0,1)$, we consider  the following Hilbert space:
\begin{equation*}
  H^s_{R,\alpha}(\Omega)=\left\{v\in L^2(\Omega)
    : \sum_{k=1}^{\infty}\lambda^s_{k,\alpha}\langle v,\phi_{k,\alpha}
    \rangle_{L^2(\Omega)}^2<\infty\right\},
  \end{equation*}
endowed with the scalar product
\begin{equation*}
  \langle v,w\rangle_{H^s_{R,\alpha}(\Omega)}=
  \sum_{k=1}^{\infty}\lambda^s_{k,\alpha}\langle v,\phi_{k,\alpha}
  \rangle_{L^2(\Omega)}\langle w,\phi_{k,\alpha}\rangle_{L^2(\Omega)}.
\end{equation*}
The Robin spectral fractional Laplacian is defined as the operator
\begin{equation*}
    (-\Delta_{R,\alpha})^s: H^s_{R,\alpha}(\Omega)\to(H^s_{R,\alpha}(\Omega))^*
\end{equation*}
acting on any $v\in H^s_{R,\alpha}(\Omega)$ as
\begin{equation*}
  \prescript{\mkern-\thinmuskip}{(H^s_{R,\alpha}(\Omega))^*}{\langle}
  (-\Delta_{R,\alpha})^s v,w\rangle_{H^s_{R,\alpha}(\Omega)}
  =\langle v,w\rangle_{H^s_{R,\alpha}(\Omega)}\quad\text{for every }
  w\in H^s_{R,\alpha}(\Omega).
\end{equation*}
Following \cite{CapDavDupSir11,Sic24}, given $u\in H^s_{R,\alpha}(\Omega)$, one
can define its extension
\begin{equation}\label{eq:def-extU}
  U(x,t)=E_{R,\alpha}(u)(x,t)=\sum_{k=1}^\infty u_{k,\alpha}
  h_{k,\alpha}(t)\phi_{k,\alpha}(x),
\end{equation}
where $u_{k,\alpha}=\langle u,\phi_{k,\alpha}\rangle_{L^2(\Omega)}$
and $h_{k,\alpha}$ is the unique solution of the ODE
\begin{equation*}
    \begin{cases}
        h''+\frac{1-2s}{t}h'-\lambda_{k,\alpha}h=0 &\mathrm{in \ }(0,+\infty),\\
        h(0)=1,\\
        h(t)\to0\quad \text{as }t\to+\infty.
    \end{cases}
\end{equation*}
By \cite{CapDavDupSir11}, for
every $k\geq1$  we have 
\begin{equation*}
    -\lim_{t\to0^+}t^{1-2s}h'_{k,\alpha}(t)=\kappa_s\lambda_{k,\alpha}^s,
\end{equation*}
where $\kappa_s>0$ is the same positive constant depending only on $s$ appearing in Subsection
  \ref{subsec:neumann}.

After observing that
  \begin{equation*}
    W\mapsto   \bigg( \int_{\mathcal C_\Omega^\infty}t^{1-2s}|\nabla W|^2\,dz+
    \alpha\int_{ \partial_L \mathcal{C}_{\Omega}^\infty}t^{1-2s}
    W^2\,dS\bigg)^{1/2}
\end{equation*}
is a norm on $H^1(\mathcal C_\Omega^\infty,t^{1-2s}dz)$ equivalent to the
standard weighted norm \eqref{eq:normaH1weighted}, direct calculations show that
\begin{equation*}
  U\in H^1(\mathcal C_\Omega^\infty,t^{1-2s}dz);
\end{equation*}
furthermore,
\begin{equation*}
  \int_{\mathcal C_\Omega^\infty}t^{1-2s}|\nabla U|^2\,dz+
    \alpha\int_{ \partial_L \mathcal{C}_{\Omega}^\infty}t^{1-2s}
    U^2\,dS=\kappa_s\langle u,u\rangle_{H^s_{R,\alpha}(\Omega)}
\end{equation*}
and, letting ${\mathop{\rm Tr}}$ be
 the trace map in \eqref{passoalletracce},
\begin{equation*}
  \mathop{\rm Tr}U=u.
\end{equation*}
In particular, this shows that 
$H^s_{R,\alpha}(\Omega)$ is continuously embedded in $H^s(\Omega)$.

One can also show that, 
for every $\Phi \in H^1(\mathcal C_\Omega^\infty,t^{1-2s}dz)$,
  \begin{equation*}
    \mathop{\rm Tr}\Phi\in H^s_{R,\alpha}(\Omega)
\end{equation*}
and
\begin{equation*}
  \kappa_s
  \left\langle \mathop{\rm Tr}\Phi, \mathop{\rm Tr}\Phi
  \right\rangle_{H^s_{R,\alpha}(\Omega)}\leq
  \int_{\mathcal C_\Omega^\infty}t^{1-2s}|\nabla \Phi|^2\,dz
  +
    \alpha\int_{ \partial_L \mathcal{C}_{\Omega}^\infty}t^{1-2s}
    \Phi^2\,dS. 
\end{equation*}
In particular,
$H^s_{R,\alpha}(\Omega)=\big\{\mathop{\rm Tr}\Phi: \Phi \in
H^1(\mathcal C_\Omega^\infty,t^{1-2s}dz)\big\}$.

Finally, we observe that, if $u\in H^s_{R,\alpha}(\Omega)$ and
$U=E_{R,\alpha}(u)$ is its extension as in \eqref{eq:def-extU}, then
$U$ is a weak solution to
\begin{equation*}
    \begin{cases}
      -\mathop{\rm div}(t^{1-2s}\nabla U)=0
      &\text{in }\mathcal C_\Omega^\infty,\\
      \nabla_x U\cdot \nu+\alpha U=0
      &\text{on }\partial_L\mathcal C_\Omega^\infty,\\
      \mathop{\rm Tr}U=u
      &\text{on }\Omega,\\
      -\lim_{t\to0^+}t^{1-2s}\partial_tU=\kappa_s\mathcal
      (-\Delta_{R,\alpha})^su &\text{on }\Omega.
    \end{cases}
\end{equation*}
This means that 
\begin{equation*}
  \int_{\mathcal C_\Omega^\infty}t^{1-2s}\nabla U\cdot\nabla\Phi+
  \alpha\int_{ \partial_L \mathcal{C}_{\Omega}^\infty}t^{1-2s}
    U\Phi\,dS
    =\kappa_s
  \left\langle u,\mathop{\rm Tr}\Phi\right\rangle_{H^s_{R,\alpha}(\Omega)},
\end{equation*}
for all $\Phi\in H^1(\mathcal C_\Omega^\infty,t^{1-2s}dz)$.

By a weak solution to problem \eqref{prob:nonloc} with $\alpha>0$ we
mean a function $u\in H^s_{R,\alpha}(\Omega)$ such that
\begin{equation}\label{form:deb:2}
  \langle u,\psi\rangle_{H^s_{R,\alpha}(\Omega)}=
  \int_{\Omega}\xi(x)u(x)\psi(x)\,dx\quad \text{for all
    $\psi\in H^s_{R,\alpha}(\Omega)$}.
\end{equation}
Hence, if $u\in H^s_{R,\alpha}(\Omega)$  is a weak solution to
  problem \eqref{prob:nonloc} with $\alpha>0$, then its extension $U$
  defined in \eqref{eq:def-extU} satisfes
  \begin{multline}\label{eq:weak-rob}
    \int_{\mathcal C_\Omega^\infty}t^{1-2s}\nabla U\cdot\nabla\Phi\,dz +
    \alpha\int_{ \partial_L \mathcal{C}_{\Omega}^\infty}t^{1-2s} U\Phi\,dS\\
    =\kappa_s\int_{\Omega}\xi(x)u(x)\mathop{\rm Tr}\Psi(x) \,dx\quad
    \text{for all $\Psi\in H^1(\mathcal C_\Omega^\infty,t^{1-2s}dz)$}.
\end{multline}

\subsection{Local Asymptotics and strong unique continuation principle
  for solutions to the spectral fractional Laplacian}

\begin{proof}[Proof of Theorem \ref{thm:2}]
  The proof strongly relies on the continuity of the trace map in
  \eqref{passoalletracce}. Indeed, let $x_0\in\partial\Omega$ and $u$
  be a nontrivial weak solution to \eqref{prob:nonloc} in the sense
  clarified in \eqref{form:deb:1} and \eqref{form:deb:2}. We then
  consider its extension $U$, as defined in \eqref{eq:def-extV} and
  \eqref{eq:def-extU}. Letting $\mathcal C_\Omega$ be as in
  \eqref{eq:cil} for any fixed $T>0$, we have that $U$ belongs to
  $H^1(\mathcal{C}_\Omega,t^{1-2s}dz)$ and weakly solves
  \begin{equation*}
    \begin{cases}
      -\mathop{\rm div}(t^{1-2s}\nabla U)=0
      &\text{in }\mathcal C_\Omega,\\
      \nabla U\cdot \nu+\alpha U=0
      &\text{on  }\partial_L\mathcal C_\Omega,\\
        \mathop{\rm Tr}U=u &\text{on }\Omega,\\
        -\lim_{t\to0^+}t^{1-2s}\partial_tU=\kappa_s\xi u &\text{on  }\Omega,
    \end{cases}
\end{equation*}
because $U$ satisfies \eqref{eq:weak-neu} if $\alpha=0$ and
\eqref{eq:weak-rob} if $\alpha>0$, respectively. In any case, $U$
  is a solution to problem \eqref{eqrisoltadaV}, in the weak
  formulation given in \eqref{eq:wf}, with $h=\kappa_s\xi$.
We observe that
  $U\not\equiv0$, because $\mathop{\rm Tr}U=u\not\equiv0$.
This allows us to apply Theorem \ref{thm:1} to $U$ and deduce that
there exist $m_0\in\N$ (even if $n=1$) and a nontrivial homogeneous
polynomial $P$ of degree $m_0$ (which is invariant under reflection
across the hyperplane $\{(x,t)\in\R^{n+1}:x\cdot {\mathbf N}_0=0\}$)
such that $\mathop{\rm div}(|t|^{1-2s}\nabla P)=0$ in $\R^{n+1}$,
and, for any $\delta>0$,
    \begin{equation*}
      \frac{U(z_0+\lambda z)}{\lambda^{m_0}}\to P(z)\quad
      \text{as $\lambda\to 0^+$},
    \end{equation*}
    in $H^1((C_\delta\times(0,T)) \cap B_1,t^{1-2s}dz)$ and in
    $C^{0,\tau}_{\mathrm{loc}}((\overline{
      C_\delta}\setminus\{0\})\times[0,\infty))$ for some
    $\tau\in(0,1)$, where $z_0=(x_0,0)$ and $C_\delta$ is defined in
    \eqref{eq:def_cono}.  From this and the continuity of the trace
    map in \eqref{passoalletracce}, we conclude that
    \eqref{convaPtilde} holds with
    \begin{equation}\label{PtildePx0}
        \tilde P(x) = P(x,0).
    \end{equation}
    To complete the proof, it remains to show that $\tilde P$ is
    nontrivial.
If, by contradiction, $\tilde P$ vanishes identically in
    $\R^n$, then the function
    \begin{equation*}
        \hat P(x,t)=      \begin{cases}
          P(x,t)&\text{if $t\geq0$},\\
          0&\text{if $t<0$},          
      \end{cases}
    \end{equation*}
    solves the equation $\mathop{\rm div}(|t|^{1-2s}\nabla \hat P)=0$
    in $\R^{n+1}$, being non-zero yet vanishing on an entire
    half-space $\R_+^{n+1}$. This contradicts well-known unique
    continuation principles for elliptic equations with Muckenhoupt
    weights, see \cite{GarLin86} and \cite{TaoZhang08}.
  The proof is now complete.
\end{proof}

\section*{Acknowledgments}
The authors are members of the GNAMPA research group of INdAM-Istituto
Nazionale di Alta Matematica. The first and the third authors are
partially supported by the INdAM--GNAMPA 2026 project ``Struttura fine e
regolarit\`a in problemi variazionali non-lineari'' (CUP E53C25002010001). The second author is partially supported by the
INdAM--GNAMPA 2026 project ``Asymptotic analysis of variational problems''
(CUP E53C25002010001).

\end{document}